\documentclass[a4paper]{article}

\usepackage[left=1.5cm,top=2.5cm,right=1.5cm,bottom=2.5cm]{geometry}

\title{Operations on the set of scalar and matrix-valued quiddity sequences}

\author{\bf{Ra\'{u}l Felipe} \\
        CIMAT \\
        Callej\'{o}n Jalisco s/n Mineral de Valenciana \\
        Guanajuato, Gto, M\'exico. \\
        raulf@cimat.mx}

\date{}
\usepackage{graphicx,array,delarray}
\usepackage{latexsym}
\usepackage{amscd}
\usepackage{amsfonts, amsmath, amssymb}
\usepackage[usenames]{color}

\usepackage[english]{babel}

\newtheorem{theorem}{\textbf{Theorem}}

\newtheorem{corollary}[theorem]{\textbf{Corollary}}

\newtheorem{example}{\textbf{Example}}
\newtheorem{definition}[theorem]{\textbf{Definition}}

\newtheorem{lemma}[theorem]{\textbf{Lemma}}

\newtheorem{proposition}[theorem]{\textbf{Proposition}}

\newtheorem{remark}[theorem]{Remark}

\newenvironment{proof}[1][Proof]{\noindent\textbf{#1.} }{\ \rule{0.5em}{0.5em}}

\begin{document}

\maketitle

\begin{abstract}
Our purpose with this paper is, in first place, to recast the space of quiddity sequences corresponding to usual frieze patterns as a different type of SET operad, and second to introduce and study $\mathfrak{M}$-quiddity sequences where $\mathfrak{M}$ is a monodromy block matrix of order two. Also, we examine some related topic as are the possibility of to define matrix-valued friezes patterns and noncommutative signed Chebyshev polynomials. \\

\qquad\qquad\qquad\qquad\qquad\qquad\qquad\qquad\qquad\qquad\qquad To the memory of my colleague Nancy L\'{o}pez-Reyes \\

\bigskip

\noindent{\it{2020 Mathematics Subject Classification (MSC2020): 39A06, 15A24}}

\bigskip


\end{abstract}

\section{Introduction}
We star with some necessary definitions. Given a block matrix
\begin{equation}\label{i1}
\mathfrak{M}=\left(
               \begin{array}{cc}
                 m_{11} & m_{12} \\
                 m_{21} & m_{22} \\
               \end{array}
             \right),
\end{equation}
being each $m_{ij}$ a complex matrix of order $l$ for $i,j=1,2$ such that $|\mathfrak{M}|\neq 0$, we say that the finite sequence of complex block matrices
\begin{equation*}
\left(
  \begin{array}{cc}
    \mathfrak{p}_{1} & \mathfrak{q}_{1} \\
    I & O \\
  \end{array}
\right),\cdots, \left(
  \begin{array}{cc}
    \mathfrak{p}_{n} & \mathfrak{q}_{n} \\
    I & O \\
  \end{array}
\right),
\end{equation*}
\textbf{decomposes} $\mathfrak{M}$, if the following equality holds
\begin{equation}\label{i2}
\left(
  \begin{array}{cc}
    \mathfrak{p}_{n} & \mathfrak{q}_{n} \\
    I & O \\
  \end{array}
\right)\cdots \left(
  \begin{array}{cc}
    \mathfrak{p}_{1} & \mathfrak{q}_{1} \\
    I & O \\
  \end{array}
\right)=\mathfrak{M},
\end{equation}
in which case, the bi-vector $(\mathfrak{p}_{1},\ldots, \mathfrak{p}_{n},\mathfrak{q}_{1},\ldots, \mathfrak{q}_{n})$ is said to be
an \textbf{$\mathfrak{M}$-quiddity sequence} of length $n$. The set of all $\mathfrak{M}$-quiddity sequences is denoted by $\mathfrak{QS}(\mathfrak{M})$. It is clear that $\mathfrak{QS}(\mathfrak{M})=\bigcup_{n=1}^{\infty}\mathfrak{QS}_{n}(\mathfrak{M})$ where
for all $n\in \mathbb{N}$, $\mathfrak{QS}_{n}(\mathfrak{M})$ denotes the set of all $\mathfrak{M}$-quiddity sequences of length $n$.
Here, $I$ is the identity matrix of order $l$. We can also say that $\mathfrak{QS}(\mathfrak{M})$ is the set of solutions of equation
(\ref{i2}) or of the decomposition problem for the matrix $\mathfrak{M}$ and that the $\mathfrak{M}$-quiddity sequences are its solutions.

The main question addressed in the present research is the following\,: to provide the space $\mathfrak{QS}(\mathfrak{M})$ with certain
products such that from two $\mathfrak{M}$-quiddity sequences we can construct a new $\mathfrak{M}$-quiddity sequence in such a way that
$\mathfrak{QS}(\mathfrak{M})$ acquires a certain structure of SET operad. Taking into account that this topic is connected with difference equations with periodic coefficients, below $\mathfrak{M}$ will be called \textbf{monodromy matrix}.

We have to distinguish two cases of special interest\,:
\begin{itemize}
  \item $l=1$ for $\mathfrak{M}=\pm\,\left(
                                  \begin{array}{cc}
                                    1 & 0 \\
                                    0 & 1 \\
                                  \end{array}
                                \right)=\pm\,I_{2}$, and $\mathfrak{q}_{1}=\cdots=\mathfrak{q}_{n}=-1$. In this context, we say that
                                the vector $(\mathfrak{p}_{1},\ldots, \mathfrak{p}_{n})$ is an $\mathfrak{M}$-quiddity sequence.
  \item $l$ arbitrary for $\mathfrak{M}=\pm\,\left(
                                  \begin{array}{cc}
                                    I & 0 \\
                                    0 & I \\
                                  \end{array}
                                \right)$, and $\mathfrak{q}_{1},\cdots,\mathfrak{q}_{n}\in M_{l}(\mathbb{R})$.
\end{itemize}

An important part of our work is dedicated to show that this question is connected with extensions to the non-commutative setting
(\textbf{matrix spaces of certain order}) of different long-studied subjects for scalar case. We list some of these themes\,:
\begin{enumerate}
  \item matrix version of the theory of frieze patterns.
  \item difference equations with periodic matrix coefficients and non-commutative Chebyshev polynomials.
\end{enumerate}

Our work suggests that many of these objects related to $\mathfrak{M}$-quiddity sequences can also be multiplied of some manner
giving rise to an object of the same type which will be the subject of future research.

For convenience and for reasons which will be clear after, in this point, we will make a brief summary of the theory of frieze patterns.
The finite frieze patterns of positive integers were introduced by Coxeter in $1971$ ,and studied later in detail by Conway and Coxeter in $1973$.
They are nothing more that objects formed by $n+1$ bi-infinite rows which turn out to be of period $n$. The first and the last
rows are composed of zeros; the second row and the row before the last are filled by $1$. The entries in the remaining rows are positive
integers and they are calculated by the following unimodular rule: for every diamond with entries $a, b, c$ and $d$ we have $ad=1+bc$.
In general, a finite frieze can be seen in the form
\small.
\begin{equation*}
  \begin{array}{ccccccccccccccccc}
      & 0 &  & 0 &  & 0 &  &  &  &  & &  & & & \\
      &  & 1 &  & 1 &  & 1 &  &  &  &  &  & & & \\
      &  &  & m_{-1,-1} &  & m_{0,0} &  & m_{1,1} &  &  & &  & & & \\
      &  & \cdots &  & \ddots &  & \ddots &  & \ddots &  & \cdots & & & &  \\
      &  &  &  &  & m_{-1,n-5} &  & m_{0,n-4} &  & m_{1,n-3} & & & & & \\
      &  &  &  &  &  & 1 &  & 1 &  & 1 &  & &  &  \\
      &  &  &  &  &  &  & 0 &  & 0 &  & 0 &  & &  \\
  \end{array}.
  \end{equation*}
\normalsize
For instance, for $n=5$, we have the following finite frieze pattern
\begin{equation}\label{ejemplo1}
  \begin{array}{ccccccccccccccccc}
    0 &  & 0 &  & 0 &  & 0 &  & 0 &  & 0 & &  & & & \\
     & 1 &  & 1 &  & 1 &  & 1 &  & 1 &  & 1 &  & & & \\
     &  & 2 &  & 2 &  & 1 &  & 3 &  & 1 & & 2 & & & \\
     & \cdots &  & 3 &  & 1 &  & 2 &  & 2 &  & 1 & & 3 &  & \cdots \\
     &  &  &  & 1 &  & 1 &  & 1 &  & 1 &  & 1 &  & 1 & \\
     &  &  &  &  & 0 &  & 0 &  & 0 &  & 0 & & 0 & & 0 \\
  \end{array}.
  \end{equation}

Conway and Coxeter showed that there is an one to one correspondence between finite frieze patterns with $n+1$ rows and triangulations of
plane convex $n$-gons. \textbf{From now on, given an $n$-gon their vertices must be labeled counterclockwise by the numbers of $1$ to $n$}.
Then, this correspondence is related with elements of the first nontrivial row of a finite frieze pattern as follows, for a triangulation
of a plane convex $n$-gon, let $t_{i}$ be the number of triangles incident with the vertex $i$, resulting the vector $(t_{1},\ldots,t_{n})$. Now, if the third row is filled with infinitely many repeated copies of this vector, and the unimodular rule is used to calculate the other rows, then we obtain a frieze pattern with $n+1$ rows. Conversely, every finite frieze pattern with $n+1$ rows arises in this way. We must indicate that the frieze patterns are a particular case of more general patterns arising in the theory of bilinear discrete Hirota equations (the reader can consult \cite{zabrodin} and the references in this paper). From the beginning Conway and Coxeter found a connection between the frize patterns and the difference equations with periodic coefficients. Recent research in this direction appears in \cite{Krich}, \cite{Ilina} (development of spectral theory for periodic difference operators using integrable systems techniques), \cite{felipe}, \cite{glick}, \cite{kedem}, \cite{Maribeffa}, \cite{Maribeffa1}, \cite{oven}, \cite{ovsienko} (in relation with the pentagram map).

The vector $(t_{1},\ldots,t_{n})$ was called the \textbf{quiddity sequence} after of Conway and Coxeter. The natural numbers $t_{i}$ for $i=\overline{1,n}$
are called the entries of the quiddity sequence. For instance, the frieze pattern (\ref{ejemplo1})
has as quiddity sequence $(2,2,1,3,1)$ and it corresponds to the following triangulation of the $5$-gon
\begin{equation}\label{figura1}
 \begin{array}{c}
   \includegraphics[height=3cm]{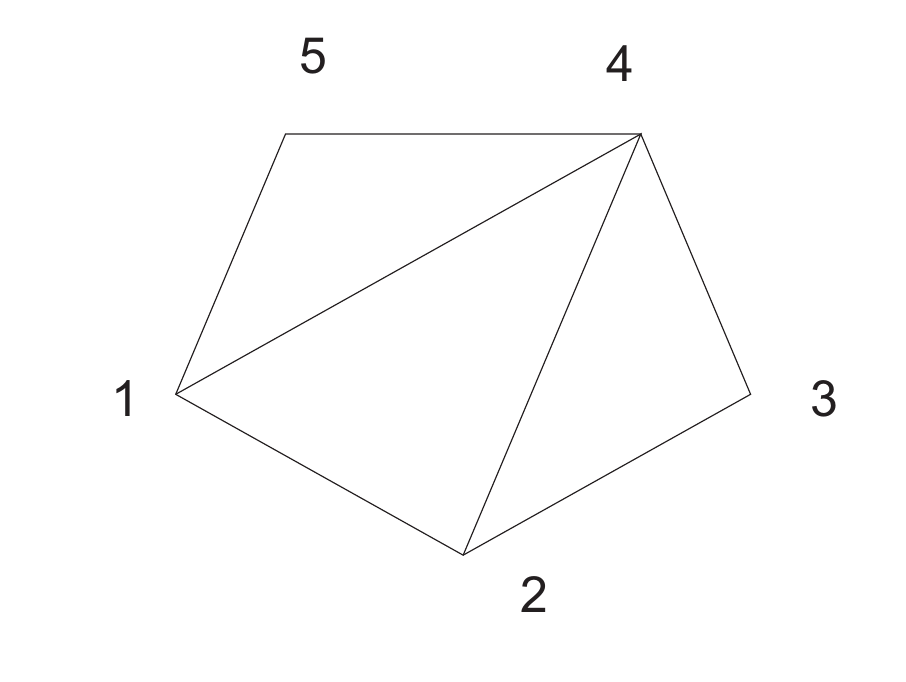}
 \end{array}.
\end{equation}

It is well known that the Conwey-Coxeter quiddity sequences are solutions of (\ref{i2}) for $\mathfrak{M}=-\,I_{2}$, see \cite{Conway}.
However, these are not all $(-I_{2})$-quiddity sequences whose components are all positive integers. Indeed, from Ovsienko's work in
\cite{ovsienko1}, it follows that in order to obtain all $(-I_{2})$-quiddity sequences of positive integers, it is necessary
to consider $3D$-dissections of a special type which result in the notion of $3D$-quiddity sequences (the reader will find more details
in this regard in the section \ref{secllamada1}).

We also wish to point that interesting results on frieze patterns can be found in the following recent papers\,: \cite{baur}, \cite{baur1}, \cite{cuntz}, \cite{cuntz1}, \cite{cuntz2},\cite{MorGen}, \cite{MorGen1} which represent only a small selection of all the articles that can
be consulted on the subject (we apologize in advance for this).

Operads first appeared in the context of algebraic topology in the 70s of the last century. Recently, in particular Set Operads have been used
in other areas more applied such as free probability and database theory in computer science, see \cite{male} and \cite{spi} for more details. Next, we make a different utility of the Operad theory.

Below, the set $\{1,2,\cdots,n-1,n\}$ is abbreviated to [n]. We now define a nonsymmetric (ns) SET \textbf{quiddity-operad} as a collection
of sets
\begin{equation}\label{operad1}
\mathcal{P}=\bigsqcup_{n\geq 1}\mathcal{P}(n),
\end{equation}
together with partial compositions maps
\begin{equation}\label{operad2}
\circ_{i}:\mathcal{P}(n)\times \mathcal{P}(m)\longrightarrow \mathcal{P}(n+m-1),\,\,\,\,\,\,\,\,\,\,\,\,\,\,n,m\geq 1,\,\,\, i\in [n]=\{1,\cdots,n\},
\end{equation}
and a distinguished element $\textbf{1}\in \mathcal{P}(1)$, the unit of $\mathcal{P}$.
This object has to satisfy the following properties
\begin{equation}\label{operad3}
(x\circ_{i} y)\circ_{i+j-1} z=x\circ_{i}(y\circ_{j}z),\,\,\,\forall x\in\mathcal{P}(n), \forall y\in \mathcal{P}(m), \forall z\in\mathcal{P}(k),
\end{equation}
for $i\in [n]$ and $j\in [m-1]$ where $2\leq m$.
\begin{equation}\label{operad4}
(x\circ_{i} y)\circ_{j+m-1} z=(x\circ_{j} z)\circ_{i} y,\,\,\,\forall x\in\mathcal{P}(n), \forall y\in \mathcal{P}(m), \forall z\in\mathcal{P}(k),
\end{equation}
and $i,j\in [n]$ such that $i< j$ for $1< n$. Finally
\begin{equation}\label{operad5}
\textbf{1}\circ_{1} x=x=x\circ_{i} \textbf{1},\,\,\,\forall x\in\mathcal{P}(n),
\end{equation}
for $i\in[n]$.

The reason for not including, as usual, $j = m$ in (\ref{operad3}) will be shown below. The interested reader may find a simple axiomatization of the customary non-symmetric operad theory in the category of sets in \cite{Chapoton} page $8$.

This paper is organized as following\,: section $2$ has as its central purpose that of introducing a structure of SET quiddity-operad on the set of all Conway-Coxeter quiddity sequences of positive integers corresponding to frieze patterns. To establish this structure, we must first define what we mean by a convex $1$-gon and a convex $2$-gon and assign their corresponding quiddity sequences to them. With this purpose, we can consider a point and a line segment as polygons without triangulations and then define their quiddity sequences as $(0)$ and $(00)$ respectively. Then, we introduce the products (\ref{producto}), (\ref{productocasoextremo}) and (\ref{2})-(\ref{6}) from which we can obtain the theorem
\ref{teorema1}. This states that for any of these products, the multiplication of two Conway-Coxeter quiddity sequences of length greater than
$3$ is again a Conway-Coxeter quiddity sequence, that is, the result is a vector of positive integers that has associated a triangulation of a
convex polygon. The main result of this section is that the set of all Conway-Coxeter quiddity sequences constitutes a nonsymmetric SET
quiddity-operad. In the section $3$, we show that the products defined in section $2$ on the set of all Conway-Coxeter quiddity sequences
also work in the space of $3D$-dissections introduced recently by Ovsienko, in other words, the multiplication of two $3D$-dissections with respect some of these products is again a $3D$-dissection. This section concludes by showing a product on the space of quiddity sequences in the case for which the monodromy matrix is the identity matrix. Concretely, the section $4$ has been developed from to introduce the notion of matrix quiddity sequence in its more general form trying to replicate the relationship that has this concept in the scalar case with difference equations. In this sense, we have developed in the matricial setting the following subjects: $a)$ the space of matrix quiddity bi-sequence, a concept introduced by us, and products on this space, $b)$ we give some elements in order to construct a theory of matrix-valued frieze patterns and $c)$ we propose a theory of noncommutative signed Chebyshev polynomials.

\section{Non symmetric set quiddity-operad structure for quiddity sequences}

\subsection{A first proposal of ns set quiddity-operad structure}

For $n\geq 3$, let us denote by $\mathcal{QS}(n)$ the set of all quiddity sequences for a plane convex $n$-gon and $\mathcal{QS}=\bigsqcup_{n\geq 1}\mathcal{QS}(n)$, where the sets $\mathcal{QS}(1)$ and $\mathcal{QS}(2)$ will be defined bellow. By simple inspection we have
\begin{equation*}
\mathcal{QS}(3)=\{(1,1,1)\},
\end{equation*}
\begin{equation*}
\mathcal{QS}(4)=\{(2,1,2,1),(1,2,1,2)\},
\end{equation*}
and
\begin{equation*}
\mathcal{QS}(5)=\{(3,1,2,2,1),(1,3,1,2,2),(2,1,3,1,2),(2,2,1,3,1),(1,2,2,1,3)\},
\end{equation*}
for $3\leq n\leq 5$.

Assume for the moment that $3\leq n,m$ and consider $\mathcal{T}=(t_{1},\ldots,t_{n})\in \mathcal{QS}(n)$, $\mathcal{S}=(s_{1},\ldots,s_{m})\in \mathcal{QS}(m)$ arbitraries, then we define for any $i\in[n]$
\begin{align}\label{producto}
\mathcal{T}\circ_{i}\mathcal{S}&=(t_{1},\ldots,t_{n})\circ_{i}(s_{1},\ldots,s_{m}) \nonumber \\
&=(t_{1},\ldots,t_{i-1}, t_{i}+s_{1}+1,s_{2},\ldots,s_{m-1}, s_{m}+1, t_{i+1}+1,t_{i+2},\ldots,t_{n}),
\end{align}
we must clarify that if $i=n$ then one takes $i+1$ as $1$, in other words
\begin{align}\label{productocasoextremo}
\mathcal{T}\circ_{n}\mathcal{S}&=(t_{1},\ldots,t_{n})\circ_{n}(s_{1},\ldots,s_{m}) \nonumber \\
&=(t_{1}+1,t_{2},\ldots,t_{n-1}, t_{n}+s_{1}+1,s_{2},\ldots,s_{m-1}, s_{m}+1).
\end{align}

If $\mathcal{S}\in\mathcal{QS}(n)$ then $\mathcal{S}_{\bigtriangleup}$ denotes the triangulation for a plane convex $n$-gon which gives rise to the quiddity sequence $\mathcal{S}$. Also, below we will use the notation $(i,i+1)(\mathcal{P})$ to indicate the side $(i,i+1)$ of polygon $\mathcal{P}$. If the polygon $\mathcal{P}$ has $n$ vertices we identify the vertex $n+1$ with the vertex $1$.

\begin{theorem}\label{teorema1}Let us suppose that $\mathcal{T}=(t_{1},\ldots,t_{n})\in \mathcal{QS}(n)$ and $\mathcal{S}=(s_{1},\ldots,s_{m})\in \mathcal{QS}(m)$, where $n,m\geq 3$. Then $\mathcal{T}\circ_{i}\mathcal{S}\in \mathcal{QS}(n+m-1)$ for all $i\in [n]$.
\end{theorem}
\begin{proof}We take advantage of the one to one correspondence mentioned above between the set of triangulations of plane convex polygons and quiddity sequences. We claim that $\mathcal{T}\circ_{i}\mathcal{S}$ is the quiddity sequence for a triangulation of a plane convex polygon $\mathcal{P}_{n+m+1}$ with $n+m-1$ vertices. In fact, we know that $\mathcal{T}$ and $\mathcal{S}$ correspond to triangulations $\mathcal{T}_{\bigtriangleup}$ and $\mathcal{S}_{\bigtriangleup}$ of two polygons $\mathcal{P}_{n}$ and $\mathcal{P}_{m}$ with $n$ and $m$ vertices respectively.
The vertices of each one of them labeled counterclockwise by the elements of $[n]$ for
$\mathcal{T}$ and by the set $[m]$ for $\mathcal{S}$. Then, $\mathcal{T}\circ_{i}\mathcal{S}$ is the quiddity sequence of a triangulation $(\mathcal{T}\circ_{i}\mathcal{S})_{\bigtriangleup}$ for a polygon with $n+m-1$
vertices constructed by identifying or overlapping the vertex $i$ of $\mathcal{P}_{n}$ with the vertex $1$ of $\mathcal{P}_{m}$ and joining by means
of introducing a side, the vertex $i+1$ of $\mathcal{P}_{n}$ with the vertex $m$ of $\mathcal{P}_{m}$. Next, we label the vertices of the new polygon from the vertex $1$ of $\mathcal{P}_{n}$. The triangulation $(\mathcal{T}\circ_{i}\mathcal{S})_{\bigtriangleup}$ associated to $\mathcal{T}\circ_{i}\mathcal{S}$ is formed by the union of the triangulations corresponding to $\mathcal{T}$ and $\mathcal{S}$, taking into account the new labeled of the vertices, plus the sides (now turned into internal sides of the new polygon) $(i,i+1)$ of $\mathcal{P}_{n}$ and $(1,m)$ of $\mathcal{P}_{m}$ but now with the new labels. In other words,
$(\mathcal{T}\circ_{i}\mathcal{S})_{\bigtriangleup}=\mathcal{T}_{\bigtriangleup}\cup\mathcal{S}_{\bigtriangleup}\cup(i,i+1)(\mathcal{P}_{n})\cup(m,1)
(\mathcal{P}_{m})$.

To see that $\mathcal{T}\circ_{i}\mathcal{S}$ is the quiddity sequence of the above
triangulation $(\mathcal{T}\circ_{i}\mathcal{S})_{\bigtriangleup}$ for the polygon $\mathcal{P}_{n+m+1}$, observe that the number of triangles incident in the vertex $i$ with respect to $(\mathcal{T}\circ_{i}\mathcal{S})_{\bigtriangleup}$ is conformed for the triangles incident in $i$ as vertex of $\mathcal{P}_{n}$ for the triangulation $\mathcal{T}_{\bigtriangleup}$, the triangles incident in the
vertex $1$ of $\mathcal{P}_{m}$ for $\mathcal{S}_{\bigtriangleup}$, plus $1$
(the latter given by the triangle formed with the vertices $i$ and $i+1$ of $\mathcal{P}_{n}$ and the vertex $m$ of $\mathcal{P}_{m}$ in the old labeling). On other hand, it is clear that with respect to the triangulation constructed $(\mathcal{T}\circ_{i}\mathcal{S})_{\bigtriangleup}$ of polygon $\mathcal{P}_{n+m+1}$, the number of triangles incident  for the vertices $i+1$ of $\mathcal{P}_{n}$ and $m$ of $\mathcal{P}_{m}$ have increased
by one. The remaining vertices keep the same number of triangles that incident them.
\end{proof}

It follows from the proof of previous theorem that the product (\ref{producto}) can be adapted to the plane convex polygons with their triangulations when the number of vertices of both polygons is greater than or equal to $3$.

\begin{example}\label{ejemplo2}
To show this fact we give an example
\begin{equation*}
 \begin{array}{c}
   \includegraphics[height=3cm]{./susy1.pdf}
 \end{array}\circ_{2} \begin{array}{c}
   \includegraphics[height=3cm]{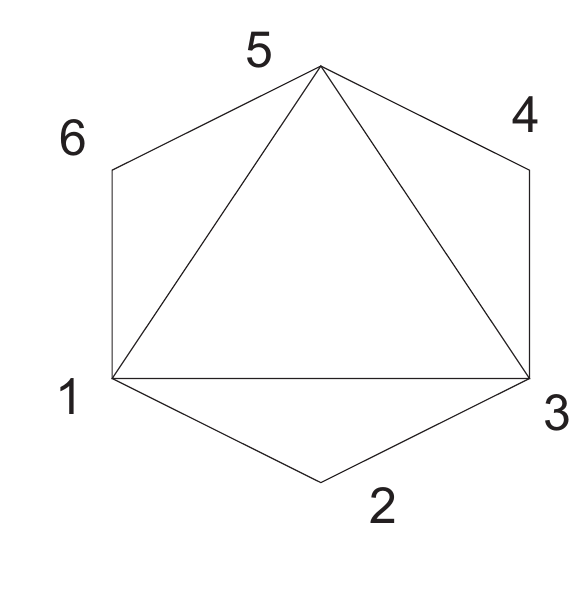}
 \end{array}=\begin{array}{c}
   \includegraphics[height=3cm]{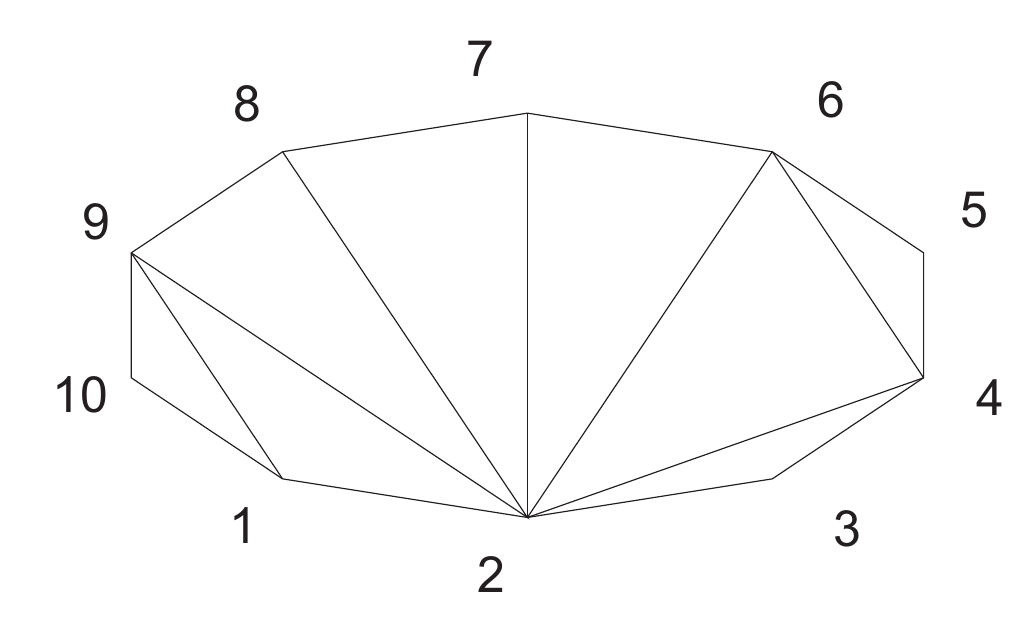}
 \end{array},
\end{equation*}
thus, as it was already defined
\begin{equation*}
(2,2,1,3,1)\circ_{2}(3,1,3,1,3,1)=(2,6,1,3,1,3,2,2,3,1).
\end{equation*}
\end{example}

Now, we set
\begin{equation}\label{1}
\mathcal{QS}(1)=\{(0)=1_{\mathcal{QS}}\},\,\,\,\, \mathcal{QS}(2)=\{(00)\}.
\end{equation}

The assignation (\ref{1}) is justified if a point and a line segment are considered as polygons without triangulations, labeled by $[1]=\{1\}$ and $[2]=\{1,2\}$ respectively, and next we move by graphical considerations: since, when it is superimposing a point with any vertex of a polygon, then the polygon is recovered, this suggests to define
\begin{equation}\label{2}
(0)\circ_{1}(t_{1},\ldots,t_{n})=(t_{1},\ldots,t_{n})\circ_{i}(0)=(t_{1},\ldots,t_{n}),
\end{equation}
for all $(t_{1},\ldots,t_{n})\in \mathcal{QS}(n)$ where $n\geq 1$ and $i\in[n]$. On other hand, we set
\begin{equation}\label{3}
(t_{1},\cdots,t_{n})\circ_{i}(00)=(t_{1},\cdots,t_{i-1},t_{i}+1,1,t_{i+1}+1,t_{i+2},\cdots,t_{n}),
\end{equation}
in particular
\begin{equation}\label{3 caso extremo}
(t_{1},\cdots,t_{n})\circ_{n}(00)=(1,t_{1}+1,t_{2},\cdots,t_{n-1},t_{n}+1),
\end{equation}
also we put
\begin{equation}\label{4}
(00)\circ_{1}(t_{1},\cdots,t_{n})=(t_{1}+1,t_{2},\cdots,t_{n-1},t_{n}+1,1),
\end{equation}
\begin{equation}\label{5}
(00)\circ_{2}(t_{1},\cdots,t_{n})=(1,t_{1}+1,t_{2},\cdots,t_{n-1},t_{n}+1),
\end{equation}
$\forall\,\,\mathcal{T}=(t_{1},\ldots,t_{n})\in \mathcal{QS}(n)$ where $n\geq 3$. Finally, we define
\begin{equation}\label{6}
(00)\circ_{1}(00)=(00)\circ_{2}(00)=(1,1,1).
\end{equation}

Next, we can illustrate (\ref{3}), (\ref{4}) and (\ref{5}) with three examples:
\begin{equation*}
 \begin{array}{c}
   \includegraphics[height=3cm]{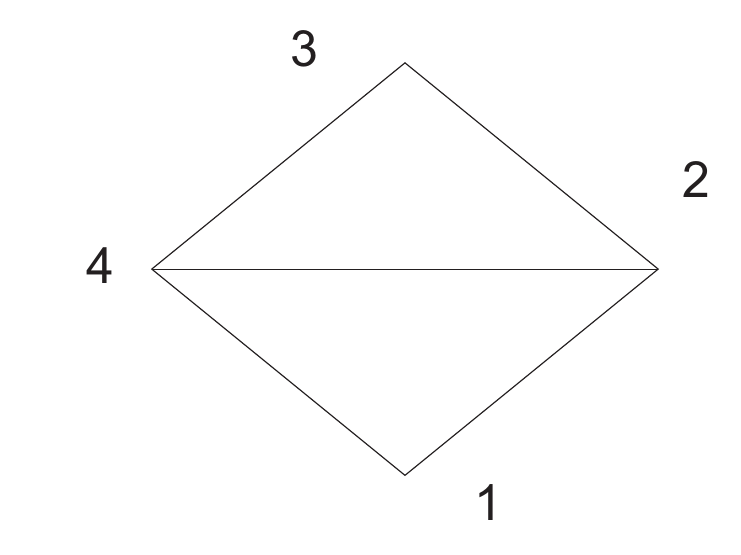}
 \end{array}\circ_{2} \begin{array}{c}
   \includegraphics[height=3cm]{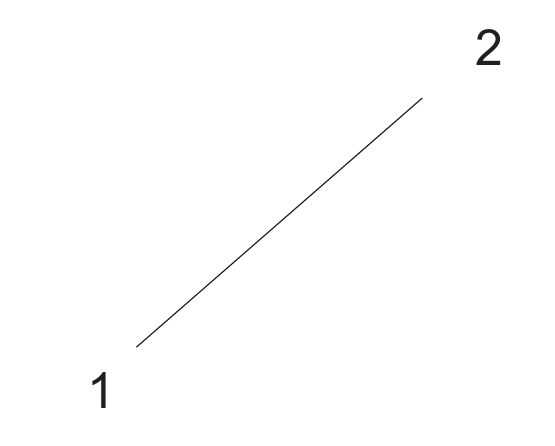}
 \end{array}=\begin{array}{c}
   \includegraphics[height=3cm]{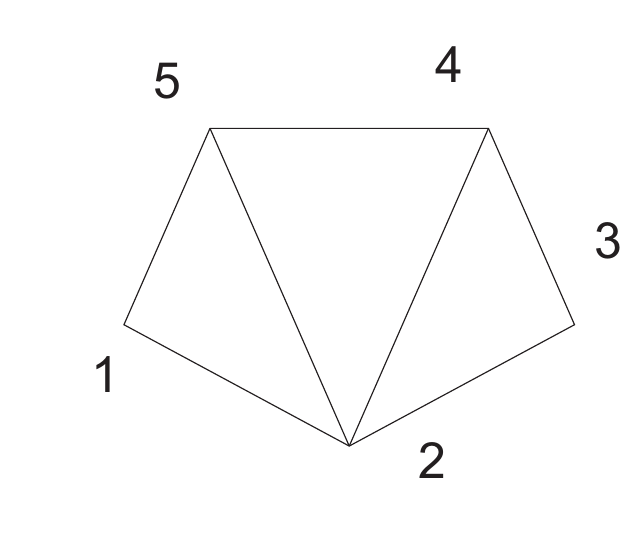}
 \end{array},
\end{equation*}
that is
\begin{equation*}
(1,2,1,2)\circ_{2}(00)=(1,3,1,2,2).
\end{equation*}

\begin{equation*}
 \begin{array}{c}
   \includegraphics[height=3cm]{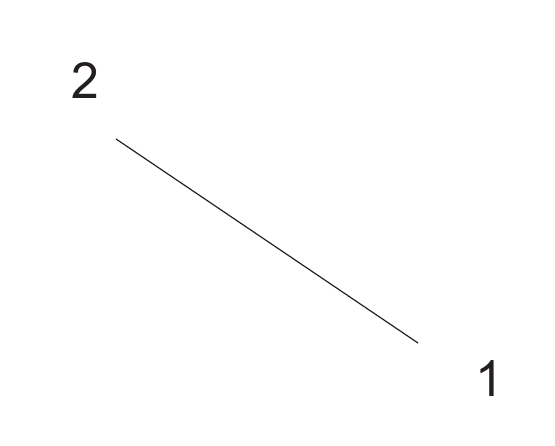}
 \end{array}\circ_{1} \begin{array}{c}
   \includegraphics[height=3cm]{./susy4.pdf}
 \end{array}=\begin{array}{c}
   \includegraphics[height=3cm]{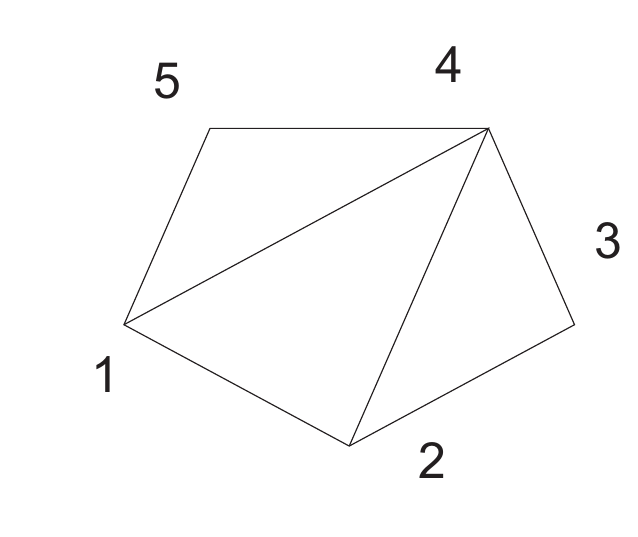}
 \end{array},
\end{equation*}
that translated into the quiddity sequences give us
\begin{equation*}
(00)\circ_{1}(1,2,1,2)=(2,2,1,3,1).
\end{equation*}

Now, $(00)\circ_{2}(1,2,1,2)=(2,2,1,3,1)$ corresponding to
\begin{equation*}
 \begin{array}{c}
   \includegraphics[height=3cm]{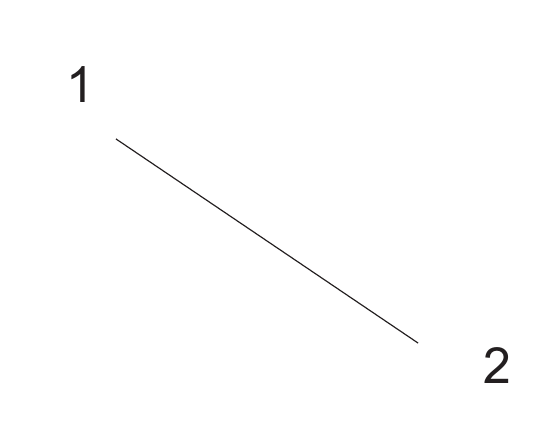}
 \end{array}\circ_{2} \begin{array}{c}
   \includegraphics[height=3cm]{./susy4.pdf}
 \end{array}=\begin{array}{c}
   \includegraphics[height=3cm]{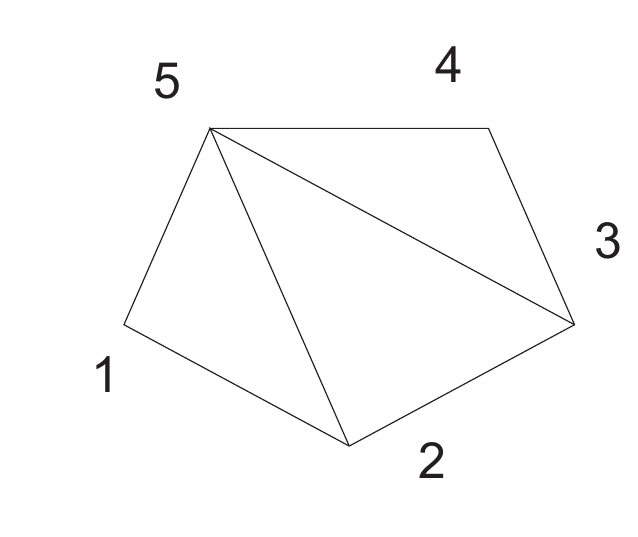}
 \end{array}.
\end{equation*}

The necessity to impose (\ref{6}) is evident.

\begin{corollary}Suppose that $1\leq n,m$. Then, for all $\mathcal{T}=(t_{1},\ldots,t_{n})\in \mathcal{QS}(n)$ and any $\mathcal{S}=(s_{1},\ldots,s_{m})\in \mathcal{QS}(m)$, we have that $\mathcal{T}\circ_{i}\mathcal{S}\in \mathcal{QS}(n+m-1)$.
\end{corollary}
\begin{proof}It follows from theorem \ref{teorema1} and (\ref{2})-(\ref{6}).
\end{proof}

\begin{theorem}\label{llamada}
The collection $\mathcal{QS}$ equipped with the products (\ref{producto}), (\ref{2})-(\ref{6}) is a ns SET quiddity-operad.
\end{theorem}
\begin{proof}(\ref{2}) implies that (\ref{operad5}) holds. Assume $3\leq n$, $k$ arbitrary and recall that $2\leq m$. Let us suppose that $\mathcal{T}=(t_{1},\ldots,t_{n})\in \mathcal{QS}(n)$, $\mathcal{S}=(s_{1},\ldots,s_{m})\in \mathcal{QS}(m)$ and $\mathcal{R}=(r_{1},\ldots,r_{k})\in \mathcal{QS}(k)$ are arbitrary.
We pass to prove (\ref{operad3}).

Suppose that $1<i< n$. For all $j\in[m-1]$ we have $i\leq i+j-1\leq i+m-2$. Hence
\begin{align}
(\mathcal{T}\circ_{i}\mathcal{S})\circ_{i+j-1}\mathcal{R}&=((t_{1},\ldots,t_{n})\circ_{i}(s_{1},\ldots,s_{m}))\circ_{i+j-1}(r_{1},\ldots,r_{k})
\nonumber \\
&=(t_{1},\ldots,t_{i-1}, t_{i}+s_{1}+1,s_{2},\ldots,s_{m-1}, s_{m}+1, t_{i+1}+1,t_{i+2},\ldots,t_{n}) \nonumber \\
&\quad\,\,\circ_{i+j-1}(r_{1},\ldots,r_{k}). \nonumber
\end{align}

We have three cases: $1)\,\,i=i+j-1$, $2)\,\,i< i+j-1< i+m-2$ and $3)\,\,i+j-1=i+m-2$.

For the first case ($i=i+j-1$,  that is, $j=1$), we have (also, $j\leq n+m-1$)
\begin{equation*}
(\mathcal{T}\circ_{i}\mathcal{S})\circ_{i}\mathcal{R}=(t_{1},\ldots,t_{i-1}, t_{i}+s_{1}+r_{1}+2,r_{2},\ldots,r_{k}+1,s_{2}+1,\ldots,s_{m-1},
s_{m}+1, t_{i+1}+1,t_{i+2},\ldots,t_{n-1},t_{n}).
\end{equation*}

On other hand,
\begin{equation*}
\mathcal{T}\circ_{i}(\mathcal{S}\circ_{1}\mathcal{R})=(t_{1},\ldots,t_{n})\circ_{i}((s_{1},\ldots,s_{m})\circ_{1}(r_{1},\ldots,r_{k}))
=(t_{1},\ldots,t_{n})\circ_{i}(s_{1}+r_{1}+1,r_{2},\cdots,r_{k}+1,s_{2}+1,\cdots,s_{m}). \nonumber
\end{equation*}

Hence, it shows that
\begin{equation}\label{7}
(\mathcal{T}\circ_{i+(1)-1}\mathcal{S})\circ_{i}\mathcal{R}=\mathcal{T}\circ_{i}(\mathcal{S}\circ_{1}\mathcal{R}),
\end{equation}
equation, in this case equivalent to (\ref{operad3}) for $1<i<n$ and $j=1$. Let us assume that $i+j-1=i+m-2$, thus $j=m-1$. Then
\begin{equation*}
(\mathcal{T}\circ_{i}\mathcal{S})\circ_{i+(m-1)-1}\mathcal{R}= (t_{1},\ldots,t_{i-1}, t_{i}+s_{1}+1,s_{2},\ldots,s_{m-1}+r_{1}+1,r_{2},\cdots,r_{k}+1,
s_{m}+2,t_{i+1}+1,t_{i+2},\ldots,t_{n}),
\end{equation*}
and
\begin{align}
\mathcal{T}\circ_{i}(\mathcal{S}\circ_{m-1}\mathcal{R})&=(t_{1},\ldots,t_{n})\circ_{i}((s_{1},\ldots,s_{m})\circ_{m-1}(r_{1},\ldots,r_{k})) \nonumber  \\
&\,\,\,\,\,\,\,(t_{1},\ldots,t_{n})\circ_{i}(s_{1},\ldots,s_{m-1}+r_{1}+1,r_{2},\cdots,r_{k-1},r_{k}+1,s_{m}+1). \nonumber
\end{align}

The last two equalities imply that
\begin{equation}\label{8}
(\mathcal{T}\circ_{i}\mathcal{S})\circ_{i+(m-1)-1}\mathcal{R}=\mathcal{T}\circ_{i}(\mathcal{S}\circ_{m-1}\mathcal{R}),
\end{equation}
then (\ref{operad3}) holds for $1<i<n$ and $j=m-1$.

Assume now that $i< i+j-1< i+m-2$, then $1< j<m-1$. We have
\begin{align}
(\mathcal{T}\circ_{i}\mathcal{S})\circ_{i+j-1}\mathcal{R}&=(t_{1},\ldots,t_{i-1}, t_{i}+s_{1}+1,s_{2},\cdots,s_{j}+r_{1}+1,r_{2},\cdots,r_{k}+1,
\nonumber \\
&\quad\,\,\,s_{j+1}+1,\cdots,s_{m-1},s_{m}+1, t_{i+1}+1,t_{i+2},\ldots,t_{n}), \nonumber
\end{align}
also
\begin{align}
\mathcal{T}\circ_{i}(\mathcal{S}\circ_{j}\mathcal{R})&=(t_{1},\ldots,t_{n})\circ_{i}((s_{1},\ldots,s_{m})\circ_{j}(r_{1},\ldots,r_{k})) \nonumber  \\
&\,\,\,\,\,\,\,(t_{1},\ldots,t_{n})\circ_{i}(s_{1},\ldots,s_{j}+r_{1}+1,r_{2},\cdots,r_{k}+1,s_{j+1}+1,\cdots,s_{m-1},s_{m}). \nonumber
\end{align}

Thus, the equation (\ref{operad3}) holds in the case $1<i<n$ and $1< j<m-1$. Hence, taking into account also (\ref{7}) and (\ref{8}) we conclude that
for $1<i<n$ and $j\in [m-1]$ the equation (\ref{operad3}) is satisfied.

Now, we suppose that $i=1$ and $j=1$ then
\begin{align}\label{81}
(\mathcal{T}\circ_{1}\mathcal{S})\circ_{1}\mathcal{R}&=((t_{1},\ldots,t_{n})\circ_{1}(s_{1},\ldots,s_{m}))\circ_{1}(r_{1},\ldots,r_{k}) \nonumber \\
&=(t_{1}+s_{1}+1,s_{2},\ldots,s_{m-1},s_{m}+1,t_{2}+1,t_{3},\ldots,t_{n})\circ_{1}(r_{1},\ldots,r_{k}) \nonumber \\
&=(t_{1}+s_{1}+r_{1}+2,r_{2},\ldots,r_{k-1},r_{k}+1,s_{2}+1,s_{3},\ldots,s_{m-1},s_{m}+1,t_{2}+1,t_{3},\ldots,t_{n}),
\end{align}
and
\begin{align}\label{82}
\mathcal{T}\circ_{1}(\mathcal{S}\circ_{1}\mathcal{R})&=(t_{1},\ldots,t_{n})\circ_{1}((s_{1},\ldots,s_{m})\circ_{1}(r_{1},\ldots,r_{k})) \nonumber \\
&=(t_{1},\ldots,t_{n})\circ_{1}(s_{1}+r_{1}+1,r_{2},\ldots,r_{k-1},r_{k}+1,s_{2}+1,s_{3},\ldots,s_{m}) \nonumber \\
&=(t_{1}+s_{1}+r_{1}+2,r_{2},\ldots,r_{k-1},r_{k}+1,s_{2}+1,s_{3},\ldots,s_{m-1},s_{m}+1,t_{2}+1,t_{3},\ldots,t_{n}),
\end{align}

Hence, (\ref{81}) and (\ref{82}) imply that $(\mathcal{T}\circ_{1}\mathcal{S})\circ_{1}\mathcal{R}=\mathcal{T}\circ_{1}(\mathcal{S}\circ_{1}\mathcal{R})$.
Consider now the case $i=1$ and $j=m-1$. We have
\begin{align}\label{83}
(\mathcal{T}\circ_{1}\mathcal{S})\circ_{m-1}\mathcal{R}&=((t_{1},\cdots,t_{n})\circ_{1}(s_{1},\cdots,s_{m}))\circ_{m-1}(r_{1},\cdots,r_{k}) \nonumber  \\
&=(t_{1}+s_{1}+1,s_{2},\cdots,s_{m-1},s_{m}+1,t_{2}+1,t_{3},\cdots,t_{n})\circ_{m-1}(r_{1},\cdots,r_{k}) \nonumber \\
&=(t_{1}+s_{1}+1,s_{2},\cdots,s_{m-2},s_{m-1}+r_{1}+1,r_{2},\cdots,r_{k-1},r_{k}+1,s_{m}+2,t_{2}+1,t_{3},\cdots,t_{n}),
\end{align}
furthermore
\begin{align}\label{84}
\mathcal{T}\circ_{1}(\mathcal{S}\circ_{m-1}\mathcal{R})&=(t_{1},\cdots,t_{n})\circ_{1}((s_{1},\cdots,s_{m})\circ_{m-1}(r_{1},\cdots,r_{k})) \nonumber  \\
&=(t_{1},\cdots,t_{n})\circ_{1}(s_{1},\cdots,s_{m-2},s_{m-1}+r_{1}+1,r_{2},\cdots,r_{k-1},r_{k}+1,s_{m}+1) \nonumber \\
&=(t_{1}+s_{1}+1,s_{2}\cdots,s_{m-2},s_{m-1}+r_{1}+1,r_{2},\cdots,r_{k-1},r_{k}+1,s_{m}+2,t_{2}+1,t_{3},\cdots,t_{n}),
\end{align}
it shows of (\ref{83}) and (\ref{84}) that $(\mathcal{T}\circ_{1}\mathcal{S})\circ_{m-1}\mathcal{R}=\mathcal{T}\circ_{1}(\mathcal{S}\circ_{m-1}\mathcal{R})$. Let us look the case for which $i=1$ and $1< j< m-1$,
we have
\begin{align}\label{85}
(\mathcal{T}\circ_{1}\mathcal{S})\circ_{j}\mathcal{R}&=(t_{1}+s_{1}+1,s_{2},\cdots,s_{j-1},s_{j}+r_{1}+1,r_{2},\cdots,r_{k-1},r_{k}+1,s_{j+1}+1,s_{j+2},\cdots,s_{m-1},s_{m}+1,t_{2}+1, \nonumber \\
&\,\,\,\,\,\,\,\,\,\,\,t_{3},\cdots,t_{n}),
\end{align}
\begin{align}\label{86}
\mathcal{T}\circ_{1}(\mathcal{S}\circ_{j}\mathcal{R})&=(t_{1},\cdots,t_{n})\circ_{1}((s_{1},\cdots,s_{m})\circ_{j}(r_{1},\cdots,r_{k})) \nonumber  \\
&=(t_{1},\cdots,t_{n})\circ_{1}(s_{1},\cdots,s_{j-1},s_{j}+r_{1}+1,r_{2},\cdots,r_{k-1},r_{k}+1,s_{j+1}+1,s_{j+2},\cdots,s_{m}) \nonumber \\
&=(t_{1}+s_{1}+1,s_{2},\cdots,s_{j-1},s_{j}+r_{1}+1,r_{2},\cdots,r_{k-1},r_{k}+1,s_{j+1}+1,s_{j+2},\cdots,s_{m-1},s_{m}+1,t_{2}+1, \nonumber \\
&\,\,\,\,\,\,\,\,\,\,\,t_{3},\cdots,t_{n}),
\end{align}
hence (\ref{85}) and (\ref{86}) prove that $(\mathcal{T}\circ_{1}\mathcal{S})\circ_{j}\mathcal{R}=\mathcal{T}\circ_{1}(\mathcal{S}\circ_{j}\mathcal{R})$. Then, the axiom (\ref{operad3}) is fulfilled  in the case $i=1$ and $1\leq j\leq m-1$.

We continue the proof for $i=n$ and $j=1$. Thus, one must prove that
\begin{equation}\label{87}
(\mathcal{T}\circ_{n}\mathcal{S})\circ_{n}\mathcal{R}=\mathcal{T}\circ_{n}(\mathcal{S}\circ_{1}\mathcal{R}),
\end{equation}
we have
\begin{align}\label{88}
(\mathcal{T}\circ_{n}\mathcal{S})\circ_{n}\mathcal{R}&=((t_{1},\cdots,t_{n})\circ_{n}(s_{1},\cdots,s_{m}))\circ_{n}(r_{1},\cdots,r_{k}) \nonumber \\
&=(t_{1}+1,t_{2},\cdots,t_{n-1},t_{n}+s_{1}+1,s_{2},\cdots,s_{m-1},s_{m}+1)\circ_{n}(r_{1},\cdots,r_{k}) \nonumber \\
&=(t_{1}+1,t_{2},\cdots,t_{n-1},t_{n}+s_{1}+r_{1}+2,r_{2},\cdots,r_{k-1},r_{k}+1,s_{2}+1,s_{3},\cdots,s_{m-1},s_{m}+1),
\end{align}
and
\begin{align}\label{89}
\mathcal{T}\circ_{n}(\mathcal{S}\circ_{1}\mathcal{R})&=(t_{1},\cdots,t_{n})\circ_{n}((s_{1},\cdots,s_{m})\circ_{1}(r_{1},\cdots,r_{k})) \nonumber \\
&=(t_{1},\cdots,t_{n})\circ_{n}(s_{1}+r_{1}+1,r_{2},\cdots,r_{k-1},r_{k}+1,s_{2}+1,s_{3}\cdots,s_{m}) \nonumber \\
&=(t_{1}+1,t_{2},\cdots,t_{n-1},t_{n}+s_{1}+r_{1}+2,r_{2},\cdots,r_{k-1},r_{k}+1,s_{2}+1,s_{3}\cdots,s_{m-1},s_{m}+1),
\end{align}
then of (\ref{88}) and (\ref{89}) we have (\ref{87}). Suppose now that $i=n$ and $j=m-1$ then
\begin{align}
(\mathcal{T}\circ_{n}\mathcal{S})\circ_{n+m-2}\mathcal{R}&=((t_{1},\cdots,t_{n})\circ_{n}(s_{1},\cdots,s_{m}))\circ_{n+m-2}(r_{1},\cdots,r_{k}) \nonumber \\
&=(t_{1}+1,t_{2},\cdots,t_{n-1},t_{n}+s_{1}+1,s_{2},\cdots,s_{m-1},s_{m}+1)\circ_{n+m-2}(r_{1},\cdots,r_{k}) \nonumber \\
&=(t_{1}+1,t_{2},\cdots,t_{n-1},t_{n}+s_{1}+1,s_{2},\cdots,s_{m-1}+r_{1}+1,r_{2},\cdots,r_{k-1},r_{k}+1,s_{m}+2), \nonumber
\end{align}
on the other hand
\begin{align}
\mathcal{T}\circ_{n}(\mathcal{S}\circ_{m-1}\mathcal{R})&=(t_{1},\cdots,t_{n})\circ_{n}((s_{1},\cdots,s_{m})\circ_{m-1}(r_{1},\cdots,r_{k})) \nonumber \\
&=(t_{1},\cdots,t_{n})\circ_{n}(s_{1},\cdots,s_{m-1}+r_{1}+1,r_{2},\cdots,r_{k-1},r_{k}+1,s_{m}+1) \nonumber \\
&=(t_{1}+1,t_{2},\cdots,t_{n-1},t_{n}+s_{1}+1,s_{2},\cdots,s_{m-1}+r_{1}+1,r_{2},\cdots,r_{k-1},r_{k}+1,s_{m}+2), \nonumber
\end{align}
the two last equalities imply that $(\mathcal{T}\circ_{n}\mathcal{S})\circ_{n+m-2}\mathcal{R}=\mathcal{T}\circ_{n}(\mathcal{S}\circ_{m-1}\mathcal{R})$.

For $i=n$ and $1<j<m-1$ we have
\begin{align}
(\mathcal{T}\circ_{n}\mathcal{S})\circ_{n+j-1}\mathcal{R}&=((t_{1},\cdots,t_{n})\circ_{n}(s_{1},\cdots,s_{m}))\circ_{n+j-1}(r_{1},\cdots,r_{k}) \nonumber \\
&=(t_{1}+1,t_{2},\cdots,t_{n-1},t_{n}+s_{1}+1,s_{2},\cdots,s_{m-1},s_{m}+1)\circ_{n+j-1}(r_{1},\cdots,r_{k}) \nonumber \\
&=(t_{1}+1,t_{2},\cdots,t_{n-1},t_{n}+s_{1}+1,s_{2},\cdots,s_{j-1},s_{j}+r_{1}+1,r_{2},\cdots,r_{k-1},r_{k}+1,s_{j+1}+1,s_{j+2}, \nonumber \\
&\,\,\,\,\,\,\,\,\,\,\,\cdots,s_{m-1},s_{m}+1), \nonumber
\end{align}
and also
\begin{align}
\mathcal{T}\circ_{n}(\mathcal{S}\circ_{j}\mathcal{R})&=(t_{1},\cdots,t_{n})\circ_{n}((s_{1},\cdots,s_{m})\circ_{j}(r_{1},\cdots,r_{k})) \nonumber \\
&=(t_{1},\cdots,t_{n})\circ_{n}(s_{1},\cdots,s_{j-1},s_{j}+r_{1}+1,r_{2},\cdots,r_{k-1},r_{k}+1,s_{j+1}+1,s_{j+2},\cdots,s_{m}) \nonumber \\
&=(t_{1}+1,t_{2},\cdots,t_{n-1},t_{n}+s_{1}+1,s_{2},\cdots,s_{j-1},s_{j}+r_{1}+1,r_{2},\cdots,r_{k-1},r_{k}+1,s_{j+1}+1,s_{j+2}, \nonumber \\
&\,\,\,\,\,\,\,\,\,\,\cdots,s_{m-1},s_{m}+1), \nonumber
\end{align}
thus $(\mathcal{T}\circ_{n}\mathcal{S})\circ_{n+j-1}\mathcal{R}=\mathcal{T}\circ_{n}(\mathcal{S}\circ_{j}\mathcal{R})$. It implies that axiom (\ref{operad3}) holds for $3\leq n$, $2\leq m$ and $1\leq k$.

Now, in any context, when $n=1$ then necessarily $x=\mathcal{T}=\textbf{1}=\textbf{1}_{\mathcal{QS}}$, hence the axiom (\ref{operad3}) reduces to the identity
$$\mathcal{S}\circ_{j} \mathcal{R}=\mathcal{S}\circ_{j} \mathcal{R},\,\,\,\, for\,\,\,\, \mathcal{S}\in\mathcal{QS}(m)\,\,\,\,\,and\,\,\,\,\, \mathcal{R}\in\mathcal{QS}(k),$$
if (\ref{operad5}) holds. Thus, in our case, taking into account that (\ref{operad5}) is satisfied, then (\ref{operad3}) holds if $n=1$, for all $j\in[m-1]$ where $2\leq m$ is arbitrary and for any $k$ such that $1\leq k$.

We turn next to the proof when $n=2$, for $2\leq m$ and $1\leq k$. Two identities shall be proved
\begin{equation}\label{90}
((00)\circ_{1}\mathcal{S})\circ_{j}\mathcal{R}=(00)\circ_{1}(\mathcal{S}\circ_{j}\mathcal{R}),
\end{equation}
and
\begin{equation}\label{91}
((00)\circ_{2}\mathcal{S})\circ_{j+1}\mathcal{R}=(00)\circ_{2}(\mathcal{S}\circ_{j}\mathcal{R})
\end{equation}
for all $\mathcal{S}\in\mathcal{QS}(m)$ and for any $\mathcal{R}\in\mathcal{QS}(k)$. One divides the proof of (\ref{90}) in $3$ cases. Suppose $j=1$ then
\begin{align}
((00)\circ_{1}\mathcal{S})\circ_{1}\mathcal{R}&=(s_{1}+1,s_{2},\cdots,s_{m-1},s_{m}+1,1)\circ_{1}\mathcal{R} \nonumber \\
&=(s_{1}+r_{1}+2,r_{2},\cdots,r_{k-1},r_{k}+1,s_{2}+1,s_{3},\cdots,s_{m-1},s_{m}+1,1), \nonumber
\end{align}
and
\begin{align}
(00)\circ_{1}(\mathcal{S}\circ_{1}\mathcal{R})&=(00)\circ_{1}(s_{1}+r_{1}+1,r_{2},\cdots,r_{k-1},r_{k}+1,s_{2}+1,s_{3},\cdots,s_{m-1},s_{m}) \nonumber \\
&=(s_{1}+r_{1}+2,r_{2},\cdots,r_{k-1},r_{k}+1,s_{2}+1,s_{3},\cdots,s_{m-1},s_{m}+1,1), \nonumber
\end{align}
thus $((00)\circ_{1}\mathcal{S})\circ_{1}\mathcal{R}=(00)\circ_{1}(\mathcal{S}\circ_{1}\mathcal{R})$. For $j=m-1$ we have
\begin{align}
((00)\circ_{1}\mathcal{S})\circ_{m-1}\mathcal{R}&=(s_{1}+1,s_{2},\cdots,s_{m-1},s_{m}+1,1)\circ_{m-1}\mathcal{R} \nonumber \\
&=(s_{1}+1,s_{2},\cdots,s_{m-1}+r_{1}+1,r_{2},\cdots,r_{k-1},r_{k}+1,s_{m}+2,1), \nonumber
\end{align}
moreover
\begin{align}
(00)\circ_{1}(\mathcal{S}\circ_{m-1}\mathcal{R})&=(00)\circ_{1}(s_{1},\cdots,s_{m-2},s_{m-1}+r_{1}+1,r_{2},\cdots,r_{k-1},r_{k}+1,s_{m}+1) \nonumber \\
&=(s_{1}+1,\cdots,s_{m-2},s_{m-1}+r_{1}+1,r_{2},\cdots,r_{k-1},r_{k}+1,s_{m}+2,1), \nonumber
\end{align}
it shows that $((00)\circ_{1}\mathcal{S})\circ_{m-1}\mathcal{R}=(00)\circ_{1}(\mathcal{S}\circ_{m-1}\mathcal{R})$. Let us consider $1<j<m-1$ then
\begin{align}
((00)\circ_{1}\mathcal{S})\circ_{j}\mathcal{R}&=(s_{1}+1,s_{2},\cdots,s_{m-1},s_{m}+1,1)\circ_{j}\mathcal{R} \nonumber \\
&=(s_{1}+1,s_{2},\cdots,s_{j-1},s_{j}+r_{1}+1,r_{2},\cdots,r_{k-1},r_{k}+1,s_{j+1}+1,s_{j+2},\cdots,s_{m-1},s_{m}+1,1), \nonumber
\end{align}
on other hand
\begin{align}
(00)\circ_{1}(\mathcal{S}\circ_{j}\mathcal{R})&=(00)\circ_{1}(s_{1},\cdots,s_{j-1},s_{j}+r_{1}+1,r_{2},\cdots,r_{k-1},r_{k}+1,s_{j+1}+1,s_{j+2},\cdots,s_{m})
\nonumber \\
&=(s_{1}+1,s_{2},\cdots,s_{j-1},s_{j}+r_{1}+1,r_{2},\cdots,r_{k-1},r_{k}+1,s_{j+1}+1,s_{j+2},\cdots,s_{m-1},s_{m}+1,1). \nonumber
\end{align}

Hence, we have proved (\ref{90}). To check (\ref{91}) we must work in a similar manner. First, suppose $j=m-1$ then
\begin{align}
((00)\circ_{2}\mathcal{S})\circ_{m}\mathcal{R}&=(1,s_{1}+1,s_{2},\cdots,s_{m-1},s_{m}+1)\circ_{m}(r_{1},\cdots,r_{k}) \nonumber \\
&=(1,s_{1}+1,s_{2},\cdots,s_{m-2},s_{m-1}+r_{1}+1,r_{2},\cdots,r_{k-1},r_{k}+1,s_{m}+2), \nonumber
\end{align}
and
\begin{align}
(00)\circ_{2}(\mathcal{S}\circ_{m-1}\mathcal{R})&=(00)\circ_{2}(s_{1},\cdots,s_{m-2},s_{m-1}+r_{1}+1,r_{2},\cdots,r_{k-1},r_{k}+1,s_{m}+1) \nonumber \\
&=(1,s_{1}+1,s_{2},\cdots,s_{m-2},s_{m-1}+r_{1}+1,r_{2},\cdots,r_{k-1},r_{k}+1,s_{m}+2), \nonumber
\end{align}
thus, $((00)\circ_{2}\mathcal{S})\circ_{m}\mathcal{R}=(00)\circ_{2}(\mathcal{S}\circ_{m-1}\mathcal{R})$. Consider $j=1$
\begin{align}
((00)\circ_{2}\mathcal{S})\circ_{2}\mathcal{R}&=(1,s_{1}+1,s_{2},\cdots,s_{m-1},s_{m}+1)\circ_{2}(r_{1},\cdots,r_{k}). \nonumber
\end{align}

The proof of (\ref{operad4}) essentially follows from the fact that $i< j$ where $i,j\in [n]$, so it will be omitted. For example,
for $1\leq i \leq n-1$ and $j=i+1$ we have
\begin{equation*}
(\mathcal{S}\circ_{i}\mathcal{T})=(s_{1},\cdots,s_{i-1},s_{i}+t_{1}+1,t_{2},\cdots,t_{m}+1,s_{i+1}+1,s_{i+2},\cdots,s_{n}),
\end{equation*}
and
\begin{equation*}
(\mathcal{S}\circ_{i+1}\mathcal{R})=(s_{1},\cdots,s_{i},s_{i+1}+r_{1}+1,r_{2},\cdots,r_{k}+1,s_{i+2}+1,s_{i+3},\cdots,s_{n}),
\end{equation*}
hence,
\begin{align*}
(\mathcal{S}\circ_{i}\mathcal{T})\circ_{i+m}\mathcal{R}&=(\mathcal{S}\circ_{i+1}\mathcal{R})\circ_{i}\mathcal{T} \\
&=(s_{1},\cdots,s_{i-1},s_{i}+t_{1}+1,t_{2},\cdots,t_{m}+1,s_{i+1}+r_{1}+2,r_{2},\cdots,r_{k-1},r_{k}+1,s_{i+2}+1,s_{i+3},\cdots,s_{n}).
\end{align*}
\end{proof}

\begin{remark}\label{observacion1}
Note that in general for $\mathcal{T}=(t_{1},\ldots,t_{n})\in \mathcal{QS}(n)$, $\mathcal{S}=(s_{1},\ldots,s_{m})\in \mathcal{QS}(m)$ and $\mathcal{R}=(r_{1},\ldots,r_{k})\in \mathcal{QS}(k)$ the equation
\begin{equation*}
(\mathcal{T}\circ_{i}\mathcal{S})\circ_{i+m-1}\mathcal{R}=\mathcal{T}\circ_{i}(\mathcal{S}\circ_{m}\mathcal{R}),
\end{equation*}
is not hold in the context of quiddity sequences and triangulations of labeled polygons.

In fact, for instance we have (in our example $i=2$ and $j=m=4$)
\begin{equation}\label{exclusion1(j=m)}
((31221)\circ_{2}(2121))\circ_{5=2+4-1}(111)=(34122321)\circ_{5}(111)=(3412412421),
\end{equation}
but from (\ref{productocasoextremo})
\begin{equation}\label{exclusion2(j=m)}
(31221)\circ_{2}((2121)\circ_{4}(111))=(31221)\circ_{2}(312312)=(3512313321).
\end{equation}

Thus, taking into account (\ref{exclusion1(j=m)}) and (\ref{exclusion2(j=m)}) one concludes that
$$((31221)\circ_{2}(2121))\circ_{5}(111)\neq (31221)\circ_{2}((2121)\circ_{4}(111)).$$
\end{remark}

We illustrate graphically the previous remark, thus in correspondence with this, the reader can check that

\begin{align}
& \left(\begin{array}{c}
   \includegraphics[height=3cm]{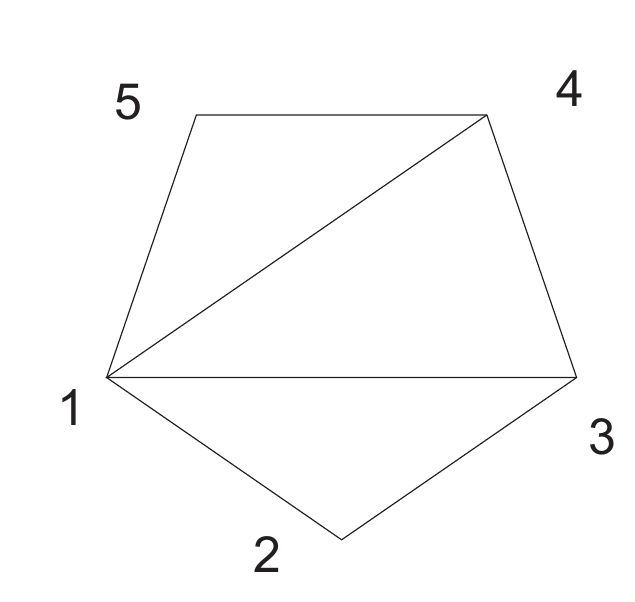}
 \end{array}\circ_{2} \begin{array}{c}
   \includegraphics[height=2cm]{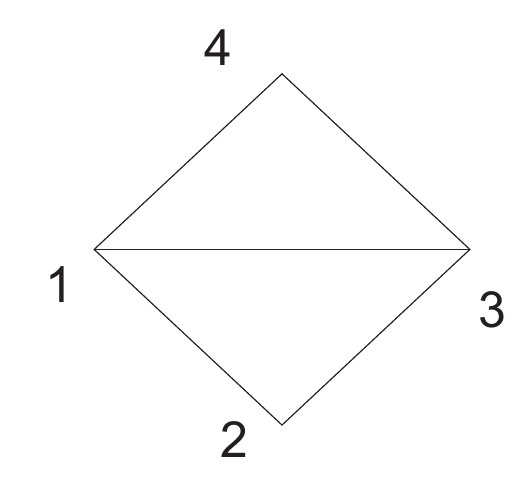}
 \end{array}\right)\circ_{5}\begin{array}{c}
   \includegraphics[height=2cm]{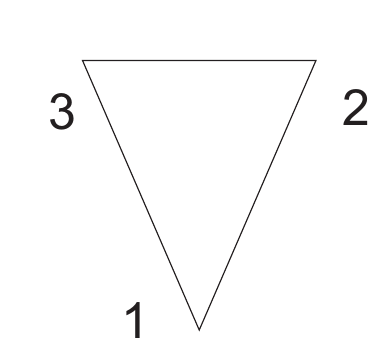}
 \end{array}= \nonumber \\
&\begin{array}{c}
   \includegraphics[height=4cm]{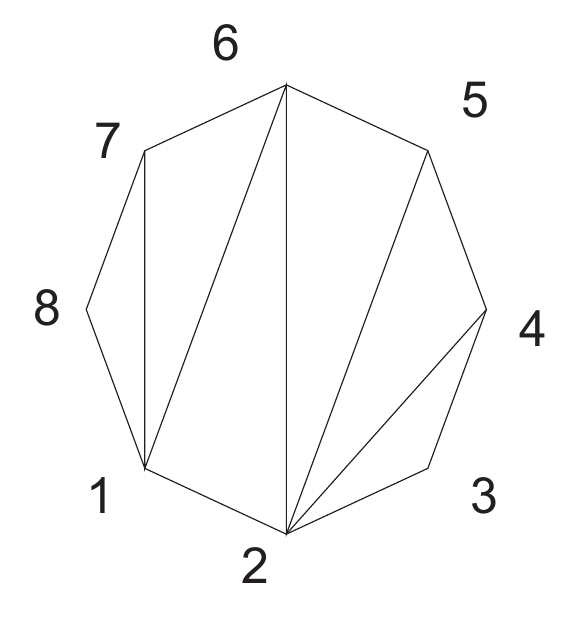}
 \end{array}\circ_{5}\begin{array}{c}
   \includegraphics[height=2cm]{./susy13.pdf}
 \end{array}=\begin{array}{c}
   \includegraphics[height=4cm]{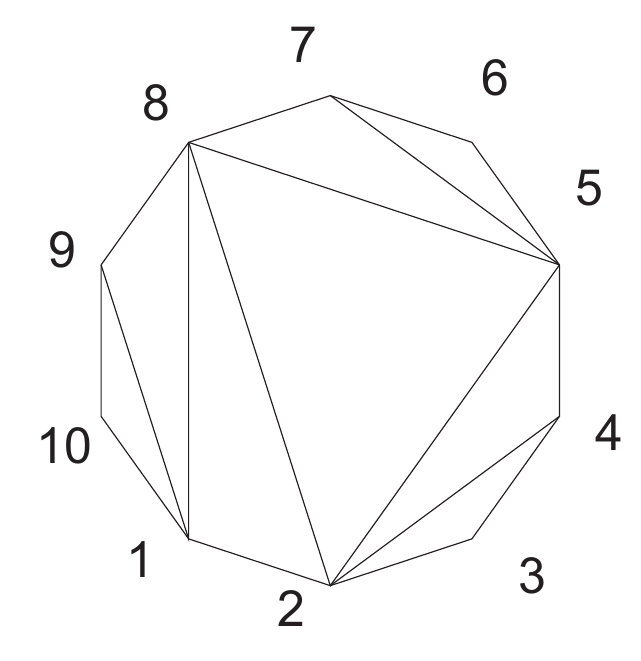}
 \end{array}. \nonumber
\end{align}

On other hand,

\begin{align}
 & \begin{array}{c}
   \includegraphics[height=3cm]{./susy11.pdf}
 \end{array}\circ_{2} \left(\begin{array}{c}
   \includegraphics[height=2cm]{./susy12.pdf}
 \end{array}\circ_{4}\begin{array}{c}
   \includegraphics[height=2cm]{./susy13.pdf}
 \end{array}\right)= \nonumber \\
 & \begin{array}{c}
   \includegraphics[height=3cm]{./susy11.pdf}
 \end{array}\circ_{2}\begin{array}{c}
   \includegraphics[height=3cm]{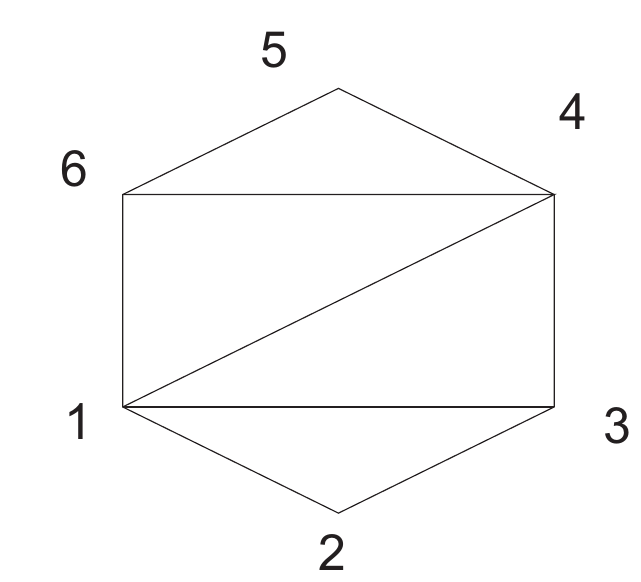}
   \end{array}=\begin{array}{c}
   \includegraphics[height=4cm]{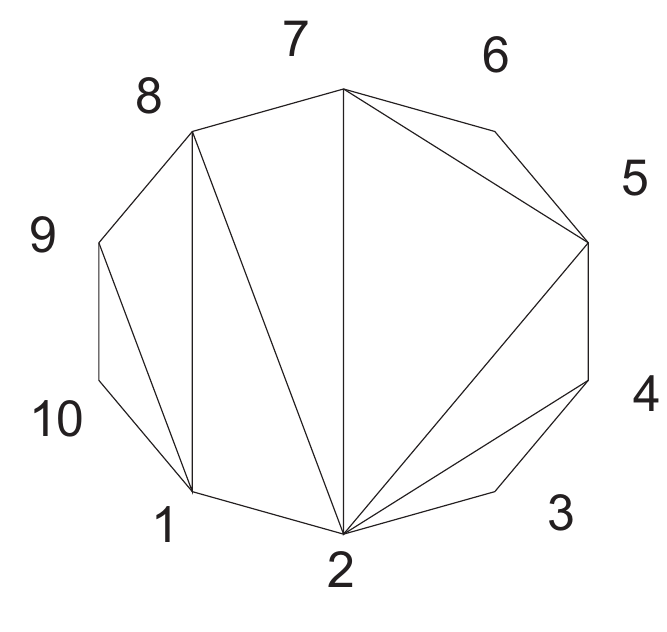}
   \end{array}. \nonumber
\end{align}

In order to reaffirm the remark \ref{observacion1}, we prove the following

\begin{proposition}Let us consider $3\leq n,m,k$. Then, for all $\mathcal{T}=(t_{1},\ldots,t_{n})\in \mathcal{QS}(n)$, all $\mathcal{S}=(s_{1},\ldots,s_{m})\in \mathcal{QS}(m)$ and any $\mathcal{R}=(r_{1},\ldots,r_{k})\in \mathcal{QS}(k)$ we have
\begin{equation}\label{9}
(\mathcal{T}\circ_{i}\mathcal{S})\circ_{i+m-1}\mathcal{R}\neq \mathcal{T}\circ_{i}(\mathcal{S}\circ_{m}\mathcal{R}),
\end{equation}
for any $i\in[n]$.
\end{proposition}
\begin{proof}First, let us assume that $1<i<n$ then
\begin{align}
(\mathcal{T}\circ_{i}\mathcal{S})\circ_{i+m-1}\mathcal{R}&=((t_{1},\cdots,t_{n})\circ_{i}(s_{1},\cdots,s_{m}))\circ_{i+m-1}(r_{1},\cdots,r_{k}) \nonumber \\
&=(t_{1},\cdots,t_{i-1},t_{i}+s_{1}+1,s_{2},\cdots,s_{m-1},\underbrace{s_{m}+1},t_{i+1}+1,t_{i+2},\cdots,t_{n})\circ_{i+m-1}(r_{1},\cdots,r_{k}), \nonumber
\end{align}
where we have indicated the $(i+m-1)-th$ component of $\mathcal{T}\circ_{i}\mathcal{S}$ which is equal to $s_{m}+1$. Hence,
\begin{align}\label{10}
(\mathcal{T}\circ_{i}\mathcal{S})\circ_{i+m-1}\mathcal{R}&=(t_{1},\cdots,t_{i-1},t_{i}+s_{1}+1,s_{2},\cdots,s_{m-1},s_{m}+r_{1}+2,r_{2}
\cdots,r_{k-1},r_{k}+1,t_{i+1}+2,t_{i+2},\cdots,t_{n}).
\end{align}

Now,
\begin{align}\label{11}
\mathcal{T}\circ_{i}(\mathcal{S}\circ_{m}\mathcal{R})&=(t_{1},\cdots,t_{n})\circ_{i}((s_{1},\cdots,s_{m})\circ_{m}(r_{1},\cdots,r_{k})) \nonumber \\
&=(t_{1},\cdots,t_{n})\circ_{i}(s_{1}+1,s_{2},\cdots,s_{m-1},s_{m}+r_{1}+1,r_{2},\cdots,r_{k-1},r_{k}+1) \nonumber \\
&=(t_{1},\cdots,t_{i-1},t_{i}+s_{1}+2,s_{2},\cdots,s_{m-1},s_{m}+r_{1}+1,r_{2},\cdots,r_{k-1},r_{k}+2,t_{i+1}+1,t_{i+2},\cdots,t_{n}).
\end{align}

From (\ref{10}) and (\ref{11}) follows (\ref{9}). We now analyze the case $i=1$
\begin{align}\label{12}
(\mathcal{T}\circ_{1}\mathcal{S})\circ_{m}\mathcal{R}&=(t_{1}+s_{1}+1,s_{2},\cdots,s_{m-1},s_{m}+1,t_{2}+1,t_{3},\cdots,t_{n-1},t_{n})\circ_{m}(r_{1},\cdots,r_{k}) \nonumber \\
&=(t_{1}+s_{1}+1,s_{2},\cdots,s_{m-1},s_{m}+r_{1}+2,r_{2},\cdots,r_{k}+1,t_{2}+2,t_{3},\cdots,t_{n-1},t_{n}),
\end{align}
and
\begin{align}\label{13}
\mathcal{T}\circ_{1}(\mathcal{S}\circ_{m}\mathcal{R})&=\mathcal{T}\circ_{1}((s_{1},\cdots,s_{m})\circ_{m}(r_{1},\cdots,r_{k}))=\mathcal{T}\circ_{1}(s_{1}+1,s_{2},\cdots,s_{m-1},s_{m}+r_{1}+1,r_{2},\cdots,r_{k-1},r_{k}+1)   \nonumber \\
&=(t_{1}+s_{1}+2,s_{2},\cdots,s_{m-1},s_{m}+r_{1}+1,r_{2},\cdots,r_{k-1},r_{k}+2,t_{2}+1,t_{3},\cdots,t_{n-1},t_{n}),
\end{align}
(\ref{12}) and (\ref{13}) show that $(\mathcal{T}\circ_{1}\mathcal{S})\circ_{m}\mathcal{R}\neq \mathcal{T}\circ_{1}(\mathcal{S}\circ_{m}\mathcal{R})$. Finally, assume that $i=n$ then
\begin{align}\label{14}
(\mathcal{T}\circ_{n}\mathcal{S})\circ_{n+m-1}\mathcal{R}&=((t_{1},\cdots,t_{n})\circ_{n}(s_{1},\cdots,s_{m}))\circ_{n+m-1}(r_{1},\cdots,r_{k}) \nonumber \\
&=(t_{1}+1,t_{2},\cdots,t_{n-1},t_{n}+s_{1}+1,s_{2},\cdots,s_{m-1},s_{m}+1)\circ_{n+m-1}(r_{1},\cdots,r_{k}) \nonumber \\
&=(t_{1}+2,t_{2},\cdots,t_{n-1},t_{n}+s_{1}+1,s_{2},\cdots,s_{m-1},s_{m}+r_{1}+2,r_{2},\cdots,r_{k-1},r_{k}+1),
\end{align}
on other hand,
\begin{align}\label{15}
\mathcal{T}\circ_{n}(\mathcal{S}\circ_{m}\mathcal{R})&=(t_{1},\cdots,t_{n})\circ_{n}((s_{1},\cdots,s_{m})\circ_{m}(r_{1},\cdots,r_{k})) \nonumber \\
&=(t_{1},\cdots,t_{n})\circ_{n}(s_{1}+1,s_{2},\cdots,s_{m-1},s_{m}+r_{1}+1,r_{2},\cdots,r_{k-1},r_{k}+1) \nonumber \\
&=(t_{1}+1,t_{2},\cdots,t_{n-1},t_{n}+s_{1}+2,s_{2},\cdots,s_{m-1},s_{m}+r_{1}+1,r_{2},\cdots,r_{k-1},r_{k}+2),
\end{align}
thus of (\ref{14}) and (\ref{15}), we have that $(\mathcal{T}\circ_{n}\mathcal{S})\circ_{n+m-1}\mathcal{R}\neq \mathcal{T}\circ_{n}(\mathcal{S}\circ_{m}\mathcal{R})$.
\end{proof}

\subsection{A comment about the previous subsection}

In this subsection, we introduce a new product for quiddity sequences very different which was defined in the previous subsection.
With this intention we will maintain the notations, denominations and conventions introduced earlier.

Let $\mathcal{T}=(t_{1},\cdots,t_{n})$ and $\mathcal{S}=(s_{1},\cdots,s_{m})$ be two quiddity sequences. Define
\begin{align}\label{101}
\mathcal{T}\bullet_{i}\mathcal{S}&=(t_{1},\cdots,t_{n})\bullet_{i}(s_{1},\cdots,s_{m}) \nonumber \\
&\,\,\,\,\,\,\,\,\,(t_{1},\cdots,t_{i-2},t_{i-1}+1,s_{2}+1,s_{3},\cdots,s_{m},t_{i}+s_{1}+1,t_{i+1},\cdots,t_{n}),
\end{align}
where for $i=1$ we have
\begin{align}\label{1011}
\mathcal{T}\bullet_{1}\mathcal{S}&=(t_{1},\cdots,t_{n})\bullet_{1}(s_{1},\cdots,s_{m}) \nonumber \\
&\,\,\,\,\,\,\,\,\,(t_{1}+s_{1}+1,,t_{2},\cdots,t_{n}+1,s_{2}+1,s_{3},\cdots,s_{m}),
\end{align}
for which one considers that $t_{0}$ coincides with $t_{n}$.
\begin{example}
\begin{equation}\label{102}
(2,2,1,3,1)\bullet_{2}(3,1,3,1,3,1)=(3,2,3,1,3,1,6,1,3,1),
\end{equation}
\begin{equation}\label{103}
(2,2,1,3,1)\bullet_{1}(3,1,3,1,3,1)=(6,2,1,3,2,2,3,1,3,1).
\end{equation}
\end{example}
\begin{theorem}For all $\mathcal{T}\in\mathcal{QS}(n)$ and any $\mathcal{S}\in\mathcal{QS}(m)$, we have that $\mathcal{T}\bullet_{i}\mathcal{S}\in\mathcal{QS}(n+m-1)$ for $i=1,\cdots,n$, where $n,m\geq 3$.
\end{theorem}
\begin{proof}In fact, $\mathcal{T}\bullet_{i}\mathcal{S}$ is the quiddity sequence of a triangulation $(\mathcal{T}\bullet_{i}\mathcal{S})_{\bigtriangleup}$
of a polygon $\mathcal{P}_{n+m-1}$ of $n+m-1$ vertices. We know that $\mathcal{T}$ is determined by a triangulation $\mathcal{T}_{\bigtriangleup}$ of a polygon of $n$ vertices $\mathcal{P}_{n}$, and in the same form $\mathcal{S}$ corresponds to a triangulation $\mathcal{S}_{\bigtriangleup}$ of a polygon
$\mathcal{P}_{m}$ of $m$ vertices. We remember that the vertices of any $s$-gon are labeled counterclockwise by the set $\{1,\cdots,s\}$.
To create $(\mathcal{T}\bullet_{i}\mathcal{S})_{\bigtriangleup}$ we overlap the vertex $i$ of $\mathcal{P}_{n}$ with the vertex $1$ of $\mathcal{P}_{m}$ and we connect the vertex $i-1$ of $\mathcal{P}_{n}$ with the vertex $2$ of $\mathcal{P}_{m}$ with a segment, giving rise to a polygon $\mathcal{P}_{n+m-1}$ with $n+m-1$ vertices. As it is customary, vertex $0$ of $\mathcal{P}_{n}$ is identified with its vertex $n$.
The vertices of this new polygon are labeled starting from the vertex $1$ of $\mathcal{P}_{n}$. One can then see that the triangularization of
$\mathcal{P}_{n+m-1}$ is generated by $(\mathcal{T}\bullet_{i}\mathcal{S})_{\bigtriangleup}=\mathcal{T}_{\bigtriangleup}\cup\mathcal{S}_{\bigtriangleup}\cup(i-1,i)\mathcal{P}_{n}
\cup(1,2)\mathcal{P}_{n}$. One can check that the quiddity sequence of this triangulation is precisely $\mathcal{T}\bullet_{i}\mathcal{S}$.
\end{proof}

\section{Generalized quiddity sequences which decompose the matrices $\pm\, Id$}\label{secllamada1}

We start this section by summarizing the results which will be presented here\,: a) first, we extend the products introduced in the previous section for usual quiddity sequences of triangulations to the more general space $\mathfrak{D}_{3d}$ of all $3d$-quiddity sequences corresponding to $3d$-dissections which turns out to be closed under this products, b) on the other hand, in subsection 3.2 some products are introduced
on the set of quiddity sequences ($Id$-quiddity sequences) associated with the identity matrix.

Given a sequence $(a_{1},\cdots,a_{n})\in\mathbb{C}^{n}$ and $n\in\mathbb{N}$ arbitrary, one introduces, as usual, the matrix $M_{n}(a_{1},\cdots,a_{n})\in SL_{2}(\mathbb{C})$ defined by the product
\begin{equation}\label{sec2-1}
M_{n}(a_{1},\cdots,a_{n})=\left(
                            \begin{array}{cc}
                              a_{n} & -1 \\
                              1 & 0 \\
                            \end{array}
                          \right)\left(
                            \begin{array}{cc}
                              a_{n-1} & -1 \\
                              1 & 0 \\
                            \end{array}
                          \right)\cdots\left(
                            \begin{array}{cc}
                              a_{1} & -1 \\
                              1 & 0 \\
                            \end{array}
                          \right)=M(a_{n})\cdots M(a_{1}).
\end{equation}

\begin{definition}Let us suppose that $2\leq n$, a finite sequence $(a_{1},\cdots,a_{n})\in\mathbb{C}^{n}$ will be called a \textbf{generalized quiddity sequence} of length $n$ if $M_{n}(a_{1},\cdots,a_{n})=-Id$ where $M_{n}(a_{1},\cdots,a_{n})$ is given by (\ref{sec2-1}), that is,
the generalized quiddity sequences are, in the terminology of section $1$, neither more nor less than what was previously called for us $(-I_{2})$-quiddity sequences.
\end{definition}

This section is dedicated to study the following question: how to construct a new generalized quiddity sequences from two given generalized quiddity sequences?. As we can see below, the products introduced in the previous section on triangulations may help answer, in a certain sense, the question asked recently.

\subsection{Products on the space of generalized $3d$-quiddity sequences}

In \cite{ovsienko1}, V. Ovsienko has characterized the sequence of \textbf{positive integers} $(a_{1},\cdots,a_{n})$ such that
$M_{n}(a_{1},\cdots,a_{n})=-Id$ in terms of $3d$-dissections introduced by him. This problem was studied in \cite{Conway}, \cite{Coxeter}.
As it was already mentioned before, a result of Conway and Coxeter in the last mentioned papers, establishes a one-to-one correspondence between the solutions of the equation $M_{n}(a_{1},\cdots,a_{n})=-Id$, such that, $a_{1}+a_{2}+\cdots+a_{n}=3n-6$, and triangulations of $n$-gons.

\begin{definition}(See \cite{ovsienko1}) A $3d$-dissection is a partition of a convex $n$-gon into sub-polygons by means of pairwise non-crossing diagonals, such that the number of vertices of every sub-polygon is
a multiple of $3$. The $3d$-quiddity sequence of a $3d$-dissection of an $n$-gon is the (cyclically ordered) $n$-tuple of numbers $(a_{1},\cdots, a_{n})$ such that $a_{i}$ is the number of sub-polygons adjacent to
$i$-th vertex of the $n$-gon. A triangulation of an $n$-gon is justly a $3d$-dissection in which any sub-polygon of the partition is a triangle.

A $3d$-dissection is said to be even if the total number of sub-polygons with even number of vertices ($6$-gons, $12$-gons,\ldots) is even. Otherwise we say that the $3d$-dissection is odd. Hence, triangulations can be considered even $3d$-dissections. In the same way the $3d$-quiddity sequence $\mathcal{S}$ corresponding to a $3d$-dissection $\mathcal{S}_{\triangle}$ of an $n$-gon will be called even (resp. odd) if the $3d$-dissection $\mathcal{S}_{\triangle}$ is even (resp. odd).
\end{definition}

V. Ovsienko proved the following assertion: \textbf{every $3d$-quiddity sequence $(a_{1},\cdots, a_{n})$ of an even $3d$-dissection of an $n$-gon is a solution of the equation $M_{n}(a_{1},\cdots,a_{n})=-Id$. Conversely, every solution of positive integers $(a_{1},\cdots, a_{n})$ of the equation $M_{n}(a_{1},\cdots,a_{n})=-Id$ is a $3d$-quiddity sequence of an even $3d$-dissection of an $n$-gon}. It implies that all $3d$-quiddity sequence is a generalized quiddity sequence and on the other hand all generalized quiddity sequence of positive integers must be a $3d$-quiddity sequence.

The $3d$-quiddity sequences are related with other topics, for example consider the linear equation
\begin{equation}\label{sec2-2}
v_{i+1}=a_{i}v_{i}-v_{i-1},
\end{equation}
with known coefficients $(a_{i})_{i\in\mathbb{Z}}$ and (indeterminate) sequence $(v_{i})_{i\in\mathbb{Z}}$. There is a one-to-one
correspondence between even $3d$-quiddity sequences
(that is solutions of the equation $M_{n}(a_{1},\cdots,a_{n})=-Id$ with $a_{i}\in \mathbb{N}$ for $i=1,\ldots,n$) and the equation (\ref{sec2-2})
with positive integer $n$-periodic coefficients $a_{i}$ such that every solution $(v_{i})_{i\in\mathbb{Z}}$ of (\ref{sec2-2}) is $n$-antiperiodic ($v_{i+n}=-v_{i}$ for all $i$).

Let $S$ be a $3d$-dissection, denote by $\mathfrak{i}(S)$ the number of sub-polygons with even number of vertices ($6$-gons, $12$-gons,\ldots)
that are part of $S$. It will be called the Ovsienko index.
We write $\mathfrak{D}_{3d}$ for the set of all $3d$-dissections, then $\mathfrak{D}_{3d}=\mathfrak{ED}_{3d}\sqcup \mathfrak{OD}_{3d}$ where
$$\mathfrak{ED}_{3d}=\sqcup_{m\in \mathbb{N}\cup\{0\}}\,\,\mathfrak{ED}_{3d}^{m}=\{S\in \mathfrak{D}_{3d}|\,\, \mathfrak{i}(S)=2m\},$$ and
$$\mathfrak{OD}_{3d}=\sqcup_{m\in \mathbb{N}}\,\,\mathfrak{OD}_{3d}^{^{m}}=\{S\in \mathfrak{D}_{3d}|\,\, \mathfrak{i}(S)=2m+1\},$$
in other words, $\mathfrak{ED}_{3d}$ (resp. $\mathfrak{OD}_{3d}$) is the class of even (resp. odd) $3d$-dissections. Note that
the set $\mathfrak{ED}_{3d}^{0}$ corresponds to the triangulations.
The Ovsienko index should
be extended to $3d$-quiddity sequences, thus if $\mathcal{S}$ is a $3d$-quiddity sequences we define $i(\mathcal{S})=i(\mathcal{S}_{\triangle})$ where $\mathcal{S}_{\triangle}$ is the $3d$-dissection giving rise to $\mathcal{S}$. We also extend the notation, in this case $\mathcal{S}\in \mathfrak{ED}_{3d}$ (resp. $\mathcal{S}\in \mathfrak{OD}_{3d}$) means that $\mathcal{S}_{\triangle}\in \mathfrak{ED}_{3d}$ (resp. $\mathcal{S}_{\triangle}\in \mathfrak{OD}_{3d}$).

\textbf{Clearly, we can extend the products (\ref{producto})-(\ref{productocasoextremo}) (the $\circ_{k}$) and (\ref{101})-(\ref{1011}) (the
$\bullet_{k}$) to $3d$-dissections. In fact, these products can be described by glueing two $3d$-dissections through of a triangle being the final result a new $3d$-dissection. First the vertices of each dissection are labeled counterclockwise, then we proceed exactly the same as it was done in the previous section: overlapping a vertex of the first dissection with the vertex $1$ of the second dissection and joining with a segment a pair of adjacent vertices to those who were overlapped}. Observe that the dissection obtained corresponds to a polygon whose number of vertices is equal to the sum of the amount of vertices of each dissection minus one.

\begin{example}The $3d$-quiddity sequence $(1,1,1,1,1,1)$ represents a hexagon, then $(1,1,1,1,1,1)\circ_{1}(1,1,1,1,1,1)=(3,1,1,1,1,2,2,1,1,1,1)$ is an even $3d$-dissection of two hexagons and one triangle. Now
\begin{align*}
M_{11}(3,1,1,1,1,2,2,1,1,1,1)&=M_{4}(1,1,1,1)M_{2}(2,2)M_{4}(1,1,1,1)M_{1}(3) \\
&=\left(
    \begin{array}{cc}
      -1 & 1 \\
      -1 & 0 \\
    \end{array}
  \right)\left(
    \begin{array}{cc}
      3 & -2 \\
      2 & -1 \\
    \end{array}
  \right)\left(
    \begin{array}{cc}
      -1 & 1 \\
      -1 & 0 \\
    \end{array}
  \right)\left(
    \begin{array}{cc}
      3 & -1 \\
      1 & 0 \\
    \end{array}
  \right)=-Id,
\end{align*}
which is in correspondence with the result reported by Ovsienko already mentioned above.

On other hand, $(1,1,1,1,1,1)\bullet_{2}(1,1,1,1,1,1)=(2,2,1,1,1,1,3,1,1,1,1)$. Hence, coinciding with the criterion of Ovsienko we obtain
\begin{align*}
M_{11}(2,2,1,1,1,1,3,1,1,1,1)&=M_{4}(1,1,1,1)M_{1}(3)M_{4}(1,1,1,1)M_{2}(2,2) \\
&=\left(
    \begin{array}{cc}
      -1 & 1 \\
      -1 & 0 \\
    \end{array}
  \right)\left(
    \begin{array}{cc}
      3 & -1 \\
      1 & 0 \\
    \end{array}
  \right)\left(
    \begin{array}{cc}
      -1 & 1 \\
      -1 & 0 \\
    \end{array}
  \right)\left(
    \begin{array}{cc}
      3 & -2 \\
      2 & -1 \\
    \end{array}
  \right)=-Id,
\end{align*}
\end{example}

The following proposition is very useful because this shows that the Ovsienko index turns out to be additive with respect to products (\ref{producto})-(\ref{productocasoextremo}) and (\ref{101})-(\ref{1011}) of $3d$-dissections.

\begin{proposition}Suppose that $\mathcal{S}=(s_{1},\cdots,s_{n}), \mathcal{T}=(t_{1},\cdots,t_{m})\in\mathfrak{D}_{3d}$ where $n,m\geq 3$ then
\begin{equation*}
\mathfrak{i}(\mathcal{S}\circ_{k}\mathcal{T})=\mathfrak{i}(\mathcal{T})+\mathfrak{i}(\mathcal{S}),\,\,\,\,\,k=1,2,\ldots,n,
\end{equation*}
and
\begin{equation*}
\mathfrak{i}(\mathcal{S}\bullet_{k}\mathcal{T})=\mathfrak{i}(\mathcal{T})+\mathfrak{i}(\mathcal{S}),\,\,\,\,\,k=1,2,\ldots,n,
\end{equation*}
\end{proposition}
\begin{proof}
Straightforward.
\end{proof}

Hence, we have

\begin{corollary}Assume that $\mathcal{S}=(s_{1},\cdots,s_{n}), \mathcal{T}=(t_{1},\cdots,t_{m})\in \mathfrak{ED}_{3d}$ then $\mathcal{S}
\circ_{k}\mathrm{T}$ and $\mathcal{S}\bullet_{k}\mathcal{T}$ belong to $\mathfrak{ED}_{3d}$. Additionally, if $\mathcal{S}=(s_{1},\cdots,s_{n}), \mathcal{T}=(t_{1},\cdots,t_{m})\in \mathfrak{OD}_{3d}$ we have
$\mathcal{S}\circ_{k}\mathcal{T}, \mathcal{S}\bullet_{k}\mathcal{T}\in \mathfrak{ED}_{3d}$.
\end{corollary}

As a consequence of all this we have the following theorem

\begin{theorem}
Two even (resp. two odd) $3d$-dissections $\mathcal{S}=(s_{1},\cdots,s_{n})$ and $\mathcal{T}=(t_{1},\cdots,t_{m})$ give place to $2n$ new solutions
of positive integers $\mathcal{S}\circ_{k}\mathcal{T}$, $\mathcal{S}\bullet_{k}\mathcal{T}$ for $k=1,2,\ldots,n$ of the equation $M_{n}(a_{1},\cdots,a_{n+m-1})=-Id$.
\end{theorem}

In this point, we leave any type of quiddity sequences of positive integers. We denote by $\mathcal{A}_{n}$ the set of all generalized quiddity sequences of length $n$. Observe that $\mathcal{A}_{2}=\{(0,0)\}$ and $\mathcal{A}_{3}=\{(1,1,1)\}$.
Next, we will see that the products $\circ_{k}$ and $\bullet_{k}$ can be defined between elements of the class $\mathfrak{A}=\sqcup_{n\geq 2}\mathcal{A}_{n}$. From now on, we write $M(a_{1},\cdots,a_{n})$
instead of $M_{n}(a_{1},\cdots,a_{n})$.

\begin{example}Let us suppose that $\lambda\neq 0 \in \mathbb{C}$ then one can prove that $(1,\lambda+1,\frac{2}{\lambda},\lambda,\frac{2}{\lambda}+1)\in \mathcal{A}_{5}$. Really, this $5$-sequence induces a complex-valued frieze pattern.
\end{example}

\begin{theorem}Let $A=(a_{1},\cdots,a_{n})\in \mathcal{A}_{n}$ and $B=(b_{1},\cdots,b_{m})\in \mathcal{A}_{m}$ be two generalized quiddity
sequences. Then $A\circ_{k}B, A\bullet_{k}B \in \mathcal{A}_{n+m-1}$ for $k=1,\ldots,n$.
\end{theorem}
\begin{proof}First, we will do the proof  for the products $\circ_{k}$. Suppose that $k\neq n$ (we recall that $(a_{1},\cdots,a_{n})\in \mathbb{C}^{n}$ and $(b_{1},\cdots,b_{m})\in \mathbb{C}^{m})$ then
\begin{align*}
M(A\circ_{k} B)&=M(a_{1},\cdots,a_{k-1},a_{k}+b_{1}+1,b_{2},\cdots,b_{m}+1,a_{k+1}+1,a_{k+2},\cdots,a_{n}) \\
&=M(a_{n})\cdots M(a_{k+2})M(a_{k+1}+1)M(b_{m}+1)\cdots M(b_{2})M(a_{k}+b_{1}+1)M(a_{k-1})\cdots M(a_{1}),
\end{align*}
now, we know that
\begin{equation*}
M(a_{k+1}+1)M(b_{m}+1)=M(a_{k+1})\left(
                                    \begin{array}{cc}
                                      1 & 1 \\
                                      -1 & 0 \\
                                    \end{array}
                                  \right)
M(b_{m}),
\end{equation*}
thus
\begin{equation*}
M(A\circ_{k} B)=M(a_{n})\cdots M(a_{k+2})M(a_{k+1})\left(
                                                      \begin{array}{cc}
                                                        1 & 1 \\
                                                        -1 & 0 \\
                                                      \end{array}
                                                    \right)
M(b_{m})\cdots M(b_{2})M(a_{k}+b_{1}+1)M(a_{k-1})\cdots M(a_{1}),
\end{equation*}
moreover
\begin{equation*}
M(a_{k}+b_{1}+1)=\left(
                   \begin{array}{cc}
                     a_{k}+b_{1}+1 & -1 \\
                     1 & 0 \\
                   \end{array}
                 \right)=\left(
                   \begin{array}{cc}
                     b_{1} & -1 \\
                     1 & 0 \\
                   \end{array}
                 \right)\left(
                   \begin{array}{cc}
                     0 & 1 \\
                     -1 & -1 \\
                   \end{array}
                 \right)\left(
                   \begin{array}{cc}
                     a_{k} & -1 \\
                     1 & 0 \\
                   \end{array}
                 \right)=M(b_{1})\left(
                   \begin{array}{cc}
                     0 & 1 \\
                     -1 & -1 \\
                   \end{array}
                 \right)M(a_{k}),
\end{equation*}
it shows that
\begin{equation*}
M(A\circ_{k} B)=M(a_{n})\cdots M(a_{k+2})M(a_{k+1})\left(
                                                      \begin{array}{cc}
                                                        1 & 1 \\
                                                        -1 & 0 \\
                                                      \end{array}
                                                    \right)
M(b_{m})\cdots M(b_{2})M(b_{1})\left(
                   \begin{array}{cc}
                     0 & 1 \\
                     -1 & -1 \\
                   \end{array}
                 \right)M(a_{k})M(a_{k-1})\cdots M(a_{1}).
\end{equation*}

However, by our assumptions $M(b_{1},\cdots,b_{m})=M(b_{m})\cdots M(b_{2})M(b_{1})=-Id$. It implies
\begin{align*}
M(A\circ_{k} B)&=M(a_{n})\cdots M(a_{k+2})M(a_{k+1})\left(
                                                      \begin{array}{cc}
                                                        1 & 1 \\
                                                        -1 & 0 \\
                                                      \end{array}
                                                    \right)
\left(
  \begin{array}{cc}
    -1 & 0 \\
    0 & -1 \\
  \end{array}
\right)
\left(
                   \begin{array}{cc}
                     0 & 1 \\
                     -1 & -1 \\
                   \end{array}
                 \right)M(a_{k})M(a_{k-1})\cdots M(a_{1}) \\
                 &=M(a_{n})\cdots M(a_{k+2})M(a_{k+1})M(a_{k})M(a_{k-1})\cdots M(a_{1})=M(a_{1},\cdots,a_{n})=-Id,
\end{align*}
because $A=(a_{1},\cdots,a_{n})\in \mathcal{A}_{n}$ and
\begin{equation*}
\left(
                                                      \begin{array}{cc}
                                                        1 & 1 \\
                                                        -1 & 0 \\
                                                      \end{array}
                                                    \right)
\left(
  \begin{array}{cc}
    -1 & 0 \\
    0 & -1 \\
  \end{array}
\right)
\left(
                   \begin{array}{cc}
                     0 & 1 \\
                     -1 & -1 \\
                   \end{array}
                 \right)=\left(
                           \begin{array}{cc}
                             1 & 0 \\
                             0 & 1 \\
                           \end{array}
                         \right).
\end{equation*}

Consider now that case $k=n$, we have
\begin{align*}
M(A\circ_{n} B)&=M(a_{1}+1,\cdots,a_{n-1},a_{n}+b_{1}+1,b_{2},\cdots,b_{m}+1) \\
&=M(b_{m}+1)M(b_{m-1})\cdots M(b_{2})M(b_{1}+a_{n}+1)M(a_{n-1})\cdots M(a_{1}+1) \\
&=\left(
    \begin{array}{cc}
      1 & 1 \\
      0 & 1 \\
    \end{array}
  \right)
M(b_{m})M(b_{m-1})\cdots M(b_{2})M(b_{1}+a_{n}+1)M(a_{n-1})\cdots M(a_{1})\left(
                                                                            \begin{array}{cc}
                                                                              1 & 0 \\
                                                                              -1 & 1 \\
                                                                            \end{array}
                                                                          \right) \\
&=\left(
    \begin{array}{cc}
      1 & 1 \\
      0 & 1 \\
    \end{array}
  \right)
M(b_{m})M(b_{m-1})\cdots M(b_{2})M(b_{1})\left(
                                           \begin{array}{cc}
                                             0 & 1 \\
                                             -1 & -1 \\
                                           \end{array}
                                         \right)
M(a_{n})M(a_{n-1})\cdots M(a_{1})\left(
                                                                            \begin{array}{cc}
                                                                              1 & 0 \\
                                                                              -1 & 1 \\
                                                                            \end{array}
                                                                          \right) \\
                                                                          &=\left(
    \begin{array}{cc}
      1 & 1 \\
      0 & 1 \\
    \end{array}
  \right)
M(b_{1},\cdots,b_{m})\left(
                                           \begin{array}{cc}
                                             0 & 1 \\
                                             -1 & -1 \\
                                           \end{array}
                                         \right)
M(a_{1},\cdots,a_{n})\left(
                                                                            \begin{array}{cc}
                                                                              1 & 0 \\
                                                                              -1 & 1 \\
                                                                            \end{array}
                                                                          \right),
\end{align*}
so, taking into account that $M(a_{1},\cdots,a_{n})=M(b_{1},\cdots,b_{m})=-Id$, it implies
\begin{equation*}
M(A\circ_{n} B)=\left(
    \begin{array}{cc}
      1 & 1 \\
      0 & 1 \\
    \end{array}
  \right)\left(
    \begin{array}{cc}
      -1 & 0 \\
      0 & -1 \\
    \end{array}
  \right)\left(
                                           \begin{array}{cc}
                                             0 & 1 \\
                                             -1 & -1 \\
                                           \end{array}
                                         \right)
  \left(
    \begin{array}{cc}
      -1 & 0 \\
      0 & -1 \\
    \end{array}
  \right)\left(
                                                                            \begin{array}{cc}
                                                                              1 & 0 \\
                                                                              -1 & 1 \\
                                                                            \end{array}
                                                                          \right)=-Id.
\end{equation*}

Now, we focus our attention on the products $\bullet_{k}$. Let us assume that $k\neq 1$ then
\begin{align*}
M(A\bullet_{k} B)&=M(a_{1},\cdots,a_{k-2},a_{k-1}+1,b_{2}+1,b_{3}\cdots,b_{m},a_{k}+b_{1}+1,a_{k+1},\cdots,a_{n}) \\
&=M(a_{n})\cdots M(a_{k+1})M(a_{k}+b_{1}+1)M(b_{m})\cdots M(b_{3})M(b_{2}+1)M(a_{k-1}+1)M(a_{k-2})\cdots M(a_{1}) \\
&=M(a_{n})\cdots M(a_{k+1})M(a_{k})\left(
                                     \begin{array}{cc}
                                       0 & 1 \\
                                       -1 & -1 \\
                                     \end{array}
                                   \right)
M(b_{1})M(b_{m})\cdots M(b_{3})M(b_{2}+1) \\
&\,\,\,\,\,\,\,\,\,M(a_{k-1}+1)M(a_{k-2})\cdots M(a_{1}) \\
&=M(a_{n})\cdots M(a_{k+1})M(a_{k})\left(
                                     \begin{array}{cc}
                                       0 & 1 \\
                                       -1 & -1 \\
                                     \end{array}
                                   \right)
M(b_{1})M(b_{m})\cdots M(b_{3})M(b_{2})\left(
                                         \begin{array}{cc}
                                           1 & 1 \\
                                           -1 & 0 \\
                                         \end{array}
                                       \right) \\
&\,\,\,\,\,\,\,\,\,M(a_{k-1})M(a_{k-2})\cdots M(a_{1}).
\end{align*}

Clearly, we have $M(b_{2},b_{3},\cdots,b_{m},b_{1})=M(b_{1})M(b_{m})\cdots M(b_{3})M(b_{2})=-Id$. Indeed, we have assumed that the equality
\begin{equation*}
M(b_{1},\cdots,b_{m})=M(b_{m})\cdots M(b_{1})=-Id,
\end{equation*}
is hold, then
\begin{equation*}
M(b_{m})\cdots M(b_{2})=-(M(b_{1}))^{-1},
\end{equation*}
thus
\begin{equation*}
M(b_{2},b_{3},\cdots,b_{m},b_{1})=M(b_{1})M(b_{m})\cdots M(b_{2})=-Id.
\end{equation*}

It implies
\begin{align*}
M(A\bullet_{k} B)&=M(a_{n})\cdots M(a_{k+1})M(a_{k})\left(
                                     \begin{array}{cc}
                                       0 & 1 \\
                                       -1 & -1 \\
                                     \end{array}
                                   \right)\left(
                                            \begin{array}{cc}
                                              -1 & 0 \\
                                              0 & -1 \\
                                            \end{array}
                                          \right)\left(
                                         \begin{array}{cc}
                                           1 & 1 \\
                                           -1 & 0 \\
                                         \end{array}
                                       \right)M(a_{k-1})M(a_{k-2})\cdots M(a_{1}) \\
&=M(a_{n})\cdots M(a_{k+1})M(a_{k})\left(
                                            \begin{array}{cc}
                                              1 & 0 \\
                                              0 & 1 \\
                                            \end{array}
                                          \right)M(a_{k-1})M(a_{k-2})\cdots M(a_{1})=-Id.
\end{align*}

For $k=1$
\begin{align*}
M(A\bullet_{1} B)&=M(a_{1}+b_{1}+1,a_{2},\cdots,a_{n}+1,b_{2}+1,b_{3},\cdots,b_{m}) \\
&=M(b_{m})\cdots M(b_{3})M(b_{2}+1)M(a_{n}+1)M(a_{n-1})\cdots M(a_{2})M(a_{1}+b_{1}+1) \\
&=M(b_{m})\cdots M(b_{3})M(b_{2}+1)M(a_{n}+1)M(a_{n-1})\cdots M(a_{2})M(a_{1})\left(
                                                                                \begin{array}{cc}
                                                                                  0 & 1 \\
                                                                                  -1 & -1 \\
                                                                                \end{array}
                                                                              \right)M(b_{1}) \\
&=M(b_{m})\cdots M(b_{3})M(b_{2})\left(
                                   \begin{array}{cc}
                                     1 & 1 \\
                                     -1 & 0 \\
                                   \end{array}
                                 \right)
M(a_{n})M(a_{n-1})\cdots M(a_{2})M(a_{1})\left(
                                                                                \begin{array}{cc}
                                                                                  0 & 1 \\
                                                                                  -1 & -1 \\
                                                                                \end{array}
                                                                              \right)M(b_{1}) \\
&=M(b_{m})\cdots M(b_{3})M(b_{2})\left(
                                   \begin{array}{cc}
                                     1 & 1 \\
                                     -1 & 0 \\
                                   \end{array}
                                 \right)
M(a_{1}\cdots, a_{n})\left(
                                                                                \begin{array}{cc}
                                                                                  0 & 1 \\
                                                                                  -1 & -1 \\
                                                                                \end{array}
                                                                              \right)M(b_{1}) \\
&=M(b_{m})\cdots M(b_{3})M(b_{2})\left(
                                   \begin{array}{cc}
                                     1 & 1 \\
                                     -1 & 0 \\
                                   \end{array}
                                 \right)
\left(
  \begin{array}{cc}
    -1 & 0 \\
    0 & -1 \\
  \end{array}
\right)
\left(
\begin{array}{cc}
0 & 1 \\
-1 & -1 \\
\end{array}
\right)M(b_{1}) \\
&=M(b_{1},\cdots,b_{m})=-Id,
\end{align*}
it concludes the proof of this theorem.
\end{proof}

We notice that there are other operations that can be defined with similar properties. For example, if $\mathcal{S}$ and $\mathcal{T}$
are two $3d$-quiddity sequences which correspond to the $3d$-dissections $\mathcal{S}_{\triangle}$ and $\mathcal{T}_{\triangle}$ respectively,
then when we identify the segment joining the vertices $k$ and $k+1$ of the $3d$-dissection $\mathcal{S}_{\triangle}$ with the segment joining
the first and last vertices of the $3d$-dissection $\mathcal{T}_{\triangle}$, we obtain a $3d$-dissection $\mathcal{R}_{\triangle}$ for which its $3d$-quiddity sequence $\mathcal{R}$ satisfies the following property $i(R) = i(S)+i(T)$. This construction (and other similar ones) allows us
to construct new operations between generalized quiddity sequences.

Next, we introduce new interesting operations between generalized quiddity sequences. For $A=(a_{1},\cdots,a_{n})\in \mathcal{A}_{n}$ and $B=(b_{1},\cdots,b_{m})\in \mathcal{A}_{m}$ we define
\begin{equation}\label{productolado1}
A\boxplus_{i}B=(a_{1},\cdots,a_{n})\boxplus_{i}(b_{1},\cdots,b_{m})=(a_{1},\cdots,a_{i-1},a_{i}+b_{1},b_{2},\cdots,b_{n-1},b_{n}+a_{i+1},a_{i+2},\cdots,a_{n}),
\end{equation}
for $i=1,\cdots,n-1$, and
\begin{equation}\label{productolado2}
A\boxplus_{n}B=(a_{1},\cdots,a_{n})\boxplus_{n}(b_{1},\cdots,b_{m})=(a_{1}+b_{n},a_{2},\cdots,a_{n-1},a_{n}+b_{1},b_{2},\cdots,b_{n-1}).
\end{equation}

\begin{theorem}If $A=(a_{1},\cdots,a_{n})\in \mathcal{A}_{n}$ and $B=(b_{1},\cdots,b_{m})\in \mathcal{A}_{m}$ then $A\boxplus _{k} B\in \mathcal{A}_{n+m-2}$ for $k=1,\cdots,n$.
\end{theorem}

\begin{proof}Suppose first that $i\neq n$. Then
\begin{align*}
M(A\boxplus_{i}B)&=M(a_{1},\cdots,a_{i-1},a_{i}+b_{1},b_{2},\cdots,b_{n-1},b_{n}+a_{i+1},a_{i+2},\cdots,a_{n}) \\
&=M(a_{n})\cdots M(a_{i+2})M(b_{n}+a_{i+1})M(b_{n-1})\cdots M(b_{2})M(a_{i}+b_{1})M(a_{i-1})\cdots M(a_{1}) \\
&=M(a_{n})\cdots M(a_{i+2})M(a_{i+1})\left(
                                     \begin{array}{cc}
                                       0 & 1 \\
                                       -1 & 0 \\
                                     \end{array}
                                   \right)
M(b_{n})M(b_{n-1})\cdots M(b_{2})M(b_{1})\left(
                                     \begin{array}{cc}
                                       0 & 1 \\
                                       -1 & 0 \\
                                     \end{array}
                                   \right)
M(a_{i})   \\
&\,\,\,\,\,\,\,\,M(a_{i-1})\cdots M(a_{1})    \\
&=M(a_{n})\cdots M(a_{i+2})M(a_{i+1})\left(
                                     \begin{array}{cc}
                                       0 & 1 \\
                                       -1 & 0 \\
                                     \end{array}
                                   \right)\left(
                                            \begin{array}{cc}
                                              -1 & 0 \\
                                              0 & -1 \\
                                            \end{array}
                                          \right)\left(
                                     \begin{array}{cc}
                                       0 & 1 \\
                                       -1 & 0 \\
                                     \end{array}
                                   \right)
M(a_{i})M(a_{i-1})\cdots M(a_{1}) \\
&=-Id,
\end{align*}
because
\begin{equation*}
\left(
                                     \begin{array}{cc}
                                       0 & 1 \\
                                       -1 & 0 \\
                                     \end{array}
                                   \right)\left(
                                            \begin{array}{cc}
                                              -1 & 0 \\
                                              0 & -1 \\
                                            \end{array}
                                          \right)\left(
                                     \begin{array}{cc}
                                       0 & 1 \\
                                       -1 & 0 \\
                                     \end{array}
                                   \right)=\left(
                                             \begin{array}{cc}
                                               1 & 0 \\
                                               0 & 1 \\
                                             \end{array}
                                           \right).
\end{equation*}

For $i=n$ we obtain
\begin{align*}
M(A\boxplus_{n}B)&=M(a_{1}+b_{n},a_{2},\cdots,a_{n-1},a_{n}+b_{1},b_{2},\cdots,b_{n-1}) \\
&=M(b_{n-1})\cdots M(b_{2})M(a_{n}+b_{1})M(a_{n-1})\cdots M(a_{2})M(a_{1}+b_{n}) \\
&=M(b_{n-1})\cdots M(b_{2})M(b_{1})\left(
                                     \begin{array}{cc}
                                       0 & 1 \\
                                       -1 & 0 \\
                                     \end{array}
                                   \right)
M(a_{n})M(a_{n-1})\cdots M(a_{2})M(a_{1})\left(
                                           \begin{array}{cc}
                                             0 & 1 \\
                                             -1 & 0 \\
                                           \end{array}
                                         \right)
M(b_{n}) \\
&=M(b_{n-1})\cdots M(b_{2})M(b_{1})M(b_{n})=-Id.
\end{align*}
\end{proof}

\subsection{Products of quiddity sequences associated to the identity}

In this subsection we consider vectors $(a_{1},\cdots,a_{n})\in \mathbb{C}^{n}$ such that

\begin{equation}\label{secqsi-1}
M_{n}(a_{1},\cdots,a_{n})=\left(
                            \begin{array}{cc}
                              a_{n} & -1 \\
                              1 & 0 \\
                            \end{array}
                          \right)\left(
                            \begin{array}{cc}
                              a_{n-1} & -1 \\
                              1 & 0 \\
                            \end{array}
                          \right)\cdots\left(
                            \begin{array}{cc}
                              a_{1} & -1 \\
                              1 & 0 \\
                            \end{array}
                          \right)=M(a_{n})\cdots M(a_{1})=Id,
\end{equation}
and solution vectors of (\ref{secqsi-1}) will be called \textbf{$Id$-quiddity sequences}. We want to indicate that
the equation (\ref{secqsi-1}) is related with the recognition problem (or Deligne-Simpson problem) which can be formulated as follows:
given conjugacy classes $\mathbf{C}_{1},\cdots,\mathbf{C}_{n}$ determine whether or not one can solve the equation
\begin{equation}\label{secqsi-2}
\mathbf{A}_{1}\cdots\mathbf{A}_{n}=I,
\end{equation}
with $\mathbf{A}_{i}\in \mathbf{C}_{i}$ (see \cite{neto-silva}, \cite{crawley} and references in these). We are reviewing the case in which
$C_{i}=Cl(M(a_{i}))=\{AM(a_{i})A^{-1}|\,A\in GL_{2}\}$.
Let $\mathbf{a}=(a_{1},\cdots,a_{n})$ and $\mathbf{b}=(b_{1},\cdots,b_{m})$ be two $Id$-quiddity sequences. We propose products of the form
\begin{equation}\label{secqsi-3}
\mathbf{a}\circ_{k}\mathbf{b}=(a_{1},\cdots,a_{n})\circ_{k}(b_{1},\cdots,b_{m})=(a_{1},\cdots, a_{k-1},\mathbf{x}(a_{k},b_{1}),b_{2},\cdots,b_{m-1},\mathbf{y}(b_{m}),
\mathbf{w}(a_{k+1}),a_{k+2},\cdots,a_{n}),
\end{equation}
for $k=1,\cdots,n-1$, where $\mathbf{x}(a_{k},b_{1})$, $\mathbf{y}(b_{m})$ and $\mathbf{w}(a_{k+1})$ are transformations that should
be chose in such a way that $\mathbf{a}\circ_{k}\mathbf{b}$ is a new $Id$-quiddity sequence. With this purpose, we show two simple facts:
\begin{itemize}
  \item there is a matrix $Z$ such that
  \begin{equation}\label{secqsi-4}
  \left(
    \begin{array}{cc}
     a_{k+1}  & -1 \\
       1  & 0  \\
    \end{array}
  \right)\left(
           \begin{array}{cc}
             z_{11} & z_{12} \\
             z_{21} & z_{22} \\
           \end{array}
         \right)\left(
    \begin{array}{cc}
      b_{m}  & -1 \\
       1  & 0  \\
    \end{array}
  \right)=\left(
            \begin{array}{cc}
              \mathbf{w}(a_{k+1}) & -1 \\
              1 & 0 \\
            \end{array}
          \right)\left(
            \begin{array}{cc}
              \mathbf{y}(b_{m}) & -1 \\
              1 & 0 \\
            \end{array}
          \right),
  \end{equation}
  \item  It is possible to find a matrix $T$ satisfying
  \begin{equation}\label{secqsi-5}
  \left(
    \begin{array}{cc}
      \mathbf{x}(a_{k},b_{1}) & -1 \\
      1 & 0 \\
    \end{array}
  \right)=\left(
            \begin{array}{cc}
              b_{1} & -1 \\
              1 & 0 \\
            \end{array}
          \right)\left(
                   \begin{array}{cc}
                     t_{11} & t_{12} \\
                     t_{21} & t_{22} \\
                   \end{array}
                 \right)\left(
                          \begin{array}{cc}
                            a_{k} & -1 \\
                            1 & 0 \\
                          \end{array}
                        \right),
  \end{equation}
\end{itemize}
from (\ref{secqsi-4}) we obtain $z_{11}=1$, $z_{22}=z_{12}z_{21}+1$, $\mathbf{y}(b_{m})=b_{m}+z_{12}$ and $\mathbf{w}(a_{k+1})=a_{k+1}-z_{21}$.
On the other hand from (\ref{secqsi-5}) follows that $t_{11}=0$, $t_{12}=1$, $t_{21}=-1$ and $\mathbf{x}(a_{k},b_{1})=a_{k}+b_{1}-t_{22}$.

Now, to carry our goal to the end it is necessary to impose the condition
\begin{equation}\label{secqsi-6}
\left(
  \begin{array}{cc}
    1 & z_{12} \\
    z_{21} & z_{12}z_{21}+1 \\
  \end{array}
\right)\left(
         \begin{array}{cc}
           0 & 1 \\
           -1 & t_{22} \\
         \end{array}
       \right)=\left(
                 \begin{array}{cc}
                   1 & 0 \\
                   0 & 1 \\
                 \end{array}
               \right),
\end{equation}
then $z_{12}=-1$, $z_{21}=1$ and $t_{22}=1$ so
\begin{equation*}
\mathbf{x}(a_{k},b_{1})=a_{k}+b_{1}-1,\,\,\,\,\,\,\,\mathbf{y}(b_{m})=b_{m}-1,\,\,\,\,\,\,\,\mathbf{w}(a_{k+1})=a_{k+1}-1,
\end{equation*}
thus (\ref{secqsi-3}) takes of simple form
\begin{equation*}
\mathbf{a}\circ_{k}\mathbf{b}=(a_{1},\cdots,a_{n})\circ_{k}(b_{1},\cdots,b_{m})=(a_{1},\cdots, a_{k-1},a_{k}+b_{1}-1,b_{2},\cdots,b_{m-1},
b_{m}-1,a_{k+1}-1,a_{k+2},\cdots,a_{n}),
\end{equation*}
where $k=1,\cdots,n-1$.

We are able to prove the following result

\begin{theorem}Let $\mathbf{a}=(a_{1},\cdots,a_{n})$  and $\mathbf{b}=(b_{1},\cdots,b_{m})$ be two $Id$-quiddity sequences, then
$\mathbf{a}\circ_{k}\mathbf{b}$ is an $Id$-quiddity sequences for $k=1,\cdots,n$ if we define
\begin{equation}\label{secqsi-7}
\mathbf{a}\circ_{k}\mathbf{b}=(a_{1},\cdots,a_{n})\circ_{k}(b_{1},\cdots,b_{m})=(a_{1},\cdots, a_{k-1},a_{k}+b_{1}-1,b_{2},\cdots,b_{m-1},
b_{m}-1,a_{k+1}-1,a_{k+2},\cdots,a_{n}),
\end{equation}
for $k\neq n$ and
\begin{equation}\label{secqsi-8}
\mathbf{a}\circ_{n}\mathbf{b}=(a_{1}-1,a_{2},\cdots,a_{n-1},a_{n}+b_{1}-1,b_{2},\cdots,b_{m-1},b_{m}-1).
\end{equation}
\end{theorem}
\begin{proof}It remains to prove the case $k=n$. For this purpose, note that
\begin{equation*}
\left(
  \begin{array}{cc}
    b_{m}-1 & -1 \\
    1 & 0 \\
  \end{array}
\right)=\left(
          \begin{array}{cc}
            1 & -1 \\
            0 & 1 \\
          \end{array}
        \right)\left(
  \begin{array}{cc}
    b_{m} & -1 \\
    1 & 0 \\
  \end{array}
\right),\,\,\,\,\,\,\,\,\,\left(
                            \begin{array}{cc}
                              a_{1}-1 & -1 \\
                              1 & 0 \\
                            \end{array}
                          \right)=\left(
                            \begin{array}{cc}
                              a_{1} & -1 \\
                              1 & 0 \\
                            \end{array}
                          \right)\left(
                                   \begin{array}{cc}
                                     1 & 0 \\
                                     1 & 1 \\
                                   \end{array}
                                 \right),
\end{equation*}
and, on the other hand, we already know that
\begin{equation*}
\left(
  \begin{array}{cc}
    a_{n}+b_{1}-1 & -1 \\
    1 & 0 \\
  \end{array}
\right)=\left(
  \begin{array}{cc}
    b_{1} & -1 \\
    1 & 0 \\
  \end{array}
\right)\left(
         \begin{array}{cc}
           0 & 1 \\
           -1 & 1 \\
         \end{array}
       \right)\left(
  \begin{array}{cc}
    a_{n} & -1 \\
    1 & 0 \\
  \end{array}
\right).
\end{equation*}

Now,
\begin{equation*}
\left(
  \begin{array}{cc}
    1 & -1 \\
    0 & 1 \\
  \end{array}
\right)\left(
         \begin{array}{cc}
           0 & 1 \\
           -1 & 1 \\
         \end{array}
       \right)\left(
                \begin{array}{cc}
                  1 & 0 \\
                  1 & 1 \\
                \end{array}
              \right)=\left(
                        \begin{array}{cc}
                          1 & 0 \\
                          0 & 1 \\
                        \end{array}
                      \right),
\end{equation*}
it concludes the proof of the theorem.
\end{proof}

\section{Decomposition of the monodromy for block matrices. Left and right matrix quiddity sequences}

\subsection{General results}

In this section, our propose is to find matrix vectors $(\mathfrak{a}_{1},\mathfrak{a}_{2},\cdots,\mathfrak{a}_{n})$ solutions
of the equation
\begin{equation}\label{i3}
M(\mathfrak{a}_{1},\cdots,\mathfrak{a}_{n})=\left(
                         \begin{array}{cc}
                           \mathfrak{a}_{n} & -I \\
                           I & O \\
                         \end{array}
                       \right)\left(
                         \begin{array}{cc}
                           \mathfrak{a}_{n-1} & -I \\
                           I & O \\
                         \end{array}
                       \right)\cdots \left(
                         \begin{array}{cc}
                           \mathfrak{a}_{1} & -I \\
                           I & O \\
                         \end{array}
                       \right)=M(\mathfrak{a}_{n})\cdots M(\mathfrak{a}_{1})=-\left(
                                  \begin{array}{cc}
                                    I & O \\
                                    O & I \\
                                  \end{array}
                                \right)=-Id,
\end{equation}
where the $\mathfrak{a}_{k}$ are square $l\times l$ matrices for $k=\overline{1,n}$, being $I$ and $O$ the identity matrix and the null
matrix $l\times l$ respectively. It is important to observe that we do not suppose that the entries of each $\mathfrak{a}_{k}$ are positive integers, neither we do not assume that $|\mathfrak{a}_{k}|\neq 0$ for all $k$. Such sequences will be called \textbf{left matrix quiddity sequences}.

We recall the following result which is known as the \textbf{Schur determinant lemma} (see \cite{zhang} for more details)

\begin{lemma}Let $P,Q,S,R$ denote $n\times n$ matrices and suppose that $P$ and $R$ commute. Then the determinant $|N|$ of the $2n\times 2n$
matrix
\begin{equation*}
N=\left(
    \begin{array}{cc}
      P & Q \\
      R & S \\
    \end{array}
  \right),
\end{equation*}
is equal to the determinant of the matrix $PS-RQ$.
\end{lemma}

There exists a generalization in certain sense of the previous result which can be found also in \cite{zhang} for any square matrix $N$. Consider now that $N$ is partitioned where $P,Q,S,R$ do not necessarily
have the same dimension. Suppose $P$ is nonsingular and denote the matrix $S-RP^{-1}Q$ by $N/P$ and call it the Schur complement of $P$ in $N$, or the Schur complement of $N$ relative to $P$. The following
result is well known
\begin{theorem}(Schur's Formula) Let $N$ be a square matrix partitioned. If $P$ is nonsingular, then
\begin{equation}\label{i7}
det(N/P)=\frac{det N}{det P}.
\end{equation}
\end{theorem}

\begin{remark}The equality (\ref{i3}) is well defined, this is due to $|M(\mathfrak{a}_{k})|=1$ for $k=1,\cdots,n$. In fact, from the Schur determinant lemma we obtain
\begin{equation*}
\left |
          \begin{array}{cc}
            \mathfrak{a}_{k} & -I \\
            I & O \\
          \end{array}
 \right |=|\mathfrak{a}_{k}O-I(-I)|=1,
\end{equation*}
for all $k$.
\end{remark}

\begin{example}If $|\mathfrak{a}|\neq 0$, then the sequences $(\mathfrak{a},2\mathfrak{a}^{-1},\mathfrak{a},2\mathfrak{a}^{-1})$ and
$(\mathfrak{a},2\mathfrak{a}^{-1}+I,I,\mathfrak{a}+I,2\mathfrak{a}^{-1})$ are left matrix quiddity sequences.
\end{example}

\begin{remark}Suppose that $(\mathfrak{a}_{1},\mathfrak{a}_{2},\cdots,\mathfrak{a}_{n})$ is a left matrix quiddity sequences, then
$(\mathfrak{a}_{1},\cdots,\mathfrak{a}_{i}+I,I,\mathfrak{a}_{i+1}+I,\cdots,\mathfrak{a}_{n})$ are also left matrix quiddity sequences for $i=1,\ldots,n-1$.
Indeed, for two arbitrary matrices $\mathfrak{x}$ and $\mathfrak{y}$ we have
\begin{align*}
\left(
  \begin{array}{cc}
    \mathfrak{x}+I & -I \\
    I & O \\
  \end{array}
\right)\left(
  \begin{array}{cc}
    I & -I \\
    I & O \\
  \end{array}
\right)\left(
  \begin{array}{cc}
    \mathfrak{y}+I & -I \\
    I & O \\
  \end{array}
\right)&=\left(
  \begin{array}{cc}
    \mathfrak{x} & -(\mathfrak{x}+I) \\
    I & -I \\
  \end{array}
\right)\left(
  \begin{array}{cc}
    \mathfrak{y}+I & -I \\
    I & O \\
  \end{array}
\right) \\
&=\left(
    \begin{array}{cc}
      \mathfrak{x}\mathfrak{y}-I & -\mathfrak{x} \\
      \mathfrak{y} & -I \\
    \end{array}
  \right)=\left(
  \begin{array}{cc}
    \mathfrak{x} & -I \\
    I & O \\
  \end{array}
\right)\left(
  \begin{array}{cc}
    \mathfrak{y} & -I \\
    I & O \\
  \end{array}
\right).
\end{align*}
\end{remark}

\begin{lemma}Let $(\mathfrak{a}_{1},\mathfrak{a}_{2},\cdots,\mathfrak{a}_{n})$ be a left matrix quiddity sequence, that is, $M(\mathfrak{a}_{1},\cdots,\mathfrak{a}_{n})=-Id$, then
\begin{equation}\label{mqs1}
M(\mathfrak{a}_{k+1},\cdots,\mathfrak{a}_{n},\mathfrak{a}_{1},\cdots,\mathfrak{a}_{k})=-Id,
\end{equation}
for $k=1,\cdots,n-1$.
\end{lemma}
\begin{proof}The proof of this lemma follows simply by means of operations of conjugation on both sides of equality $M(\mathfrak{a}_{1},\cdots,\mathfrak{a}_{n})=-Id$ a finite number of times.
\end{proof}

As in the previous section there is a close relation between equation (\ref{i3}) and the left matrix-recurrent
relation
\begin{equation}\label{i4}
\mathfrak{u}_{k+1}=\mathfrak{c}_{k}\mathfrak{u}_{k}-\mathfrak{u}_{k-1},\,\,\,\,\,\,\,\,\,\,\,k\in\mathbb{Z},
\end{equation}
where $(\mathfrak{c}_{k})$ is a known $n$-periodic sequence of matrices of order $l\times l$ for each $k$ and the sequence $(\mathfrak{u}_{k})_{k\in \mathbb{Z}}\subset M_{l}(\mathbb{C})$ which is not known.

\begin{proposition}There is a one-to-one correspondence between left matrix quiddity sequences (that is matrix sequences $(\mathfrak{a}_{1},\cdots,\mathfrak{a}_{n})$
for which $M(\mathfrak{a}_{1},\cdots,\mathfrak{a}_{n})=-Id$) and left matrix-recurrent relations (\ref{i4}) for which $\mathfrak{c}_{k}=\mathfrak{a}_{k}$ is an $n$-periodic sequence that have only $n$-antiperiodic solutions.
\end{proposition}
\begin{proof}Let $(\mathfrak{a}_{1},\cdots,\mathfrak{a}_{n})$ be a left matrix quiddity sequence, $\mathfrak{u}_{1}$ and $\mathfrak{u}_{0}$
arbitrary matrices of order $l$, and let $(\mathfrak{u}_{k})_{k\in \mathbb{Z}}$ be the sequence induced by these initial matrices through the
left matrix-recurrent relation (\ref{i4}). Then, from the previous lemma
\begin{align*}
\left(
  \begin{array}{c}
    \mathfrak{u}_{k+n} \\
    \mathfrak{u}_{k-1+n} \\
  \end{array}
\right)&=\left(
          \begin{array}{cc}
            \mathfrak{a}_{n+k-1} & -I \\
            I & O \\
          \end{array}
        \right)\cdots\left(
          \begin{array}{cc}
            \mathfrak{a}_{k} & -I \\
            I & O \\
          \end{array}
        \right)
\left(
          \begin{array}{c}
            \mathfrak{u}_{k} \\
            \mathfrak{u}_{k-1} \\
          \end{array}
        \right) \\
&=\left(
    \begin{array}{cc}
      -I & O \\
      O & -I \\
    \end{array}
  \right)\left(
          \begin{array}{c}
            \mathfrak{u}_{k} \\
            \mathfrak{u}_{k-1} \\
          \end{array}
        \right),
\end{align*}
it implies that $(\mathfrak{u}_{k})_{k\in\mathbb{Z}}$ is $n$-antiperiodic, that is, $\mathfrak{u}_{k+n}=-\mathfrak{u}_{k}$ for all $k\in\mathbb{Z}$.
Thus, all solution of (\ref{i4}) is $n$-antiperiodic. Conversely, if every solution of (\ref{i4}) is $n$-antiperiodic where $(\mathfrak{c}_{k})$ is $n$-periodic. Then
\begin{equation*}
\left( \left(
          \begin{array}{cc}
            \mathfrak{a}_{n} & -I \\
            I & O \\
          \end{array}
        \right)\cdots\left(
          \begin{array}{cc}
            \mathfrak{a}_{1} & -I \\
            I & O \\
          \end{array}
        \right)-\left(
                  \begin{array}{cc}
                    -I & O \\
                    O & -I \\
                  \end{array}
                \right)
         \right )\left(
                   \begin{array}{c}
                     \mathfrak{u}_{1} \\
                     \mathfrak{u}_{0} \\
                   \end{array}
                 \right)=\left(
                           \begin{array}{c}
                             O \\
                             O \\
                           \end{array}
                         \right),
\end{equation*}
where $\mathfrak{a}_{k}=\mathfrak{c}_{k}$ for $k=1,\cdots,n$. Now, being $\mathfrak{u}_{1}$ and $\mathfrak{u}_{0}$ arbitrary we obtain $M(\mathfrak{a}_{1},\cdots,\mathfrak{a}_{n})=-Id$.
\end{proof}

Consider now the equation
\begin{equation}\label{porladerecha1}
N(\mathfrak{b}_{1},\cdots,\mathfrak{b}_{n})=\left(
                         \begin{array}{cc}
                           O & -I \\
                           I & \mathfrak{b}_{1} \\
                         \end{array}
                       \right)\left(
                         \begin{array}{cc}
                           O & -I \\
                           I & \mathfrak{b}_{2} \\
                         \end{array}
                       \right)\cdots \left(
                         \begin{array}{cc}
                           O & -I \\
                           I & \mathfrak{b}_{n} \\
                         \end{array}
                       \right)=N(\mathfrak{b}_{1})\cdots N(\mathfrak{b}_{n})=-Id,
\end{equation}
the solutions of (\ref{porladerecha1}) will be called \textbf{right matrix quiddity sequences}.

\begin{remark}A different alternative to (\ref{porladerecha1}), it will be the equality
\begin{equation}\label{porladerecha01}
L(\mathfrak{b}_{1},\cdots,\mathfrak{b}_{n})=\left(
                         \begin{array}{cc}
                           O & I \\
                           -I & \mathfrak{b}_{1} \\
                         \end{array}
                       \right)\left(
                         \begin{array}{cc}
                           O & I \\
                          -I & \mathfrak{b}_{2} \\
                         \end{array}
                       \right)\cdots \left(
                         \begin{array}{cc}
                           O & I \\
                           -I & \mathfrak{b}_{n} \\
                         \end{array}
                       \right)=L(\mathfrak{b}_{1})\cdots L(\mathfrak{b}_{n})=-Id,
\end{equation}
now, since for all matrix $\mathfrak{a}$ of order $l\times l$
\begin{equation*}
\left(
  \begin{array}{cc}
    \mathfrak{a} & -I \\
    I & O \\
  \end{array}
\right)^{-1}=\left(
               \begin{array}{cc}
                 O & I \\
                 -I & \mathfrak{a} \\
               \end{array}
             \right),
\end{equation*}
this type of equality can be obtained of (\ref{i3}) applying the inverse of each matrix on the left. However, we must recognize
that (\ref{porladerecha01}) makes more sense from of point of view of moving frame theory.
\end{remark}

\begin{theorem}If $(\mathfrak{a}_{1},\cdots,\mathfrak{a}_{n})$ is a left matrix quiddity sequence then $(\mathfrak{a^{\ast}}_{1},\cdots,\mathfrak{a^{\ast}}_{n})$ is a right matrix quiddity sequence.
Conversely, if $(\mathfrak{b}_{1},\cdots,\mathfrak{b}_{n})$ is a right matrix quiddity sequence then $(\mathfrak{b^{\ast}}_{1},\cdots,\mathfrak{b^{\ast}}_{n})$ is a left matrix quiddity sequence.
\end{theorem}
\begin{proof}For any block matrix $\left(
                                     \begin{array}{cc}
                                       \mathfrak{x} & \mathfrak{y} \\
                                       \mathfrak{z} & \mathfrak{w} \\
                                     \end{array}
                                   \right)$ its conjugate block matrix with respect to the anti-diagonal is obtained in the following form
\begin{equation*}
\left( \begin{array}{cc}
                                       \mathfrak{w}^{\ast} & \mathfrak{y}^{\ast} \\
                                       \mathfrak{z}^{\ast} & \mathfrak{x}^{\ast} \\
                                     \end{array}
                                   \right)=\left(
                                             \begin{array}{cc}
                                               O & I \\
                                               I & O \\
                                             \end{array}
                                           \right)\left(
                                                    \begin{array}{cc}
                                                      \mathfrak{x} & \mathfrak{y} \\
                                                      \mathfrak{z} & \mathfrak{w} \\
                                                    \end{array}
                                                  \right)^{\ast}\left(
                                                                  \begin{array}{cc}
                                                                    O & I \\
                                                                    I & O \\
                                                                  \end{array}
                                                                \right)=\left(
                                             \begin{array}{cc}
                                               O & I \\
                                               I & O \\
                                             \end{array}
                                           \right)\left(
                                                    \begin{array}{cc}
                                                      \mathfrak{x}^{\ast} & \mathfrak{z}^{\ast} \\
                                                      \mathfrak{y}^{\ast} & \mathfrak{w}^{\ast} \\
                                                    \end{array}
                                                  \right)\left(
                                                                  \begin{array}{cc}
                                                                    O & I \\
                                                                    I & O \\
                                                                  \end{array}
                                                                \right),
\end{equation*}
where for a matrix $H$ of order $l$ arbitrary, the notation $H^{\ast}$ denotes its complex conjugate matrix. Note also that the matrix
$J=\left(
     \begin{array}{cc}
       O & I \\
       I & O \\
     \end{array}
   \right)
$ satisfies $J^{\ast}=J$ and $J^{2}=Id$.

Let us suppose that $(\mathfrak{a}_{1},\cdots,\mathfrak{a}_{n})$ is a left matrix quiddity sequence, that is
\begin{equation*}
M(\mathfrak{a}_{1},\cdots,\mathfrak{a}_{n})=\left(
                         \begin{array}{cc}
                           \mathfrak{a}_{n} & -I \\
                           I & O \\
                         \end{array}
                       \right)\left(
                         \begin{array}{cc}
                           \mathfrak{a}_{n-1} & -I \\
                           I & O \\
                         \end{array}
                       \right)\cdots \left(
                         \begin{array}{cc}
                           \mathfrak{a}_{1} & -I \\
                           I & O \\
                         \end{array}
                       \right)=M(\mathfrak{a}_{n})\cdots M(\mathfrak{a}_{1})=-\left(
                                  \begin{array}{cc}
                                    I & O \\
                                    O & I \\
                                  \end{array}
                                \right)=-Id,
\end{equation*}
and from it follows that
\begin{equation*}
M^{\ast}(\mathfrak{a}_{1},\cdots,\mathfrak{a}_{n})=M^{\ast}(\mathfrak{a}_{1})\cdots M^{\ast}(\mathfrak{a}_{n})=
\left(
  \begin{array}{cc}
    \mathfrak{a}^{\ast}_{1} & I \\
    -I &  O \\
  \end{array}
\right)\cdots \left(
  \begin{array}{cc}
    \mathfrak{a}^{\ast}_{n} & I \\
    -I &  O
     \\
  \end{array}
\right)=-\left(
                                     \begin{array}{cc}
                                       I & O \\
                                       O & I \\
                                     \end{array}
                                   \right)=- Id,
\end{equation*}
thus
\begin{align*}
-Id=-J^{2}=JM^{\ast}(\mathfrak{a}_{1},\cdots,\mathfrak{a}_{n})J&=JM^{\ast}(\mathfrak{a}_{1})J^{2}M^{\ast}(\mathfrak{a}_{2})\cdots M^{\ast}(\mathfrak{a}_{n-1})J^{2}M^{\ast}(\mathfrak{a}_{n})J \\
&=\left(
    \begin{array}{cc}
      O & I \\
      I & O \\
    \end{array}
  \right)
\left(
  \begin{array}{cc}
    \mathfrak{a}^{\ast}_{1} & I \\
    -I &  O \\
  \end{array}
\right)\left(
         \begin{array}{cc}
           O & I \\
           I & O \\
         \end{array}
       \right)
\cdots \left(
         \begin{array}{cc}
           O & I \\
           I & O \\
         \end{array}
       \right)
\left(
  \begin{array}{cc}
    \mathfrak{a}^{\ast}_{n} & I \\
    -I & O \\
  \end{array}
\right)\left(
         \begin{array}{cc}
           O & I \\
           I & O \\
         \end{array}
       \right) \\
&=\left(
  \begin{array}{cc}
    O & -I \\
    I &  \mathfrak{a}^{\ast}_{1} \\
  \end{array}
\right)\cdots \left(
  \begin{array}{cc}
    O & -I \\
    I &  \mathfrak{a}^{\ast}_{n} \\
  \end{array}
\right)=N(\mathfrak{a}^{\ast}_{1})\cdots N(\mathfrak{a}^{\ast}_{n})=N(\mathfrak{a}^{\ast}_{1},\cdots,\mathfrak{a}^{\ast}_{n}),
\end{align*}
it shows that $(\mathfrak{a}^{\ast}_{1},\cdots,\mathfrak{a}^{\ast}_{n})$ is a right matrix quiddity sequence. The reciprocal is proved in similar form.
\end{proof}

The following proposition is evident

\begin{proposition}Consider the right matrix-recurrent equation
\begin{equation}\label{porladerecha2}
\mathfrak{u}_{k+1}=\mathfrak{u}_{k}\mathfrak{c}_{k}-\mathfrak{u}_{k-1},\,\,\,\,\,\,\,\,\,\,\,k\in\mathbb{Z},
\end{equation}
where $(\mathfrak{c}_{k})$ is an $n$-periodic sequence, then all solution of (\ref{porladerecha2}) is $n$-antiperiodic if and only if
$(\mathfrak{b}_{1},\cdots,\mathfrak{b}_{n})$ is a right matrix quiddity sequence where $\mathfrak{b}_{k}=\mathfrak{c}_{k}$ for $k=1,\cdots,n$
\end{proposition}

Next, we will make some calculations to find left matrix quiddity sequence in low dimensions. Since
\begin{equation*}
\left(
  \begin{array}{cc}
    \mathfrak{a}_{5} & -I \\
    I & O \\
  \end{array}
\right)\left(
  \begin{array}{cc}
    \mathfrak{a}_{4} & -I \\
    I & O \\
  \end{array}
\right)\left(
  \begin{array}{cc}
    \mathfrak{a}_{3} & -I \\
    I & O \\
  \end{array}
\right)=\left(
          \begin{array}{cc}
            (\mathfrak{a}_{5}\mathfrak{a}_{4}-I)\mathfrak{a}_{3}-\mathfrak{a}_{5} & -(\mathfrak{a}_{5}\mathfrak{a}_{4}-I) \\
            (\mathfrak{a}_{4}\mathfrak{a}_{3}-I) & -\mathfrak{a}_{4} \\
          \end{array}
        \right),
\end{equation*}
then it follows that the only one left matrix quiddity sequence $(\mathfrak{a}_{5},\mathfrak{a}_{4},\mathfrak{a}_{3})$ is accurately $(I,I,I)$. One takes another step
\begin{align*}
M(\mathfrak{a}_{2},\mathfrak{a}_{3},\mathfrak{a}_{4},\mathfrak{a}_{5})&=\left(
  \begin{array}{cc}
    \mathfrak{a}_{5} & -I \\
    I & O \\
  \end{array}
\right)\left(
  \begin{array}{cc}
    \mathfrak{a}_{4} & -I \\
    I & O \\
  \end{array}
\right)\left(
  \begin{array}{cc}
    \mathfrak{a}_{3} & -I \\
    I & O \\
  \end{array}
\right)\left(
  \begin{array}{cc}
    \mathfrak{a}_{2} & -I \\
    I & O \\
  \end{array}
\right) \\
&=\left(
          \begin{array}{cc}
            (\mathfrak{a}_{5}\mathfrak{a}_{4}-I)\mathfrak{a}_{3}-\mathfrak{a}_{5} & -(\mathfrak{a}_{5}\mathfrak{a}_{4}-I) \\
            (\mathfrak{a}_{4}\mathfrak{a}_{3}-I) & -\mathfrak{a}_{4} \\
          \end{array}
        \right)\left(
  \begin{array}{cc}
    \mathfrak{a}_{2} & -I \\
    I & O \\
  \end{array}
\right)  \\
&=\left(
    \begin{array}{cc}
      ((\mathfrak{a}_{5}\mathfrak{a}_{4}-I)\mathfrak{a}_{3}-\mathfrak{a}_{5})\mathfrak{a}_{2}-(\mathfrak{a}_{5}\mathfrak{a}_{4}-I) & -((\mathfrak{a}_{5}\mathfrak{a}_{4}-I)\mathfrak{a}_{3}-\mathfrak{a}_{5}) \\
      ((\mathfrak{a}_{4}\mathfrak{a}_{3}-I)\mathfrak{a}_{2}-\mathfrak{a}_{4}) & -(\mathfrak{a}_{4}\mathfrak{a}_{3}-I) \\
    \end{array}
  \right),
\end{align*}
hence, so that the equation $M(\mathfrak{a}_{2},\mathfrak{a}_{3},\mathfrak{a}_{4},\mathfrak{a}_{5})=-Id$ is tru one must has  $\mathfrak{a}_{4}\mathfrak{a}_{3}-I=I$, that is, $\frac{\mathfrak{a}_{4}\mathfrak{a}_{3}}{2}=I$ meaning that $\mathfrak{a}_{4}$ and
$\mathfrak{a}_{3}$ are invertible matrices, even more $\mathfrak{a}_{4}=2\mathfrak{a}^{-1}_{3}$. From the form of the remaining entries
of the previous matrix follows that $\mathfrak{a}_{3}=\mathfrak{a}_{5}=\mathfrak{m}$ and $\mathfrak{a}_{2}=\mathfrak{a}_{4}=2\mathfrak{a}^{-1}_{3}=2\mathfrak{m}^{-1}$. Thus $(2\mathfrak{m}^{-1},\mathfrak{m},2\mathfrak{m}^{-1},\mathfrak{m})$
and its cyclical permutation are the only ones left matrix quiddity sequences of length $4$. \textbf{The last step is interesting because it leads to a possible theory of matrix-valued frieze patterns as it will be seen after}.

We have
\begin{theorem}Each left matrix quiddity sequence $(\mathfrak{a}_{1},\mathfrak{a}_{2},\mathfrak{a}_{3},\mathfrak{a}_{4},\mathfrak{a}_{5})$
with $|\mathfrak{a}_{1}|\neq 0$ and $|\mathfrak{a}_{3}|\neq 0$ is of the forma
\begin{equation}\label{cont1}
(\mathfrak{a}_{1},\mathfrak{a}_{2},\mathfrak{a}_{3},\mathfrak{a}_{4},\mathfrak{a}_{5})=(\mathfrak{c},\mathfrak{c}^{-1}(I+(\mathfrak{c}+I)
\mathfrak{d}^{-1}),\mathfrak{d},(\mathfrak{c}+I)\mathfrak{d}^{-1},
(\mathfrak{c}^{-1}+I)(\mathfrak{c}+I)^{-1}(\mathfrak{d}+I)),
\end{equation}
for which $[\mathfrak{c},\mathfrak{d}]=O$.
\end{theorem}
\begin{proof}Indeed,
\begin{align*}
M(\mathfrak{a}_{1},\mathfrak{a}_{2},\mathfrak{a}_{3},\mathfrak{a}_{4},\mathfrak{a}_{5})&=\left(
    \begin{array}{cc}
      ((\mathfrak{a}_{5}\mathfrak{a}_{4}-I)(\mathfrak{a}_{3}\mathfrak{a}_{2}-I)-\mathfrak{a}_{5}\mathfrak{a}_{2}) & -((\mathfrak{a}_{5}\mathfrak{a}_{4}-I)\mathfrak{a}_{3}-\mathfrak{a}_{5}) \\
      ((\mathfrak{a}_{4}\mathfrak{a}_{3}-I)\mathfrak{a}_{2}-\mathfrak{a}_{4}) & -(\mathfrak{a}_{4}\mathfrak{a}_{3}-I) \\
    \end{array}
  \right)\left(
           \begin{array}{cc}
             \mathfrak{a}_{1} & -I \\
             I & O \\
           \end{array}
         \right) \\
         &=\left(
    \begin{array}{cc}
      ((\mathfrak{a}_{5}\mathfrak{a}_{4}-I)(\mathfrak{a}_{3}\mathfrak{a}_{2}-I)-\mathfrak{a}_{5}\mathfrak{a}_{2})\mathfrak{a}_{1} -((\mathfrak{a}_{5}\mathfrak{a}_{4}-I)\mathfrak{a}_{3}-\mathfrak{a}_{5}) & -((\mathfrak{a}_{5}\mathfrak{a}_{4}-I)(\mathfrak{a}_{3}\mathfrak{a}_{2}-I)-\mathfrak{a}_{5}\mathfrak{a}_{2}) \\
      ((\mathfrak{a}_{4}\mathfrak{a}_{3}-I)\mathfrak{a}_{2}-\mathfrak{a}_{4})\mathfrak{a}_{1}-(\mathfrak{a}_{4}\mathfrak{a}_{3}-I) & -((\mathfrak{a}_{4}\mathfrak{a}_{3}-I)\mathfrak{a}_{2}-\mathfrak{a}_{4}) \\
    \end{array}
  \right).
\end{align*}

The equality $M(\mathfrak{a}_{1},\mathfrak{a}_{2},\mathfrak{a}_{3},\mathfrak{a}_{4},\mathfrak{a}_{5})=-Id$ results in four equations
\begin{enumerate}
  \item $(\mathfrak{a}_{4}\mathfrak{a}_{3}-I)\mathfrak{a}_{2}-\mathfrak{a}_{4}=I$,
  \item $(\mathfrak{a}_{5}\mathfrak{a}_{4}-I)\mathfrak{a}_{3}-\mathfrak{a}_{5}=I$,
  \item $(\mathfrak{a}_{5}\mathfrak{a}_{4}-I)(\mathfrak{a}_{3}\mathfrak{a}_{2}-I)-\mathfrak{a}_{5}\mathfrak{a}_{2}=O$,
  \item $(\mathfrak{a}_{4}\mathfrak{a}_{3}-I)(\mathfrak{a}_{2}\mathfrak{a}_{1}-I)-\mathfrak{a}_{4}\mathfrak{a}_{1}=O$.
\end{enumerate}

Now, from $3.$ and $2.$ follow that $((\mathfrak{a}_{5}\mathfrak{a}_{4}-I)\mathfrak{a}_{3}-\mathfrak{a}_{5})\mathfrak{a}_{2}=(\mathfrak{a}_{5}\mathfrak{a}_{4}-I)$, that is,
$\mathfrak{a}_{5}=(\mathfrak{a}_{2}+I)\mathfrak{a}_{4}^{-1}$. In the same way, using $4.$ and $1.$ we obtain $\mathfrak{a}_{4}=(\mathfrak{a}_{1}+I)\mathfrak{a}_{3}^{-1}$ which implies that $\mathfrak{a}_{5}=(\mathfrak{a}_{2}+I)\mathfrak{a}_{3}(\mathfrak{a}_{1}+I)^{-1}$. On the other hand, the first equation give us $\mathfrak{a}_{1}\mathfrak{a}_{2}=I+\mathfrak{a}_{4}=I+(\mathfrak{a}_{1}+I)\mathfrak{a}_{3}^{-1}$, thus $\mathfrak{a}_{2}=\mathfrak{a}_{1}^{-1}(I+(\mathfrak{a}_{1}+I)\mathfrak{a}_{3}^{-1})$. We claim that $[\mathfrak{a}_{1},\mathfrak{a}_{3}]=O$;
to see this we use the second equation. In fact, equation $2.$ implies that $\mathfrak{a}_{2}\mathfrak{a}_{3}=I+\mathfrak{a}_{5}$ hence
$\mathfrak{a}_{2}\mathfrak{a}_{3}(\mathfrak{a}_{1}+I)-\mathfrak{a}_{2}\mathfrak{a}_{3}-\mathfrak{a}_{3}=\mathfrak{a}_{1}+I$ thus $(\mathfrak{a}_{2}\mathfrak{a}_{3}-I)\mathfrak{a}_{1}=I+\mathfrak{a}_{3}$. Then $(\mathfrak{a}_{1}^{-1}(I+(\mathfrak{a}_{1}+I)\mathfrak{a}_{3}^{-1})\mathfrak{a}_{3}-I)\mathfrak{a}_{1}=\mathfrak{a}_{3}+I$, it shows that $(\mathfrak{a}_{1}^{-1}(\mathfrak{a}_{3}+\mathfrak{a}_{1}+I)-I)\mathfrak{a}_{1}=I+\mathfrak{a}_{3}$ and so
$\mathfrak{a}_{1}^{-1}\mathfrak{a}_{3}\mathfrak{a}_{1}=\mathfrak{a}_{3}$. This proves the statement. We can now write $\mathfrak{a}_{5}$ in a simple
form as function of $\mathfrak{a}_{1}$ and $\mathfrak{a}_{3}$, concretely $\mathfrak{a}_{5}=(\mathfrak{a}_{1}^{-1}+I)(\mathfrak{a}_{1}+I)^{-1}
(\mathfrak{a}_{3}+I)$. To prove the theorem it is enough take $\mathfrak{c}=\mathfrak{a}_{1}$ and $\mathfrak{d}=\mathfrak{a}_{3}$.
The theorem has been proved.
\end{proof}

The left matrix quiddity sequence $\mathfrak{G}=(\mathfrak{c},(I+\mathfrak{e})\mathfrak{c}^{-1}, (I+\mathfrak{c})\mathfrak{e}^{-1},\mathfrak{e},\mathfrak{c}^{-1}(I+\mathfrak{c}+\mathfrak{e})\mathfrak{e}^{-1})$ will be called the Gauss matrix
quiddity sequence (we recall that $[\mathfrak{c},\mathfrak{e}]=0$). To name $\mathfrak{G}$ in this way would be justified by the following proposition

\begin{proposition}Let us assume that $\mathfrak{A},\mathfrak{B}\in GL_{l}(\mathbb{K})$ such that $[\mathfrak{A},\mathfrak{B}]=0$ and suppose that $I+\mathfrak{A}$, $I+\mathfrak{B}$ and $I+\mathfrak{A+B}$ are invertible (for example, if
$\|\mathfrak{A}\|< \frac{1}{2},\,\|\mathfrak{B}\|< \frac{1}{2}$). Define the Gauss map
\begin{equation}\label{nf2}
\mathcal{G}\,:\,\,\,\,(\mathfrak{A},\mathfrak{B})\longrightarrow (\mathfrak{B},\mathfrak{A}^{-1}(\mathfrak{B}+I)),
\end{equation}
then, we have $\mathcal{G}^{5}(\mathfrak{A},\mathfrak{B})=(\mathfrak{A},\mathfrak{B})$.
\end{proposition}
\begin{proof}The proof is similar to the scalar case and so it will be omitted.
\end{proof}

The connection of this result with a possible theory of matrix-valued frieze patterns will be made clear later in the paper.

\subsection{The previous subsection revisited for $n$-antiperiodic sequences in the projective space $\mathbb{P}_{1}(M_{l}(\mathbb{R}))$}

We begin this subsection showing a general process of decomposition for the monodromy.
Let $G$ be a group and let $1_{G}$ be the identity of $G$. A left action of $G$ on a set $\Omega$ is a map $G\times \Omega\longrightarrow \Omega$,
that is, $(g,\omega)\longrightarrow g \cdot \omega$ such that $(hg)\cdot \omega=h\cdot (g\cdot \omega)$ for all $\omega\in \Omega$ and $h,g\in G$
and $1_{G}\cdot \omega=\omega$ for all $\omega\in \Omega$. We will always assume unless we say otherwise that the action is transitive, that is
$\Omega$ is a $G$-space.

Suppose that $G$ is a finite group and let $V$ be a vector space over $\mathbb{C}$. Denote by $GL(V)$ the linear group consisting of all
invertible linear maps $A: V\longrightarrow V$. A representation of $G$ over $V$ is an action $G\times V: (g,v)\longrightarrow \rho(g)v$
where $\rho(g)\in GL(V)$ for all $g\in G$. It means that $\rho : G \longrightarrow GL(V)$ such that $\rho(hg)=\rho(h)\rho(g)$ and
$\rho(1_{G})=1_{V}$ where $1_{V}$ is the identity element of $GL(V)$. The dimension of $V$ is denoted by $l$. We shall denote the representation
by the pair $(\rho,V)$.

From now on, during the section unless stated otherwise, we work with a fixed representation $(\rho,V)$.

We say that a sequence $\mathfrak{v}=\{v_{k}\}_{k\in\mathbb{Z}}\subset V$ is a twisted $n$-gon with respect to the representation $\rho$
if there exists $m_{\mathfrak{v}}\in G$ such that $v_{k+n}=m_{\mathfrak{v}}\cdot v_{k}=\rho(m_{\mathfrak{v}})v_{k}$ for every $k\in \mathbb{Z}$.
The element $m_{\mathfrak{v}}\in G$ is called the monodromy of $\mathfrak{v}$. A twisted $n$-gon $\mathfrak{v}=\{v_{k}\}_{k\in\mathbb{Z}}$
is called regular if for all $k$ the vectors $v_{k}, v_{k+1},\cdots, v_{k+l-1}$ constitute a basis of $V$. We denote by $\mathcal{P}_{n}$ the
set of all twisted $n$-gons. Observe that $G$ acts on $\mathcal{P}_{n}$ in the following form $g\cdot \mathfrak{v}=\{g\cdot v_{k}\}_{k\in\mathbb{Z}}$
where the monodromy of $g\cdot \mathfrak{v}$ is $gm_{\mathfrak{v}}g^{-1}$.

Now, we will consider recurrent equations of the form
\begin{equation}\label{fm1}
v_{k+1}=g_{k}\cdot v_{k}=\rho(g_{k})v_{k},
\end{equation}
where $\{g_{k}\}_{k\in \mathbb{Z}}\subset G$ is an $n$-periodic sequence and $v_{0}\in V$ is arbitrary. We say that a twisted $n$-gon is of
recurrent type if this satisfies a recurrent equation (\ref{fm1}) for some $n$-periodic sequence $\{g_{k}\}_{k\in \mathbb{Z}}\subset G$.
More concretely, we will say that $\mathfrak{v}$ is a twisted $n$-gon with respect to  $\{g_{k}\}_{k\in \mathbb{Z}}$. The set of all the pairs
$(\mathfrak{v}=\{v_{k}\}_{k\in\mathbb{Z}},\{g_{k}\}_{k\in \mathbb{Z}})$ where $\mathfrak{v}$  is a twisted $n$-gon of recurrent type with
respect to $\{g_{k}\}_{k\in \mathbb{Z}}$ is denoted by $\mathcal{RP}_{n}$.

We have

\begin{theorem}Suppose that all the solutions of (\ref{fm1}) satisfy the condition $v_{k+n}=m\cdot v_{k}=\rho(m)v_{k}$ for every $k\in\mathbb{Z}$
and some $m\in G$ fixed, then
\begin{equation}\label{fm2}
\rho(m)=\rho(g_{n-1})\rho(g_{n-2})\cdots \rho(g_{1})\rho(g_{0}).
\end{equation}

Reciprocally, let us assume that (\ref{fm2}) is true and $mg_{k}=g_{k}m$ for $k=0,1,\ldots ,n-2,n-1$. Then, every solution of (\ref{fm1})
satisfies the condition  $v_{k+n}=m\cdot v_{k}=\rho(m)v_{k}$ for every $k\in\mathbb{Z}$.
\end{theorem}
\begin{proof}Suppose that each solution of (\ref{fm1}) satisfies the condition $v_{k+n}=m\cdot v_{k}=\rho(m)v_{k}$ for every $k\in\mathbb{Z}$ and
some $m\in G$ fixed. Then,
\begin{equation*}
\rho(m)v_{0}=m\cdot v_{0}=v_{n}=g_{n-1}\cdot v_{n-1}=(g_{n-1}g_{n-2})\cdot v_{n-2}=(g_{n-1}g_{n-2}\cdots g_{1})\cdot v_{1}=
(g_{n-1}g_{n-2}\cdots g_{1}g_{0})\cdot v_{0},
\end{equation*}
hence
\begin{equation*}
\rho(m)v_{0}=\rho(g_{n-1}g_{n-2}\cdots g_{1}g_{0})v_{0}=\rho(g_{n-1})\rho(g_{n-2})\cdots \rho(g_{1})\rho(g_{0})v_{0},
\end{equation*}
now, $v_{0}$ is an arbitrary vector. It shows that (\ref{fm2}) holds.

Conversely, assume that (\ref{fm2}) is true, thus $v_{n}=\rho(m)v_{0}=m\cdot v_{0}$. It should be proven that $v_{k+n}=\rho(m)v_{k}=m\cdot v_{k}$ for very $k\in\mathbb{Z}$.
Recall that $\{g_{k}\}_{k\in\mathbb{Z}}$ is an $n$-periodic sequence. Hence,
\begin{equation*}
v_{n+1}=g_{n}\cdot v_{n}=g_{1}\cdot v_{n}=g_{1}\cdot (m\cdot v_{0})=(g_{1}m)\cdot v_{0}=(mg_{1})\cdot v_{0}=m\cdot(g_{1}\cdot v_{0})=m\cdot v_{1}.
\end{equation*}

On the other hand, from the recurrent equation (\ref{fm1}), $v_{0}=g_{n-1}\cdot v_{-1}$. So,
\begin{equation*}
g_{n-1}\cdot v_{n-1}=v_{n}=m\cdot v_{0}=m\cdot (g_{n-1}\cdot v_{-1})=g_{n-1}\cdot (m\cdot v_{-1}),
\end{equation*}
it shows that $v_{n-1}=m\cdot v_{-1}$. Therefore, the equality $v_{k+n}=\rho(m)v_{k}=m\cdot v_{k}$ is hold when $|k|\leq 1$. Suppose now that this equality holds for $|k|\leq r$ where $1< r$. From the induction
hypothesis, we obtain
\begin{equation*}
v_{(r+1)+n}=g_{r+n}\cdot v_{r+n}=g_{r+n}\cdot (m\cdot v_{r})=(g_{r+n}m)\cdot v_{r}=(g_{r}m)\cdot v_{r}=(mg_{r})\cdot v_{r}=m\cdot (g_{r}\cdot v_{r})=m\cdot g_{r+1},
\end{equation*}
and
\begin{equation*}
v_{-(r+1)+n}=g^{-1}_{-(r+1)+n}\cdot v_{-r+n}=g^{-1}_{-(r+1)}\cdot (m\cdot v_{-r})=(g^{-1}_{-(r+1)}m)\cdot v_{-r}=(mg^{-1}_{-(r+1)})\cdot v_{-r}=m\cdot(g^{-1}_{-(r+1)}\cdot v_{-r})=m\cdot v_{-(r+1)}.
\end{equation*}
\end{proof}

Two twisted $n$-gons, $\mathfrak{v}=\{v_{k}\}_{k\in\mathbb{Z}}\subset V$ and $\mathfrak{w}=\{w_{k}\}_{k\in\mathbb{Z}}\subset V$ are called
equivalent and we write $\mathfrak{v}\sim \mathfrak{w}$ if there is $h\in G$ such that $v_{k}=h\cdot w_{k}=\rho(h)w_{k}$ for all
$k\in \mathbb{Z}$. One can see that if $\mathfrak{v}\sim \mathfrak{w}$ then $m_{\mathfrak{w}}=h^{-1}m_{\mathfrak{v}}h$ where $m_{\mathfrak{v}}$
and $m_{\mathfrak{w}}$ are the monodromy of $\mathfrak{v}$ and $\mathfrak{w}$ respectively. Moreover, if $\mathfrak{v}$ is solution of a recurrent
equation (\ref{fm1}) with respect to an $n$-periodic sequence $\{g_{k}\}_{k\in Z}\subset G$ then the sequence $\mathfrak{w}$ satisfies the
recurrent equation $w_{k+1}=\hat{g}_{k}\cdot w_{k}=\rho(\hat{g}_{k})w_{k}$ in which $\hat{g}_{k}=h^{-1}g_{k}h$ for all $k\in\mathbb{Z}$, note
that the sequence $\{\hat{g}_{k}\}_{k\in\mathbb{Z}}$ is $n$-periodic.

The expression (\ref{fm2}) is called a monodromy decomposition and we say that $(\rho(g_{0}),\rho(g_{1}),\cdots,\rho(g_{n-2}),\rho(g_{n-1}))$ is a
$GL(V)$-quiddity sequence if $\rho(m)=-1_{V}=-\rho(1_{G})$.

Recall that we are working with a fixed representation $(V,\rho)$. It is well known that $G$ acts on $G^{n}$ by means of the diagonal action.
Define a map $\varrho :\mathcal{RP}_{n}\longrightarrow G^{n}$ of the following manner for all pair $(\mathfrak{v}=\{v_{k}\}_{k\in\mathbb{Z}},\{g_{k}\}_{k\in \mathbb{Z}})\in \mathcal{RP}_{n}$ we put  $\varrho(\mathfrak{v}=\{v_{k}\}_{k\in\mathbb{Z}})=(g_{1},\cdots ,g_{n})=(\varrho_{1}(\mathfrak{v}),
\cdots,\varrho_{i}(\mathfrak{v}),\cdots\varrho_{n}(\mathfrak{v}))$.

By way of motivation we will consider $n$-antiperiodic sequences in the projective space $\mathbb{P}_{1}(M_{l}(\mathbb{R}))$.
Briefly, we now review the $(k-1)$-dimensional right-projective spaces over the real $l\times l$ matrices \cite{schwarz-zaks}.
Real matrices of order $sl\times tl$ with $t,s\geq 1$ and $t\neq s$ or $t=s$ for $t,s\geq 2$ are denoted by calligraphic capital letters. One writes the $sl\times l$ matrix $\mathcal{Y}$
in block form: $\mathcal{Y}=\left(
                              \begin{array}{c}
                                Y_{1} \\
                                \vdots \\
                                Y_{s} \\
                              \end{array}
                            \right)
$, in which each $Y_{i}$ is an $l\times l$ matrix. $R_{0}(sl^{2})$ will be the set of real or complex
matrices $\mathcal{Y}$ of rank equal to $l$. $R_{0}(sl^{2})$ is a connected topological space and its topology is defined by means of any generalized matrix norm.

Two matrices $\mathcal{Y}=\left(
                            \begin{array}{c}
                              Y_{1} \\
                              \vdots \\
                              Y_{s} \\
                            \end{array}
                          \right)
$ and $\mathcal{U}=\left(
                     \begin{array}{c}
                       U_{1} \\
                       \vdots \\
                       U_{s} \\
                     \end{array}
                   \right)
$ of $R_{0}(sl^{2})$ are right- or column-equivalent if there exists an $l\times l$ invertible matrix $S$
such that
\begin{equation}\label{s1}
\mathcal{U}=\left( \begin{array}{c}
                       U_{1} \\
                       \vdots \\
                       U_{s} \\
                     \end{array}
                   \right)=\left(
                            \begin{array}{c}
                              Y_{1} \\
                              \vdots \\
                              Y_{s} \\
                            \end{array}
                          \right)S=\mathcal{Y}S,\,\,\,\,\,\,|S|\neq 0.
\end{equation}

This relation partitions $R_{0}(sl^{2})$ into equivalence classes of column-equivalent matrices. These equivalence classes are the points of the $(s-1)$-dimensional right-projective space over
the real or complex $l\times l$ matrices $\mathbb{P}_{(s-1)}(M_{l}(\mathbb{R}))$. The projective
mappings $\mathcal{C}$ of this left-projective space are given by means of constant invertible $sl\times sl$ matrices. $\mathcal{C}$ is written
in block form
\begin{equation}\label{s2}
\mathcal{C}=\left(
              \begin{array}{ccc}
                C_{11} & \cdots & C_{1s} \\
                \vdots &  & \vdots \\
                C_{s1} & \cdots & C_{ss} \\
              \end{array}
            \right),\,\,\,\,\,\,\,\,|\mathcal{C}|\neq 0,
\end{equation}
where each block $C_{ij}$, $i,j=1,\cdots,s$ is an $l\times l$ matrix. We denote the space of all projective mappings by $PGL_{sl}(M_{l}(\mathbb{R}))$ or simply $PGL_{sl}$.
For $\mathcal{C}$ fixed, one defines
\begin{equation}\label{s3}
\widetilde{\mathcal{Y}}=\left(
                          \begin{array}{c}
                            \widetilde{Y}_{1} \\
                            \vdots \\
                            \widetilde{Y}_{s} \\
                          \end{array}
                        \right)
=\mathcal{C}(\mathcal{Y})=\mathcal{C}\mathcal{Y}
=\left(
              \begin{array}{ccc}
                C_{11} & \cdots & C_{1s} \\
                \vdots &  & \vdots \\
                C_{s1} & \cdots & C_{ss} \\
              \end{array}
            \right)\left(
                     \begin{array}{c}
                       Y_{1} \\
                       \vdots \\
                       Y_{s} \\
                     \end{array}
                   \right)
            ,
\end{equation}
for all $\mathcal{Y}\in \mathbb{P}_{(s-1)}(M_{l}(\mathbb{R}))$, then $\mathcal{C}(\mathcal{Y})\in \mathbb{P}_{(s-1)}(M_{l}(\mathbb{R}))$.
If $\mathcal{U}=\mathcal{Y}S$ where $|S|\neq 0$, then $\widetilde{\mathcal{U}}=\mathcal{C}\mathcal{U}=\mathcal{C}\mathcal{Y}S=\widetilde{\mathcal{Y}}S$;
Hence, column-equivalent matrices have column-equivalent transformations. Thus, the transformation (\ref{s3}) induces a transformation of
$\mathbb{P}_{(s-1)}(M_{l}(\mathbb{R}))$ onto itself.

In this part, we will work in the projective space $\mathbb{P}_{1}(M_{l}(\mathbb{R}))$.

\begin{definition}A sequence $\left \{\Phi_{k}=\left(
                                  \begin{array}{c}
                                    Y^{k}_{1} \\
                                    Y^{k}_{2} \\
                                  \end{array}
                                \right) \right \}_{k\in\mathbb{Z}}\subset R_{0}(2l^{2})$ is said to be $n$-antiperiodic if $\Phi_{k+n}=-\Phi_{k}$ for all $k\in \mathbb{Z}$. An
                                $n$-antiperiodic sequence $\{\Phi_{k}\}_{k\in\mathbb{Z}}$ of $R_{0}(2l^{2})$ is called regular if $|\Phi_{k}\,\,\Phi_{k+1}|\neq 0$ for any $k\in\mathbb{Z}$.
\end{definition}

We can introduce an equivalence relation $\sim$ in the set $AP_{n}$ of $n$-antiperiodic sequences of matrices belong to $R_{0}(2l^{2})$. We say that two sequences $\{\Phi_{k}\}_{k\in\mathbb{Z}}, \{\Lambda_{k}\}_{k\in\mathbb{Z}}\in AP_{n}$ are related if there exists $G\in PGL_{2l}$ such that $\Lambda_{k}=G\Phi_{k}$ for every $k\in\mathbb{Z}$. Below, the notation $CAP_{n}=AP_{n}\diagup\sim$ will be used.

The following remark is trivial

\begin{remark}Suppose that $\mathcal{L}\in CAP_{n}$ such that there is $\{\Phi_{k}\}_{k\in\mathbb{Z}}\in\mathcal{L}$ regular. Then, all sequence belongs to $\mathcal{L}$ is regular. In this case, we say that the class $\mathcal{L}$ is regular, otherwise $\mathcal{L}$ is called non-regular.
\end{remark}

Let $\mathcal{L}$ be a regular class and $\{\Phi_{k}\}_{k\in\mathbb{Z}}\in\mathcal{L}$, then for all $k\in \mathbb{Z}$ there are two $n$-periodic sequences of matrices of order $l$, say $\{\mathfrak{l}_{k}\}_{k\in \mathbb{Z}}$ and  $\{\mathfrak{s}_{k}\}_{k\in \mathbb{Z}}$ such that
\begin{equation}\label{s4}
\Phi_{k+1}=\Phi_{k}\mathfrak{l}_{k}+\Phi_{k-1}\mathfrak{s}_{k},
\end{equation}
for all $k\in\mathbb{Z}$, thus for any $k\in\mathbb{Z}$
\begin{equation}\label{s5}
(\Phi_{k}\,\,\Phi_{k+1})=(\Phi_{k-1}\,\,\Phi_{k})\left(
                                                   \begin{array}{cc}
                                                     O & \mathfrak{s}_{k} \\
                                                     I & \mathfrak{l}_{k} \\
                                                   \end{array}
                                                 \right).
\end{equation}

From (\ref{s5}) and the Schur determinant lemma follow that $|\mathfrak{s}_{k}|\neq 0$ for $k\in\mathbb{Z}$. Define
\begin{equation*}
N(\mathfrak{l}_{1},\cdots,\mathfrak{l}_{n};\mathfrak{s}_{1},\cdots,\mathfrak{s}_{n})=
\left(
\begin{array}{cc}
 O & \mathfrak{s}_{1} \\
 I & \mathfrak{l}_{1} \\
 \end{array}
 \right)\cdots \cdots\left(
\begin{array}{cc}
 O & \mathfrak{s}_{n} \\
 I & \mathfrak{l}_{n} \\
 \end{array}
 \right)=N(\mathfrak{l}_{1};\mathfrak{s}_{1})\cdots\cdots N(\mathfrak{l}_{n};\mathfrak{s}_{n}),
\end{equation*}
observe that $|N(\mathfrak{l}_{k};\mathfrak{s}_{k})|=|\mathfrak{s}_{k}|$ where $k=1,\cdots,n$, hence $|N(\mathfrak{l}_{1},\cdots,\mathfrak{l}_{n};\mathfrak{s}_{1},\cdots,\mathfrak{s}_{n})|=|\mathfrak{s}_{1}|\cdots |\mathfrak{s}_{n}|$. Since $\{\Phi_{k}\}_{k\in\mathbb{Z}}$ is regular and
$n$-antiperiodic, we conclude that $|N(\mathfrak{l}_{1},\cdots,\mathfrak{l}_{n};\mathfrak{s}_{1},\cdots,\mathfrak{s}_{n})|=|s\mathfrak{}_{1}|\cdots |\mathfrak{s}_{n}|=1$.

It is not difficult to see that all regular solution $\{\Phi_{k}\}_{k\in\mathbb{Z}}\subset \mathbb{P}_{1}(M_{l}(\mathbb{R}))$ of the equation (\ref{s4}) with coefficients $\{\mathfrak{l}_{k}\}$ and $\{\mathfrak{s}_{k}\}$
which are $n$-periodic sequences is $n$-antiperiodic if and only if
\begin{equation*}
\left(
\begin{array}{cc}
 O & \mathfrak{s}_{1} \\
 I & \mathfrak{l}_{1} \\
 \end{array}
 \right)\cdots \cdots\left(
\begin{array}{cc}
 O & \mathfrak{s}_{n} \\
 I & \mathfrak{l}_{n} \\
 \end{array}
 \right)=-Id.
\end{equation*}

We arrive at the following definition
\begin{definition}A bi-vector $(\overline{\mathfrak{l}};\overline{\mathfrak{s}})_{n}=(\mathfrak{l}_{1},\cdots,\mathfrak{l}_{n};\mathfrak{s}_{1},\cdots,\mathfrak{s}_{n})$ is called a \textbf{right matrix quiddity bi-sequence} if
\begin{equation}\label{s6}
N((\overline{\mathfrak{l}};\overline{\mathfrak{s}})_{n})
=N(\mathfrak{l}_{1},\cdots,\mathfrak{l}_{n};\mathfrak{s}_{1},\cdots,\mathfrak{s}_{n})=N(\mathfrak{l}_{1};\mathfrak{s}_{1})\cdots\cdots N(\mathfrak{l}_{n};\mathfrak{s}_{n})=-\left(
                                                                  \begin{array}{cc}
                                                                    I & O \\
                                                                    O & I \\
                                                                  \end{array}
                                                                \right)=-Id,
\end{equation}
in this case, $n$ is called the length of $(\overline{\mathfrak{l}};\overline{\mathfrak{s}})_{n}$.
\end{definition}

It follows from (\ref{s6}) that $JN^{\ast}(\mathfrak{l}_{1},\cdots,\mathfrak{l}_{n};\mathfrak{s}_{1},\cdots,\mathfrak{s}_{n})J=(JN^{\ast}(\mathfrak{l}_{n};\mathfrak{s}_{n})J)
\cdots\cdots (JN^{\ast}(\mathfrak{l}_{1};\mathfrak{s}_{1})J)=-Id$, where as before
$J=\left(
\begin{array}{cc}
O & I \\
I & O \\
\end{array}
\right)$, thus
\begin{equation*}
\left(
\begin{array}{cc}
 \mathfrak{l}^{\ast}_{n} & \mathfrak{s}^{\ast}_{n} \\
 I & O \\
 \end{array}
 \right)\cdots\cdots \left(
\begin{array}{cc}
 \mathfrak{l}^{\ast}_{1} & \mathfrak{s}^{\ast}_{1} \\
 I & O \\
 \end{array}
 \right)=-\left(
 \begin{array}{cc}
 I & O \\
 O & I \\
 \end{array}
 \right)=-Id.
\end{equation*}

This leads to the next definition
\begin{definition}A bi-vector $(\overline{\mathfrak{p}};\overline{\mathfrak{q}})_{n}=(\mathfrak{p}_{1},\cdots,\mathfrak{p}_{n};\mathfrak{q}_{1},\cdots,\mathfrak{q}_{n})$ is called a
\textbf{left matrix quiddity bi-sequence} of length $n$ if
\begin{equation}\label{s7}
M((\overline{\mathfrak{p}};\overline{\mathfrak{q}})_{n})=M(\mathfrak{p}_{1},\cdots,\mathfrak{p}_{n};\mathfrak{q}_{1},\cdots,\mathfrak{q}_{n})
=\left(
\begin{array}{cc}
 \mathfrak{p}_{n} & \mathfrak{q}_{n} \\
 I & O \\
 \end{array}
 \right)\cdots\cdots \left(
\begin{array}{cc}
 \mathfrak{p}_{1} & \mathfrak{q}_{1} \\
 I & O \\
 \end{array}
 \right)=M(\mathfrak{p}_{n};\mathfrak{q}_{n})\cdots\cdots M(\mathfrak{p}_{1};\mathfrak{q}_{1})=-Id.
\end{equation}
\end{definition}

It is clear that there is a one-to-one correspondence between the set $\mathcal{RMQB}_{n}$ of all the right matrix quiddity bi-sequences of length
$n$ and the set $\mathcal{LMQB}_{n}$ of all the left matrix quiddity bi-sequences. So, in this part we work in the class of left matrix quiddity bi-sequences of any length $\mathcal{LMQB}=\sqcup_{n}\mathcal{LMQB}_{n}$.

\begin{example}Let us study the equation
\begin{equation}\label{s8}
M((\overline{\mathfrak{p}};\overline{\mathfrak{q}})_{3})=\left(
  \begin{array}{cc}
    \mathfrak{p}_{3} & \mathfrak{q}_{3} \\
    I & O \\
  \end{array}
\right)\left(
  \begin{array}{cc}
    \mathfrak{p}_{2} & \mathfrak{q}_{2} \\
    I & O \\
  \end{array}
\right)\left(
  \begin{array}{cc}
    \mathfrak{p}_{1} & \mathfrak{q}_{1} \\
    I & O \\
  \end{array}
\right)=-\left(
          \begin{array}{cc}
            I & O \\
            O & I \\
          \end{array}
        \right),
\end{equation}now
\begin{equation*}
M((\overline{\mathfrak{p}};\overline{\mathfrak{q}})_{3})=\left(
                                      \begin{array}{cc}
                                        \mathfrak{p}_{3}\mathfrak{p}_{2}+\mathfrak{q}_{3} & \mathfrak{p}_{3}\mathfrak{q}_{2} \\
                                        \mathfrak{p}_{2} & \mathfrak{q}_{2} \\
                                      \end{array}
                                    \right)\left(
  \begin{array}{cc}
    \mathfrak{p}_{1} & \mathfrak{q}_{1} \\
    I & O \\
  \end{array}
\right)=\left(
          \begin{array}{cc}
           (\mathfrak{p}_{3}\mathfrak{p}_{2}+\mathfrak{q}_{3})\mathfrak{p}_{1}+\mathfrak{p}_{3}\mathfrak{q}_{2}  & (\mathfrak{p}_{3}\mathfrak{p}_{2}+\mathfrak{q}_{3})\mathfrak{q}_{1} \\
            \mathfrak{p}_{2}\mathfrak{p}_{1}+\mathfrak{q}_{2} & \mathfrak{p}_{2}\mathfrak{q}_{1} \\
          \end{array}
        \right),
\end{equation*}
hence $\mathfrak{\mathfrak{p}}_{1}=\mathfrak{q}_{1}\mathfrak{q}_{2}$, $\mathfrak{p}_{2}=-\mathfrak{q}^{-1}_{1}$ and $\mathfrak{p}_{3}=\mathfrak{q}_{3}\mathfrak{q}_{1}$ such that $\mathfrak{q}_{3}\mathfrak{q}_{1}\mathfrak{q}_{2}=-I$. It follows that
$(\mathfrak{q}_{1}\mathfrak{q}_{2}, -\mathfrak{q}_{1}^{-1}, \mathfrak{q}_{3}\mathfrak{q}_{1}; \mathfrak{q}_{1},\mathfrak{q}_{2},-\mathfrak{q}_{2}^{-1}\mathfrak{q}_{1}^{-1})$
is a left matrix quiddity bi-sequence of length $3$.
\end{example}

\begin{proposition}Let $(\mathfrak{p}_{1},\cdots,\mathfrak{p}_{n};\mathfrak{q}_{1},\cdots,\mathfrak{q}_{n})$ be an element of $\mathcal{LMQB}_{n}$ then
\begin{equation}\label{s8}
(\mathfrak{p}_{1},\cdots,\mathfrak{p}_{k-1} ,\mathfrak{p}_{k}+I,I,\mathfrak{p}_{k+1}-\mathfrak{q}_{k+1},\mathfrak{p}_{k+2},\cdots,\mathfrak{p}_{n};
\mathfrak{q}_{1},\cdots,\mathfrak{q}_{k-1},\mathfrak{q}_{k},-I,\mathfrak{q}_{k+1},\mathfrak{q}_{k+2},\cdots, \mathfrak{q}_{n})\in \mathcal{LMQB}_{n+1},
\end{equation}
for $k=1,\cdots, n-1$.
\end{proposition}

\begin{proof}In order to prove the proposition it is sufficient to observe that for $k=1,\cdots,n-1$
\begin{align*}
M(\mathfrak{p}_{k+1}-\mathfrak{q}_{k+1};\mathfrak{q}_{k+1})M(I;-I)M(\mathfrak{p}_{k}+I;\mathfrak{q}_{k})&=\left(
  \begin{array}{cc}
    \mathfrak{p}_{k+1}-\mathfrak{q}_{k+1} & \mathfrak{q}_{k+1} \\
    I & O \\
  \end{array}
\right)\left(
         \begin{array}{cc}
           I & -I \\
           I & O \\
         \end{array}
       \right)\left(
                \begin{array}{cc}
                  \mathfrak{p}_{k}+I & \mathfrak{q}_{k} \\
                  I & O \\
                \end{array}
              \right) \\
&=\left(
    \begin{array}{cc}
      \mathfrak{p}_{k+1} & \mathfrak{q}_{k+1}-\mathfrak{p}_{k+1} \\
      I & -I \\
    \end{array}
  \right)\left(
                \begin{array}{cc}
                  \mathfrak{p}_{k}+I & \mathfrak{q}_{k} \\
                  I & O \\
                \end{array}
              \right) \\
&=\left(
    \begin{array}{cc}
       \mathfrak{p}_{k+1}\mathfrak{p}_{k} + \mathfrak{q}_{k+1} & \mathfrak{p}_{k+1}\mathfrak{q}_{k} \\
       \mathfrak{p}_{k} & \mathfrak{q}_{k} \\
    \end{array}
  \right)=M(\mathfrak{p}_{k+1};\mathfrak{q}_{k+1})M(\mathfrak{p}_{k};\mathfrak{q}_{k}).
\end{align*}
\end{proof}

We also have
\begin{theorem}Suppose that $(\overline{\mathfrak{p}};\overline{\mathfrak{q}})_{n}=(\mathfrak{p}_{1},\cdots,\mathfrak{p}_{n};\mathfrak{q}_{1},\cdots,\mathfrak{q}_{n})\in \mathcal{LMQB}_{n}$ and $(\overline{\mathfrak{l}};\overline{\mathfrak{s}})_{m}=(\mathfrak{l}_{1},\cdots,\mathfrak{l}_{m};\mathfrak{s}_{1},\cdots,\mathfrak{s}_{m})\in \mathcal{LMQB}_{m}$. Introduce the following products
\begin{align}\label{s8}
&(\overline{\mathfrak{p}};\overline{\mathfrak{q}})_{n}\circ_{k}(\overline{\mathfrak{l}};\overline{\mathfrak{s}})_{m}=  \\
&(\mathfrak{p}_{1},\cdots,\mathfrak{p}_{k-1},\mathfrak{l}_{1}-\mathfrak{s}_{1}(\mathfrak{p}_{k}+I),\mathfrak{l}_{2},\cdots,\mathfrak{l}_{m-1},
\mathfrak{l}_{m}+I,
\mathfrak{p}_{k+1}-\mathfrak{q}_{k+1},\mathfrak{p}_{k+2},\cdots,\mathfrak{p}_{n};\mathfrak{q}_{1},\cdots,\mathfrak{q}_{k-1},-\mathfrak{s}_{1}
\mathfrak{q}_{k},\mathfrak{s}_{2},
\cdots,\mathfrak{s}_{m},\mathfrak{q}_{k+1},\cdots,\mathfrak{q}_{n}), \nonumber
\end{align}
for $k=1,\cdots,n-1$ (of course $|\mathfrak{s}_{1}\mathfrak{q}_{k}|\neq 0$), and
\begin{equation}\label{s9}
(\overline{\mathfrak{p}};\overline{\mathfrak{q}})_{n}\circ_{n}(\overline{\mathfrak{l}};\overline{\mathfrak{s}})_{m}=
(\mathfrak{p}_{1}-\mathfrak{q}_{1},\mathfrak{p}_{2},\cdots,\mathfrak{p}_{n-1},\mathfrak{l}_{1}-\mathfrak{s}_{1}(\mathfrak{p}_{n}+I),\mathfrak{l}_{2},
\cdots,\mathfrak{l}_{m-1},
\mathfrak{l}_{m}+I;\mathfrak{q}_{1},\cdots,\mathfrak{q}_{n-1},-\mathfrak{s}_{1}
\mathfrak{q}_{n},\mathfrak{s}_{2},
\cdots,\mathfrak{s}_{m}),
\end{equation}

Then $(\overline{\mathfrak{p}};\overline{\mathfrak{q}})_{n}\circ_{k}(\overline{\mathfrak{l}};\overline{\mathfrak{s}})_{m}\in \mathcal{LMQB}_{n+m-1}$ for $k=1,\cdots,n$.
\end{theorem}

\begin{proof}In fact, let $1\leq k \leq n-1$ then
\begin{align*}
(\overline{\mathfrak{p}};\overline{\mathfrak{q}})_{n}\circ_{k}(\overline{\mathfrak{l}};\overline{\mathfrak{s}})_{m}&=M(\mathfrak{p}_{n};\mathfrak{q}_{n})
\cdots M(\mathfrak{p}_{k+2};\mathfrak{q}_{k+2})
M(\mathfrak{p}_{k+1}-\mathfrak{q}_{k+1};\mathfrak{q}_{k+1})M(\mathfrak{l}_{m}+I;\mathfrak{s}_{m})M(\mathfrak{l}_{m-1};\mathfrak{s}_{m-1})\cdots \\
&\,\,\,\,\,\,\,M(\mathfrak{l}_{2};\mathfrak{s}_{2})M(\mathfrak{l}_{1}-\mathfrak{s}_{1}(\mathfrak{p}_{k}+I);-\mathfrak{s}_{1}\mathfrak{q}_{k})
M(\mathfrak{p}_{k-1};\mathfrak{q}_{k-1})\cdots
M(\mathfrak{p}_{1};\mathfrak{q}_{1}),
\end{align*}
now let us note that
\begin{align}
M(\mathfrak{l}_{1}-\mathfrak{s}_{1}(\mathfrak{p}_{k}+I);-\mathfrak{s}_{1}\mathfrak{q}_{k})&=\left(
                                                                                             \begin{array}{cc}
                                                                                               \mathfrak{l}_{1}-\mathfrak{s}_{1}(\mathfrak{p}_{k}+I) & -\mathfrak{s}_{1}\mathfrak{q}_{k} \\
                                                                                                I & O \\
                                                                                             \end{array}
                                                                                           \right)  \nonumber \\
&=\left(
   \begin{array}{cc}
     \mathfrak{l}_{1} & \mathfrak{s}_{1} \\
     I & O \\
   \end{array}
 \right)\left(
          \begin{array}{cc}
            O & I \\
            -I & -I \\
          \end{array}
        \right)\left(
                 \begin{array}{cc}
                   \mathfrak{p}_{k} & \mathfrak{q}_{k} \\
                   I & O \\
                 \end{array}
               \right)=M(\mathfrak{l}_{1};\mathfrak{s}_{1})\left(
                                                             \begin{array}{cc}
                                                               O & I \\
                                                               -I & -I \\
                                                             \end{array}
                                                           \right)M(\mathfrak{p}_{k};\mathfrak{s}_{k}), \nonumber
\end{align}
and
\begin{align*}
M(\mathfrak{p}_{k+1}-\mathfrak{q}_{k+1};\mathfrak{q}_{k+1})M(\mathfrak{l}_{m}+I;\mathfrak{s}_{m})&=\left(
                                                                                                   \begin{array}{cc}
                                                                                                     \mathfrak{p}_{k+1}-\mathfrak{q}_{k+1} & \mathfrak{q}_{k+1} \\
                                                                                                      I & O \\
                                                                                                   \end{array}
                                                                                                 \right)\left(
                                                                                                          \begin{array}{cc}
                                                                                                            \mathfrak{l}_{m}+I & \mathfrak{s}_{m} \\
                                                                                                            I & O \\
                                                                                                          \end{array}
                                                                                                        \right) \\
=&\left(
    \begin{array}{cc}
      \mathfrak{p}_{k+1} & \mathfrak{q}_{k+1} \\
      I & O \\
    \end{array}
  \right)\left(
           \begin{array}{cc}
             I & I \\
             -I & O \\
           \end{array}
         \right)\left(
                  \begin{array}{cc}
                    \mathfrak{l}_{m} & \mathfrak{s}_{m} \\
                    I & O \\
                  \end{array}
                \right)=M(\mathfrak{p}_{k+1};\mathfrak{q}_{k+1})\left(
                                                                  \begin{array}{cc}
                                                                     I & I \\
                                                                    -I & O \\
                                                                  \end{array}
                                                                \right)M(\mathfrak{l}_{m};\mathfrak{s}_{m}),
\end{align*}
then combining the hypotheses of the theorem and the last three equalities, it turns out that
\begin{align*}
(\overline{\mathfrak{p}};\overline{\mathfrak{q}})_{n}\circ_{k}(\overline{\mathfrak{l}};\overline{\mathfrak{s}})_{m}&=M(\mathfrak{p}_{n};\mathfrak{q}_{n})
\cdots M(\mathfrak{p}_{k+2};\mathfrak{q}_{k+2})
M(\mathfrak{p}_{k+1};\mathfrak{q}_{k+1})\left(
                                                                  \begin{array}{cc}
                                                                     I & I \\
                                                                    -I & O \\
                                                                  \end{array}
                                                                \right)M(\mathfrak{l}_{m};\mathfrak{s}_{m})M(\mathfrak{l}_{m-1};\mathfrak{s}_{m-1})\cdots \\
&\,\,\,\,\,\,\,M(\mathfrak{l}_{2};\mathfrak{s}_{2})M(\mathfrak{l}_{1};\mathfrak{s}_{1})\left(
                                                             \begin{array}{cc}
                                                               O & I \\
                                                               -I & -I \\
                                                             \end{array}
                                                           \right)M(\mathfrak{p}_{k};\mathfrak{s}_{k})
M(\mathfrak{p}_{k-1};\mathfrak{q}_{k-1})\cdots
M(\mathfrak{p}_{1};\mathfrak{q}_{1}) \\
&=M(\mathfrak{p}_{n};\mathfrak{q}_{n})
\cdots
M(\mathfrak{p}_{k+1};\mathfrak{q}_{k+1})\left(
                                                                  \begin{array}{cc}
                                                                     I & I \\
                                                                    -I & O \\
                                                                  \end{array}
                                                                \right)\left(
                                                                         \begin{array}{cc}
                                                                           -I & O \\
                                                                           O & -I \\
                                                                         \end{array}
                                                                       \right)\left(
                                                             \begin{array}{cc}
                                                               O & I \\
                                                               -I & -I \\
                                                             \end{array}
                                                           \right)M(\mathfrak{p}_{k};\mathfrak{s}_{k})
\cdots
M(\mathfrak{p}_{1};\mathfrak{q}_{1}) \\
&=M(\mathfrak{p}_{n};\mathfrak{q}_{n})
\cdots
M(\mathfrak{p}_{k+1};\mathfrak{q}_{k+1})M(\mathfrak{p}_{k};\mathfrak{s}_{k})
\cdots
M(\mathfrak{p}_{1};\mathfrak{q}_{1})=-Id .
\end{align*}

Here, we have taken into account that $(\overline{\mathfrak{p}};\overline{\mathfrak{q}})_{n}\in \mathcal{LMQB}_{n}$. Now,
\begin{align*}
(\overline{\mathfrak{p}};\overline{\mathfrak{q}})_{n}\circ_{n}(\overline{\mathfrak{l}};\overline{\mathfrak{s}})_{m}&=M(\mathfrak{l}_{m}+I;\mathfrak{s}_{m})
M(\mathfrak{l}_{m-1};\mathfrak{s}_{m-1})\cdots M(\mathfrak{l}_{2};\mathfrak{s}_{2})M(\mathfrak{l}_{1}-\mathfrak{s}_{1}(\mathfrak{p}_{n}+I);-\mathfrak{s}_{1}
\mathfrak{q}_{n})M(\mathfrak{p}_{n-1};\mathfrak{q}_{n-1}) \\
&\,\,\,\,\,\,\cdots M(\mathfrak{p}_{2};\mathfrak{q}_{2})M(\mathfrak{p}_{1}-\mathfrak{q}_{1};\mathfrak{q}_{1}) \\
&=\left(
    \begin{array}{cc}
      I & I \\
      O & I \\
    \end{array}
  \right)M(\mathfrak{l}_{m};\mathfrak{s}_{m})M(\mathfrak{l}_{m-1};\mathfrak{s}_{m-1})\cdots M(\mathfrak{l}_{2};\mathfrak{s}_{2})M(\mathfrak{l}_{1};\mathfrak{s}_{1})\left(
                                                             \begin{array}{cc}
                                                               O & I \\
                                                               -I & -I \\
                                                             \end{array}
                                                           \right)M(\mathfrak{p}_{n};\mathfrak{s}_{n}) \\
&\,\,\,\,\,\,M(\mathfrak{p}_{n-1};\mathfrak{q}_{n-1})\cdots M(\mathfrak{p}_{2};\mathfrak{q}_{2})M(\mathfrak{p}_{1};\mathfrak{q}_{1})\left(
                                                                                                                                      \begin{array}{cc}
                                                                                                                                        I & O \\
                                                                                                                                        -I & I \\
                                                                                                                                      \end{array}
                                                                                                                                    \right) \\
&=\left(
    \begin{array}{cc}
      I & I \\
      O & I \\
    \end{array}
  \right)\left(
           \begin{array}{cc}
             -I & O \\
             O & -I \\
           \end{array}
         \right)\left(
                                                             \begin{array}{cc}
                                                               O & I \\
                                                               -I & -I \\
                                                             \end{array}
                                                           \right)\left(
                                                                    \begin{array}{cc}
                                                                      -I & O \\
                                                                      O & -I \\
                                                                    \end{array}
                                                                  \right)\left(
                                                                  \begin{array}{cc}
                                                                  I & O \\
                                                                   -I & I \\
                                                                   \end{array}
                                                                   \right)=-Id.
\end{align*}

We finish the proof of the theorem.
\end{proof}

\begin{theorem}Let $(\overline{\mathfrak{p}};\overline{\mathfrak{q}})_{n}=(\mathfrak{p}_{1},\cdots,\mathfrak{p}_{n};\mathfrak{q}_{1},\cdots,\mathfrak{q}_{n})\in \mathcal{LMQB}_{n}$ and $(\overline{\mathfrak{l}};\overline{\mathfrak{s}})_{m}=(\mathfrak{l}_{1},\cdots,\mathfrak{l}_{m};\mathfrak{s}_{1},\cdots,\mathfrak{s}_{m})\in \mathcal{LMQB}_{m}$. Define
\begin{align}\label{s8}
&(\overline{\mathfrak{p}};\overline{\mathfrak{q}})_{n}\bullet_{k}(\overline{\mathfrak{l}};\overline{\mathfrak{s}})_{m}=  \\
&(\mathfrak{p}_{1},\cdots,\mathfrak{p}_{k-1}+I,\mathfrak{l}_{2}-\mathfrak{s}_{2},\mathfrak{l}_{3},\cdots,\mathfrak{l}_{m},
(\mathfrak{p}_{k}-\mathfrak{q}_{k})-\mathfrak{q}_{k}\mathfrak{l}_{1},\mathfrak{p}_{k+1},\cdots,\mathfrak{p}_{n};\mathfrak{q}_{1},\cdots,
\mathfrak{q}_{k-1},\mathfrak{s}_{2},
\cdots,\mathfrak{s}_{m},-\mathfrak{q}_{k}\mathfrak{s}_{1},\mathfrak{q}_{k+1},\cdots,\mathfrak{q}_{n}), \nonumber
\end{align}
for $k=2,\cdots,n$ (note that $|\mathfrak{q}_{k}\mathfrak{s}_{1}|\neq 0$), and
\begin{equation}\label{s9}
(\overline{\mathfrak{p}};\overline{\mathfrak{q}})_{n}\bullet_{1}(\overline{\mathfrak{l}};\overline{\mathfrak{s}})_{m}=
((\mathfrak{p}_{1}-\mathfrak{q}_{1})-\mathfrak{q}_{1}\mathfrak{l}_{1},\mathfrak{p}_{2},\cdots,\mathfrak{p}_{n-1},\mathfrak{p}_{n}+I,
\mathfrak{l}_{2}-\mathfrak{s}_{2},\mathfrak{l}_{3}
\cdots,\mathfrak{l}_{m};-\mathfrak{q}_{1}\mathfrak{s}_{1},\mathfrak{q}_{2},\cdots,\mathfrak{q}_{n},\mathfrak{s}_{2},\cdots,\mathfrak{s}_{m}),
\end{equation}

Then $(\overline{\mathfrak{p}};\overline{\mathfrak{q}})_{n}\bullet_{k}(\overline{\mathfrak{l}};\overline{\mathfrak{s}})_{m}\in \mathcal{LMQB}_{n+m-1}$ for $k=1,\cdots,n$.
\end{theorem}
\begin{proof}The proof of this theorem is similar to that of the previous theorem, hence it will be omitted.
\end{proof}

\subsection{The moving frame theory revisited in relation with the two previous subsections}

An alternative way of studying the fixed points under the action of the group of rigid movements on Euclidean space is through the notion of moving frames which leads to the construction of invariant functions, see \cite{Clelland}. For our purposes, in this part, we will be interested
in the theory of discrete moving frames which was founded in \cite{Mansfield}.

Let $\mathfrak{M}_{n}=(M_{l}(\mathbb{C}))^{n}$ equipped with the following metric
\begin{equation}\label{mvmf1}
\|\overline{X}\|=\sqrt{\sum_{k=1}^{n}\|\mathfrak{x}_{k}\|^{2}},
\end{equation}
for all $\overline{X}=(\mathfrak{x}_{1},\cdots,\mathfrak{x}_{n})\in \mathfrak{M}_{n}$, where $\|\mathfrak{x}_{k}\|=\sup_{r\neq 0\in \mathbb{C}^{l}}\frac{\|\mathfrak{x}_{k}r\|}{\|r\|}$ and $k=1,\ldots,n$.
Let us suppose that a group $G$ acts from the left on $M_{l}(\mathbb{C})$
and so on  $\mathfrak{M}_{n}$ for means of the product action, that is $g\cdot \overline{X}=(g\cdot \mathfrak{x}_{1},\cdots,g\cdot \mathfrak{x}_{n})$.
We will assume that the action of $G$ on $M_{l}(\mathbb{C})$ is free, that is, for all $\mathfrak{r}\in M_{l}(\mathbb{C})$ the corresponding isotropy
group $G_{\mathfrak{r}}=\{g\in G| g\cdot \mathfrak{r}=\mathfrak{r}\}$ is equal to $\{e\}$, where $e$ is the identity of $G$.

\begin{definition}A $M_{l}(\mathbb{C})$-valued function $I: \mathfrak{M}_{n}\longrightarrow M_{l}(\mathbb{C})$ is called $G$-invariant if
$I(g\cdot \overline{X})=I(\overline{X})$ for all $\overline{X}\in \mathfrak{M}_{n}$ and every $g\in G$. Note that an invariant function in our
context is a matrix-valued function.
\end{definition}

It is easy to see that the set of all $G$-invariant $M_{l}(\mathbb{C})$-valued functions is an algebra on $\mathbb{C}$. In fact, denote this set by
$I(G;M_{l}(\mathbb{C}))$. Let $I$ and $\widetilde{I}$ be two $G$-invariant $M_{l}(\mathbb{C})$-valued functions and define
$(I+\widetilde{I})(\overline{X})=I(\overline{X})+\widetilde{I}(\overline{X})$ for all $\overline{X}\in \mathfrak{M}_{n}$ then
$I+\widetilde{I}\in I(G;M_{l}(\mathbb{C}))$ because for every $g\in G$, we obtain $(I+\widetilde{I})(g\cdot\overline{X})=I(g\cdot\overline{X})+\widetilde{I}(g\cdot\overline{X})=
I(\overline{X})+\widetilde{I}(\overline{X})
=(I+\widetilde{I})(\overline{X})$. For $\alpha\in\mathbb{C}$, in the same way, we have $(\alpha I)(\overline{X})=\alpha I(\overline{X})\in I(G;M_{l}(\mathbb{C}))$ and $(I\cdot\widetilde{I})(\overline{X})=I(\overline{X})\widetilde{I}(\overline{X})\in I(G;M_{l}(\mathbb{C}))$. Now, since
$M_{l}(\mathbb{C})$ is a complex algebra it follows that $I(G;M_{l}(\mathbb{C}))$ is a complex algebra.

\begin{definition}A right moving frame $\rho$ is a function $\rho : \Omega\subset\mathfrak{M}_{n}\longrightarrow G$ such that $\rho(g\cdot \overline{X})=
\rho(\overline{X})g^{-1}$ for every $\overline{X}\in \Omega$ and any $g\in G$. On the other hand, $\rho : \Omega\subset\mathfrak{M}_{n}\longrightarrow G$
is a left moving frame if $\rho(g\cdot \overline{X})=g \rho(\overline{X})$  for $\overline{X}\in \Omega$ and $g\in G$ arbitraries. The set $\Omega$ is called of domain of $\rho$.
\end{definition}

Suppose that $\rho$ is a fixed right moving frame then for $k=1,\ldots,n$, the function $I_{k}(\overline{X})=\rho(\overline{X})\cdot \mathfrak{x}_{k}$
from $\mathfrak{M}_{n}$ into $M_{l}(\mathbb{C})$ is $G$-invariant, in other words $I_{k}\in I(G;M_{l}(\mathbb{C}))$. It is well know, observe that
$I_{k}(g\cdot\overline{X})=\rho(g\cdot\overline{X})\cdot (g\cdot\mathfrak{x}_{k})=(\rho(\overline{X})g^{-1})\cdot (g\cdot\mathfrak{x}_{k})
=\rho(\overline{X})\cdot \mathfrak{x}_{k}=I_{k}(\overline{X})$. The $I_{k}$ are called the normalized $G$-invariants and of course they depend of $\rho$.
On the other hand, any other $G$-invariant $M_{l}(\mathbb{C})$-valued function is a function of these normalized $G$-invariants, because if we take
$I\in I(G;M_{l}(\mathbb{C}))$ arbitrary, then since $I(g\cdot \overline{X})=I(\overline{X})$ for all $\overline{X}\in \mathfrak{M}_{n}$ and every $g\in G$,
in particular $I(\overline{X})=I(\rho(\overline{X})\cdot \overline{X})=I(\rho(\overline{X})\cdot \mathfrak{r}_{1},\cdots,\rho(\overline{X})\cdot \mathfrak{r}_{n})
=I(I_{1}(\overline{X}),\ldots,I_{n}(\overline{X}))$.

Next, we will examine an example, define $G_{\ltimes}=GL_{l}(\mathbb{C})\ltimes M_{l}(\mathbb{C})=\{(\mathfrak{z},\mathfrak{w})|\, \mathfrak{z}\in GL_{l}(\mathbb{C}),\mathfrak{w}\in M_{l}(\mathbb{C})\}$,
then $G_{\ltimes}$ is a group with respect to the following product $(\mathfrak{z}_{1},\mathfrak{w}_{1})(\mathfrak{z}_{2},\mathfrak{w}_{2})=(\mathfrak{z}_{1}\mathfrak{z}_{2},\mathfrak{z}_{1}\mathfrak{w}_{2}
+\mathfrak{w}_{1})$. Observe that the unit of $G_{\ltimes}$ is the pair $(I,O)$ and moreover $(\mathfrak{z},\mathfrak{w})^{-1}=(\mathfrak{z}^{-1},-\mathfrak{z}^{-1}\mathfrak{w})$. Consider the action of
$G_{\ltimes}$ on $M_{l}(\mathbb{C})$ defined of the following form $\mathfrak{r}\longrightarrow \mathfrak{z}\mathfrak{r}+\mathfrak{w}$, that is, $(\mathfrak{z},\mathfrak{w})\cdot \mathfrak{a}=\mathfrak{z}\mathfrak{r}+\mathfrak{w}$. It can be seen as a left action because
\begin{equation}\label{mvmf2}
\left(
         \begin{array}{c}
           \mathfrak{z}\mathfrak{r}+\mathfrak{w} \\
           I \\
         \end{array}
       \right)=\left(
  \begin{array}{cc}
    \mathfrak{z} & \mathfrak{w} \\
    O & I \\
  \end{array}
\right)\left(
         \begin{array}{c}
           \mathfrak{r} \\
           I \\
         \end{array}
       \right),
\end{equation}
hence, this left action extend its action to $\mathfrak{M}_{n}$ by means to the product action
\begin{equation}\label{mvmf3}
\left(
  \begin{array}{ccc}
     \mathfrak{z}\mathfrak{r}_{1}+\mathfrak{w} & \cdots & \mathfrak{z}\mathfrak{r}_{n}+\mathfrak{w} \\
    I & \cdots & I \\
  \end{array}
\right)=\left(
  \begin{array}{cc}
    \mathfrak{z} & \mathfrak{w} \\
    O & I \\
  \end{array}
\right)\left(
  \begin{array}{ccc}
     \mathfrak{r}_{1} & \cdots & \mathfrak{r}_{n} \\
    I & \cdots & I \\
  \end{array}
\right).
\end{equation}

We have
\begin{lemma}Denote $\Omega=\{\overline{X}=(\mathfrak{r}_{1},\cdots,\mathfrak{r}_{n})\in \mathfrak{M}_{n}\,|\,\mathfrak{r}_{2}-\mathfrak{r}_{1}\in GL_{l}\}$, then the function
\begin{equation}\label{mvmf4}
\rho(\overline{X})=\rho(\mathfrak{r}_{1},\cdots,\mathfrak{r}_{n})=((\mathfrak{r}_{2}-\mathfrak{r}_{1})^{-1},-(\mathfrak{r}_{2}-\mathfrak{r}_{1})^{-1}
\mathfrak{r}_{1}),
\end{equation}
from $\Omega$ into $G_{\ltimes}$ is a right moving frame.
\end{lemma}
\begin{proof}Observe that for $g=(\mathfrak{z},\mathfrak{w})\in G_{\ltimes}$ arbitrary
\begin{equation*}
\rho(g\cdot\overline{X})=\rho(g\cdot\mathfrak{r}_{1},\cdots,g\cdot\mathfrak{r}_{n})=((g\cdot\mathfrak{r}_{2}-g\cdot\mathfrak{r}_{1})^{-1},
-(g\cdot\mathfrak{r}_{2}-g\cdot\mathfrak{r}_{1})^{-1}
g\cdot\mathfrak{r}_{1}),
\end{equation*}
now we obtain
\begin{equation*}
(g\cdot\mathfrak{r}_{2}-g\cdot\mathfrak{r}_{1})^{-1}=(\mathfrak{r}_{2}-\mathfrak{r}_{1})^{-1}\mathfrak{z}^{-1},
\end{equation*}
and
\begin{equation*}
-(g\cdot\mathfrak{r}_{2}-g\cdot\mathfrak{r}_{1})^{-1}
g\cdot\mathfrak{r}_{1}=-(\mathfrak{r}_{2}-\mathfrak{r}_{1})^{-1}\mathfrak{z}^{-1}(\mathfrak{z}\mathfrak{r}_{1}+\mathfrak{w})=
-(\mathfrak{r}_{2}-\mathfrak{r}_{1})^{-1}(\mathfrak{r}_{1}+\mathfrak{z}^{-1}\mathfrak{w}).
\end{equation*}

Hence,
\begin{equation*}
\rho(g\cdot\overline{X})=((\mathfrak{r}_{2}-\mathfrak{r}_{1})^{-1}\mathfrak{z}^{-1},-(\mathfrak{r}_{2}-\mathfrak{r}_{1})^{-1}
(\mathfrak{r}_{1}+\mathfrak{z}^{-1}\mathfrak{w})).
\end{equation*}

On the other hand
\begin{equation*}
\rho(\overline{X})g^{-1}=((\mathfrak{r}_{2}-\mathfrak{r}_{1})^{-1},-(\mathfrak{r}_{2}-\mathfrak{r}_{1})^{-1}
\mathfrak{r}_{1})(\mathfrak{z}^{-1},-\mathfrak{z}^{-1}\mathfrak{w})=((\mathfrak{r}_{2}-\mathfrak{r}_{1})^{-1}\mathfrak{z}^{-1},
-(\mathfrak{r}_{2}-\mathfrak{r}_{1})^{-1}\mathfrak{z}^{-1}\mathfrak{w}-(\mathfrak{r}_{2}-\mathfrak{r}_{1})^{-1}
\mathfrak{r}_{1})),
\end{equation*}
thus $\rho(g\cdot\overline{X})=\rho(\overline{X})g^{-1}$.
\end{proof}

From this lemma follows that the normalized invariants are in this example
\begin{equation}\label{mvmf5}
I_{k}(\mathfrak{r}_{1},\cdots,\mathfrak{r}_{n})=((\mathfrak{r}_{2}-\mathfrak{r}_{1})^{-1},
-(\mathfrak{r}_{2}-\mathfrak{r}_{1})^{-1}\mathfrak{r}_{1})\cdot \mathfrak{r}_{k}=(\mathfrak{r}_{2}-\mathfrak{r}_{1})^{-1}(\mathfrak{r}_{k}-\mathfrak{r}_{1}),
\end{equation}
for $k=1,2,\ldots,n$, thus $I_{1}=O$, $I_{2}=I$, etc.

We return to the general case. A twisted $n$-gon in $M_{l}(\mathbb{C})$ is a sequence which is constructed of the following form (we recall that $G$ is a left action over $M_{l}(\mathbb{C})$ and $\mathfrak{M}_{n}$)
\begin{equation}\label{mvmf6}
(\mathfrak{q}_{k})_{k\in \mathbb{Z}}=\cdots,m^{-2}\cdot\mathfrak{r}_{1},\cdots,m^{-2}\cdot\mathfrak{r}_{n},m^{-1}\cdot\mathfrak{r}_{1},\cdots,m^{-1}\cdot\mathfrak{r}_{n},
\mathfrak{r}_{1},\cdots,\mathfrak{r}_{n},m\cdot\mathfrak{r}_{1},\cdots,m\cdot
\mathfrak{r}_{n},m^{2}\cdot\mathfrak{r}_{1},\cdots,m^{2}\cdot\mathfrak{r}_{n},\cdots,
\end{equation}
where $\overline{X}=(\mathfrak{r}_{1},\cdots,\mathfrak{r}_{n})\in \mathfrak{M}_{n}$ and $m\in G$ are fixed. In this case $m$ is called the monodromy for the twisted $n$-gon $(\mathfrak{q}_{k})_{k\in \mathbb{Z}}$. Observe that if $(\mathfrak{q}_{k})_{k\in \mathbb{Z}}$ is a twisted
$n$-gon then $\mathfrak{q}_{k+n}=m\cdot \mathfrak{q}_{k}$ for all $k$. Moreover, $G$ acts over the set of all twisted $n$-gon in the form $g\cdot(\mathfrak{q}_{k})_{k\in \mathbb{Z}}=
(g\cdot\mathfrak{q}_{k})_{k\in \mathbb{Z}}$ being it a left action.

The space of all twisted $n$-gons can be identified with $\mathfrak{M}_{n}$. If $\rho$ is a right moving frame (resp. left moving frame) and
$(\mathfrak{q}_{k})_{k\in \mathbb{Z}}$
is a twisted $n$-gon such that $(\mathfrak{q}_{k},\cdots,\mathfrak{q}_{k+n-1})\in \Omega$ for all $k\in\mathbb{Z}$, we can construct the sequence $(\rho_{k}=\rho(\mathfrak{q}_{k},\cdots,
\mathfrak{q}_{k+n-1}))_{k\in \mathbb{Z}}\subset G$.
\begin{definition}For a right moving frame given (resp. left moving frame) the element of $G$, $\mathfrak{K}_{k}=\rho_{k+1}\rho^{-1}_{k}$ (resp. $\mathfrak{K}_{k}=\rho^{-1}_{k}\rho_{k+1}$) is called the right $k$-Maurer-Cartan element (resp. left $k$-Maurer-Cartan element). The equation $\mathfrak{K}_{k}\rho_{k}=\rho_{k+1}$ (resp. $\rho_{k}\mathfrak{K}_{k}=\rho_{k+1}$) is named the discrete right
$k$-Serret-Frenet equation (resp. discrete left $k$-Serret-Frenet equation).
\end{definition}

Clearly, the elements $\mathfrak{K}_{k}$ are invariants under the action of $G$. For our previous example $$\rho_{k}=((\mathfrak{q}_{k+1}-\mathfrak{q}_{k})^{-1},-(\mathfrak{q}_{k+1}-\mathfrak{q}_{k})^{-1}
\mathfrak{q}_{k}),$$
and so $\mathfrak{K}_{k}=((\mathfrak{q}_{k+2}-\mathfrak{q}_{k+1})^{-1}(\mathfrak{q}_{k+1}-\mathfrak{q}_{k}),-(\mathfrak{q}_{k+2}-\mathfrak{q}_{k+1})^{-1}
(\mathfrak{q}_{k+1}-\mathfrak{q}_{k}))$, or also
\begin{equation*}
\rho_{k}=\left(
  \begin{array}{cc}
    (\mathfrak{q}_{k+1}-\mathfrak{q}_{k})^{-1} & -(\mathfrak{q}_{k+1}-\mathfrak{q}_{k})^{-1}
\mathfrak{q}_{k} \\
    O & I \\
  \end{array}
\right),
\end{equation*}
hence
\begin{align*}
\mathfrak{K}_{k}&=\left(
  \begin{array}{cc}
    (\mathfrak{q}_{k+2}-\mathfrak{q}_{k+1})^{-1} & -(\mathfrak{q}_{k+2}-\mathfrak{q}_{k+1})^{-1}
\mathfrak{q}_{k+1} \\
    O & I \\
  \end{array}
\right)\left(
         \begin{array}{cc}
           (\mathfrak{q}_{k+1}-\mathfrak{q}_{k}) & \mathfrak{q}_{k} \\
           O & I \\
         \end{array}
       \right) \\
&=\left(
                 \begin{array}{cc}
                   (\mathfrak{q}_{k+2}-\mathfrak{q}_{k+1})^{-1}(\mathfrak{q}_{k+1}-\mathfrak{q}_{k}) & -(\mathfrak{q}_{k+2}-\mathfrak{q}_{k+1})^{-1}(\mathfrak{q}_{k+1}-\mathfrak{q}_{k}) \\
                   O & I \\
                 \end{array}
               \right).
\end{align*}

\begin{definition}Let $H:\,\mathfrak{M}_{n}\longrightarrow M_{l}(\mathbb{C})$ be a function defined on $n$-gons. We say that $H$ is a discrete $G$-invariant function if
for every twisted $n$-gon $(\mathfrak{q}_{k})_{k\in \mathbb{Z}}$, we have $H(g\cdot (\mathfrak{q}_{k},\cdots,\mathfrak{q}_{k+n-1}))=H(\mathfrak{q}_{k},\cdots,\mathfrak{q}_{k+n-1})$
for all $k\in \mathbb{Z}$ and any $g\in G$.
\end{definition}

For instance, the quantities $H_{k;j}=\rho_{k}\cdot \mathfrak{q}_{j}$ are discrete $G$-invariant functions for all $k,j\in \mathbb{Z}$ and any other $G$-invariant function is a function of these
$H_{k;j}$. In fact, $H(\mathfrak{q}_{j},\cdots,\mathfrak{q}_{j+n-1})=H(g_{k}\cdot (\mathfrak{q}_{j},\cdots,\mathfrak{q}_{j+n-1}))=H(g_{k}\cdot \mathfrak{q}_{j},\cdots,g_{k}\cdot \mathfrak{q}_{j+n-1})=H(H_{k;j},\cdots,H_{k;j+n-1})$.

\subsection{Elements to a theory of matrix-valued frieze patterns}

In this subsection, we explore a notion of matrix frieze pattern. \\

A \textbf{left matrix-valued frieze} $\mathcal{F}_{m}$ with $p-3$ non-trivial rows and $p$ periodic will be
seen in the form
\small
\begin{equation*}
  \mathcal{F}_{m}=\begin{array}{ccccccccccccccccc}
      &  &  & O &  & O &  & O &  &  &  &  & & & & & \\
      &  & I &  & I &  & I &  & I &  &  &  & & & & & \\
      &  &  & \mathfrak{M}_{-1,-1} &  & \mathfrak{M}_{0,0} &  & \mathfrak{M}_{1,1} &  & \mathfrak{M}_{2,2} & &  & & & & & \\
      &  & \cdots &  & \ddots &  & \ddots &  & \ddots & & \ddots & & \cdots & & & &  \\
      &  &  &  &  & \mathfrak{M}_{-1,p-5} &  & \mathfrak{M}_{0,p-4} &  & \mathfrak{M}_{1,p-3} & & \mathfrak{M}_{2,p-2} & & & & & \\
      &  &  &  &  &  & I &  & I &  & I &  & I &  & & & \\
      &  &  &  &  &  &  & O &  & O &  & O &  &  & & & \\
  \end{array},
  \end{equation*}
\normalsize
where $M_{i,j}\in GL_{l}(\mathbb{R})$ for any $i,j\in\mathbb{Z}$ and such that the following diamond matrix rule
\begin{equation}\label{f1}
\mathfrak{M}_{i,j}\mathfrak{M}_{i+1,j+1}-\mathfrak{M}_{i,j+1}\mathfrak{M}_{i+1,j}=I,
\end{equation}
holds for all $i,j\in\mathbb{Z}$. If we replace (\ref{f1}) by
\begin{equation}\label{f2}
\mathfrak{M}_{i+1,j+1}\mathfrak{M}_{i,j}-\mathfrak{M}_{i+1,j}\mathfrak{M}_{i,j+1}=I,
\end{equation}
then $\mathcal{F}_{m}$ is called \textbf{right matrix-valued frieze}. It is clear that we can define other types of matrix-valued friezes depending on
the order in which the matrices are located in the two terms of (\ref{f1}). However in this section, we prefer to concentrate only in the two types previously defined of matrix-valued frieze.
A matrix-valued finite frieze which is both left and right matrix-valued finite frieze will be call a \textbf{two-sided matrix-valued frieze}.

\begin{definition}Let $\mathcal{F}_{m}$ be a matrix-valued finite frieze (left or right). The finite sequence $(\mathfrak{M}_{0,0},\cdots,\mathfrak{M}_{p-1,p-1})$ is called the \textbf{matrix frieze quiddity
sequence} of $\mathcal{F}_{m}$.
\end{definition}

From now on, we identify a matrix-valued frieze with its quiddity sequence. If $\mathcal{F}_{m}$ is a matrix-valued finite frieze its nontrivial rows are those that are located between the two rows composed only of the identity matrix. In this point, we recall the definition the \textbf{joint spectrum}.
For an $s$-tuple $T=(T_{1},\cdots, T_{s})$ of complex $l\times l$-matrices, we define the joint spectrum $\sigma(T)$ as the set of all points
$\lambda=(\lambda_{1},\cdots,\lambda_{s})\in\mathbb{C}^{s}$ for which there exists a nonzero vector $x\in\mathbb{C}^{l}$ (called the joint eigenvector)
satisfying
\begin{equation*}
T_{k}x=\lambda_{k}x,
\end{equation*}
for $k=1,\ldots, s$. If the $(T_{k})^{,}$s are commuting then $\sigma(T)\neq \emptyset$.

\begin{proposition} We give two simple properties of the matrix-valued frieze patterns
\begin{enumerate}
  \item  If we transpose the non-trivial rows of a matrix-valued frieze pattern $\mathcal{F}_{m}$, we obtain a new matrix-valued frieze pattern denoted by $\mathcal{F}_{m}^{\ast}$. Suppose that
  $\mathcal{F}_{m}$ is a right matrix-valued frieze then $\mathcal{F}_{m}^{\ast}$ is a left matrix-valued frieze. Hence, if $\mathcal{F}_{m}$ is a two-sided
matrix-valued frieze then $\mathcal{F}_{m}^{\ast}$ is a two-sided
matrix-valued frieze.
  \item Let us assume that $\mathcal{F}_{m}$ is matrix-valued frieze pattern such that all its matrices commute and let $x\in \mathbb{C}^{l}$ be a
  common eigenvector to all matrices of $\mathcal{F}_{m}$ with $\|x\|=1$.
  Then, we can construct a scalar frieze pattern $\mathcal{F}_{s}(x)$ with complex entries of the following form: if the matrix $A$ is an entry of $\mathcal{F}_{m}$ then the corresponding entry of
  $\mathcal{F}_{s}(x)$ is $(Ax,x)_{\mathbb{C}^{l}}$.
\end{enumerate}
\end{proposition}
\begin{proof} $1.$ follows of (\ref{f1}) and (\ref{f2}). To prove $2.$ observe that there exist $\lambda_{i,j}\in \mathbb{C}$ such that $\mathfrak{M}_{i,j}x=\lambda_{i,j}x$ where $x$ is a joint eigenvector, therefore again the assertion
follows for using (\ref{f1}) and (\ref{f2}).
  \end{proof}

The proof of the following proposition is a trivial calculate.

\begin{proposition}For all $\mathfrak{M}\in GL_{l}(\mathbb{R})$
\small
\begin{equation}\label{f3}
\mathcal{F}_{M}=\begin{array}{ccccccccccccc}
&  O & & O &   & O &  & O &  &  &  &  &  \\
&   & I & &  I &  & I &  & I &  &  &  &  \\
& \cdots  &  & \mathfrak{M} &  & 2\mathfrak{M}^{-1} &  & \mathfrak{M} &  & 2\mathfrak{M}^{-1} &  & \cdots &  \\
&   &  &  & I &  & I &  & I &  & I &  & \\
&   &  &  &  & O &  & O &  & O &  & O &
\end{array},
\end{equation}
\normalsize
is a two-sided matrix-valued finite frieze of period $4$, its matrix frieze quiddity sequence is $(\mathfrak{M},2\mathfrak{M}^{-1},\mathfrak{M},2\mathfrak{M}^{-1})$ called for us here
the \textbf{basic matrix frieze quiddity sequence}. Moreover, we already know that the basic matrix frieze quiddity sequence is a left
matrix quiddity sequence.
\end{proposition}

Starting from the basic matrix frieze quiddity sequence $(\mathfrak{M},2\mathfrak{M}^{-1},\mathfrak{M},2\mathfrak{M}^{-1})$ we can obtain other matrix frieze quiddity sequence. In this sense, we have
\begin{theorem}The following matrix vector $\mathcal{MQS}=(I,\mathfrak{M}+I,2\mathfrak{M}^{-1},\mathfrak{M},2\mathfrak{M}^{-1}+I)$
is a matrix frieze quiddity sequence (and therefore so are its cyclical permutations) corresponding to $5$-periodic matrix-valued frieze pattern.
\end{theorem}
\begin{proof}We justly show the matrix-valued frieze pattern corresponding to this matrix vector:
\begin{equation*}
\mathcal{F}_{\mathcal{MQS}}=
  \begin{array}{ccccccccccccccccc}
     &  & O &  & O &  & O &  & O &  & O &  &  & & & \\
     & I &  & I &  & I &  & I &  & I &  & I &  & & & \\
     &  & I &  & \mathfrak{M}+I &  & 2\mathfrak{M}^{-1} &  & \mathfrak{M} &  & 2\mathfrak{M}^{-1}+I & & I & & & \\
     & \cdots &  & \mathfrak{M} &  & 2\mathfrak{M}^{-1}+I &  & I &  & \mathfrak{M}+I &  & 2\mathfrak{M}^{-1} & & \mathfrak{M} &  & \cdots \\
     &  &  &  & I &  & I &  & I &  & I &  & I &  & I & \\
      &  &  &  &  & O &  & O &  & O &  & O &  & O &  & O \\
 \end{array}.
\end{equation*}
\end{proof}

A more general result is the following

\begin{theorem}Suppose that $\mathfrak{A},\mathfrak{B}\in GL_{l}(\mathbb{K})$ such that $[\mathfrak{A},\mathfrak{B}]=0$ then
\begin{equation}\label{nf1}
\mathfrak{Q}=(\mathfrak{A},\mathfrak{A}^{-1}(I+\mathfrak{B}),(I+\mathfrak{A})\mathfrak{B}^{-1},
\mathfrak{B},\mathfrak{B}^{-1}(I+\mathfrak{A}+\mathfrak{B})\mathfrak{A}^{-1}),
\end{equation}
is a matrix frieze quiddity sequence of a two-sided matrix-valued frieze of period $5$.
\end{theorem}
\begin{proof}Indeed, the corresponding matrix-valued frieze pattern for the row matrix vector (\ref{nf1}) is the following
\scriptsize
\begin{equation*}
  \begin{array}{ccccccccccccccccc}
   &  & O &  & O &  & O &  & O &  & O &  & \\
     & I &  & I &  & I &  & I &  & I &  &  I \\
     & \cdots & \mathfrak{A} &  & \mathfrak{A}^{-1}(I+\mathfrak{B}) &  & (I+\mathfrak{A})\mathfrak{B}^{-1} &  & \mathfrak{B} &  & \mathfrak{B}^{-1}(I+\mathfrak{A}+\mathfrak{B})\mathfrak{A}^{-1} & & \mathfrak{A} & \cdots  \\
     &  &  & \mathfrak{B} &  & \mathfrak{B}^{-1}(I+\mathfrak{A}+\mathfrak{B})\mathfrak{A}^{-1} &  & \mathfrak{A} &  & \mathfrak{A}^{-1}(I+\mathfrak{B}) &  & (I+\mathfrak{A})\mathfrak{B}^{-1} &  & \mathfrak{B} \\
     &  &  &  & I &  & I &  & I &  & I &  & I & \\
     &  &  &  &  & O &  & O &  & O &  & O &  & O \\
 \end{array}\,\,.
\end{equation*}
\end{proof}
\normalsize

\subsection{Noncommutative signed Chebyshev polynomials}

In this subsection, different types of matrix Chebyshev polynomials are introduced. Also, we show the relation of these polynomials with
the matrix periodic difference equations and their monodromy matrices. In the classic case (for scalars), the interested reader can consult
\cite{Conley}.

Let $(\mathfrak{a}_{k})_{k\geq 1}\subset GL_{l}(\mathbb{C})$ be a sequence of square matrices of certain order. The left matrix signed Chebyshev
polynomials are defined of the following recurrent form: $\mathfrak{p}_{-1}=O$, $\mathfrak{p}_{0}=I$ and
\begin{equation}\label{ii1}
\mathfrak{p}_{m}(\mathfrak{a}_{1},\cdots,\mathfrak{a}_{m})=\mathfrak{a}_{m}\mathfrak{p}_{m-1}(\mathfrak{a}_{1},\cdots,\mathfrak{a}_{m-1})-
\mathfrak{p}_{m-2}(\mathfrak{a}_{1},\cdots,\mathfrak{a}_{m-2}),
\end{equation}
for instance, $\mathfrak{p}_{1}(\mathfrak{a}_{1})=\mathfrak{a}_{1}$, $\mathfrak{p}_{2}(\mathfrak{a}_{1},\mathfrak{a}_{2})=
\mathfrak{a}_{2}\mathfrak{a}_{1}-I$, $\mathfrak{p}_{3}(\mathfrak{a}_{1},\mathfrak{a}_{2},\mathfrak{a}_{3})=
\mathfrak{a}_{3}(\mathfrak{a}_{2}\mathfrak{a}_{1}-I)-\mathfrak{a}_{1}$, etc.

\begin{lemma}For all $m\geq 1$, we have
\begin{equation}\label{ii2}
M(\mathfrak{a}_{1},\cdots,\mathfrak{a}_{m})=\left(
                                               \begin{array}{cc}
                                                 \mathfrak{p}_{m}(\mathfrak{a}_{1},\cdots,\mathfrak{a}_{m}) & -\mathfrak{p}_{m-1}(\mathfrak{a}_{2},\cdots,\mathfrak{a}_{m}) \\
                                                 \mathfrak{p}_{m-1}(\mathfrak{a}_{1},\cdots,\mathfrak{a}_{m-1}) & -\mathfrak{p}_{m-2}(\mathfrak{a}_{2},\cdots,\mathfrak{a}_{m-1}) \\
                                               \end{array}
                                             \right),
\end{equation}
we remember that $|M(\mathfrak{a}_{1},\cdots,\mathfrak{a}_{m})|=1$.
\end{lemma}
\begin{proof}We prove the lemma by induction. For $m=1$ we have
\begin{equation*}
M(\mathfrak{a}_{1})=\left(
                       \begin{array}{cc}
                         \mathfrak{a}_{1} & -I \\
                         I & O \\
                       \end{array}
                     \right)=\left(
                               \begin{array}{cc}
                                 \mathfrak{p}_{1}(\mathfrak{a}_{1}) & -\mathfrak{p}_{0} \\
                                 \mathfrak{p}_{0} & -\mathfrak{p}_{-1} \\
                               \end{array}
                             \right),
\end{equation*}
thus the result is true for $m=1$. Let us suppose that (\ref{ii2}) is hold for $m=s$, then
\begin{align*}
M(\mathfrak{a}_{1},\cdots,\mathfrak{a}_{s+1})&=\left(
                                                 \begin{array}{cc}
                                                   \mathfrak{a}_{s+1} & -I \\
                                                   I & O \\
                                                 \end{array}
                                               \right)
M(\mathfrak{a}_{1},\cdots,\mathfrak{a}_{s})=\left(
                                                 \begin{array}{cc}
                                                   \mathfrak{a}_{s+1} & -I \\
                                                   I & O \\
                                                 \end{array}
                                               \right)\left(
                                               \begin{array}{cc}
                                                 \mathfrak{p}_{s}(\mathfrak{a}_{1},\cdots,\mathfrak{a}_{s}) & -\mathfrak{p}_{s-1}(\mathfrak{a}_{2},\cdots,\mathfrak{a}_{s}) \\
                                                 \mathfrak{p}_{s-1}(\mathfrak{a}_{1},\cdots,\mathfrak{a}_{s-1}) & -\mathfrak{p}_{s-2}(\mathfrak{a}_{2},\cdots,\mathfrak{a}_{s-2}) \\
                                               \end{array}
                                             \right) \\
&=\left(
    \begin{array}{cc}
    \mathfrak{a}_{s+1}\mathfrak{p}_{s}(\mathfrak{a}_{1},\cdots,\mathfrak{a}_{s})-\mathfrak{p}_{s-1}(\mathfrak{a}_{1},\cdots,\mathfrak{a}_{s-1})  &
    -\mathfrak{a}_{s+1}\mathfrak{p}_{s-1}(\mathfrak{a}_{2},\cdots,\mathfrak{a}_{s})+\mathfrak{p}_{s-2}(\mathfrak{a}_{2},\cdots,\mathfrak{a}_{s-1}) \\
      \mathfrak{p}_{s}(\mathfrak{a}_{1},\cdots,\mathfrak{a}_{s}) & -\mathfrak{p}_{s-1}(\mathfrak{a}_{2},\cdots,\mathfrak{a}_{s}) \\
    \end{array}
  \right) \\
&=\left(
    \begin{array}{cc}
      \mathfrak{p}_{s+1}(\mathfrak{a}_{1},\cdots,\mathfrak{a}_{s},\mathfrak{a}_{s+1}) &  -\mathfrak{p}_{s}(\mathfrak{a}_{2},\cdots,\mathfrak{a}_{s+1}) \\
       \mathfrak{p}_{s}(\mathfrak{a}_{1},\cdots,\mathfrak{a}_{s}) & -\mathfrak{p}_{s-1}(\mathfrak{a}_{2},\cdots,\mathfrak{a}_{s}) \\
    \end{array}
  \right).
\end{align*}
\end{proof}

Define
\begin{equation}\label{ii3}
\mathfrak{Q}_{m}=\left(
                   \begin{array}{ccccc}
                     \mathfrak{a}_{m} & I & O & \cdots & O \\
                     I & \mathfrak{a}_{m-1} & \ddots & \ddots & \vdots \\
                     O & \ddots & \ddots & \ddots & O \\
                     \vdots & \ddots & \ddots & \ddots & I \\
                     O & \cdots & O & I & \mathfrak{a}_{1} \\
                   \end{array}
                 \right),
\end{equation}
for $m\geq 2$ and $\mathfrak{Q}_{1}=\mathfrak{a}_{1}$, where $\mathfrak{a}_{k}\in GL_{l}$ for $k=\overline{1,m}$, then we have
\begin{proposition}For all $m\geq 1$
\begin{equation}\label{ii4}
|\mathfrak{Q}_{m}|=|\mathfrak{p}_{m}(\mathfrak{a}_{1},\cdots,\mathfrak{a}_{m})|.
\end{equation}
\end{proposition}
\begin{proof}We prove the proposition first for $m=1$, $m=2$ and $m=3$. It is clear that $|\mathfrak{Q}_{1}|=|\mathfrak{a}_{1}|=|\mathfrak{p}_{1}(\mathfrak{a}_{1})|$. From the Schur determinant lemma, we obtain
\begin{equation*}
|\mathfrak{Q}_{2}|=\left |
                             \begin{array}{cc}
                               \mathfrak{a}_{2} & I \\
                               I & \mathfrak{a}_{1} \\
                             \end{array}
\right |=|\mathfrak{a}_{2}\mathfrak{a}_{1}-I|=|\mathfrak{p}_{2}(\mathfrak{a}_{1},\mathfrak{a}_{2})|.
\end{equation*}

Next, we compute $|\mathfrak{Q}_{3}|$ through the Schur's formula (also we use the Schur determinant lemma). We have
\begin{equation*}
|\mathfrak{Q}_{3}|=\left |
                             \begin{array}{ccc}
                               \mathfrak{a}_{3} & I & O \\
                               I & \mathfrak{a}_{2} & I \\
                               O & I & \mathfrak{a}_{1} \\
                             \end{array}
\right |=|\mathfrak{a}_{3}|\,\left | \begin{array}{cc}
           \mathfrak{a}_{2}-\mathfrak{a}^{-1}_{3} & I \\
           I &  \mathfrak{a}_{1}
         \end{array}
\right |=|\mathfrak{a}_{3}|\,|(\mathfrak{a}_{2}-\mathfrak{a}^{-1}_{3})\mathfrak{a}_{1}-I|=|\mathfrak{a}_{3}|\,|\mathfrak{a}^{-1}_{3}|\,
|\mathfrak{a}_{3}\mathfrak{a}_{2}\mathfrak{a}_{1}-\mathfrak{a}_{3}-\mathfrak{a}_{1}|=|\mathfrak{p}_{3}(\mathfrak{a}_{1},\mathfrak{a}_{2},
\mathfrak{a}_{3})|.
\end{equation*}

Hence, the result is true for $m=1$, $m=2$ and $m=3$.
Next, we will proceed by induction. Suppose that $|\mathfrak{Q}_{s}|=|\mathfrak{p}_{s}(\mathfrak{a}_{1},\cdots,\mathfrak{a}_{s})|$. Then
\begin{equation}\label{ii5}
|\mathfrak{Q}_{s+1}|=|\mathfrak{a}_{s+1}|\left | \begin{array}{ccccc}
                     \mathfrak{a}_{s}-\mathfrak{a}^{-1}_{s+1} & I & O & \cdots & O \\
                     I & \mathfrak{a}_{s-1} & \ddots & \ddots & \vdots \\
                     O & \ddots & \ddots & \ddots & O \\
                     \vdots & \ddots & \ddots & \ddots & I \\
                     O & \cdots & O & I & \mathfrak{a}_{1} \\
                   \end{array}  \right | =|\mathfrak{a}_{s+1}|\,|\mathfrak{p}_{s}(\mathfrak{a}_{1},\cdots,\mathfrak{a}_{s}-\mathfrak{a}^{-1}_{s+1})|.
\end{equation}

Nevertheless, observe that
\begin{align}\label{ii6}
\mathfrak{p}_{s}(\mathfrak{a}_{1},\cdots,\mathfrak{a}_{s}-\mathfrak{a}^{-1}_{s+1})&=(\mathfrak{a}_{s}-\mathfrak{a}^{-1}_{s+1})
\mathfrak{p}_{s-1}(\mathfrak{a}_{1},\cdots,\mathfrak{a}_{s-1})-\mathfrak{p}_{s-2}(\mathfrak{a}_{1},\cdots,\mathfrak{a}_{s-2}) \nonumber \\
&=\mathfrak{p}_{s}(\mathfrak{a}_{1},\cdots,\mathfrak{a}_{s})-\mathfrak{a}^{-1}_{s+1}\mathfrak{p}_{s-1}(\mathfrak{a}_{1},\cdots,\mathfrak{a}_{s-1})=
\mathfrak{a}^{-1}_{s+1}(\mathfrak{a}_{s+1}\mathfrak{p}_{s}(\mathfrak{a}_{1},\cdots,\mathfrak{a}_{s})-\mathfrak{p}_{s-1}(\mathfrak{a}_{1},
\cdots,\mathfrak{a}_{s-1})) \nonumber \\
&=\mathfrak{a}^{-1}_{s+1}\mathfrak{p}_{s+1}(\mathfrak{a}_{1},\cdots,\mathfrak{a}_{s+1}),
\end{align}
so combining (\ref{ii5}) and (\ref{ii6}) we obtain (\ref{ii4}).
\end{proof}

Let us denote by $\mathfrak{R}_{m}$ the formal inverse of $\mathfrak{Q}_{m}$ for $1\leq m$ and define $\mathfrak{P}_{m}=(-1)^{m-1}\left ((\mathfrak{R}_{m})_{m1} \right )^{-1}$ where
$(\mathfrak{R}_{m})_{m1}$ is the $l\times l$ matrix located in the position $m1$ of the matrix $\mathfrak{R}_{m}$.

\begin{proposition}For all $m\geq 1$, we have
\begin{equation}\label{ii7}
\mathfrak{P}_{m}=\mathfrak{p}_{m}(\mathfrak{a}_{1},\cdots,\mathfrak{a}_{m}).
\end{equation}
\end{proposition}
\begin{proof}It will be useful to do the calculation of $\mathfrak{P}_{1}, \mathfrak{P}_{2}$ and $\mathfrak{P}_{3}$. For $m=1$, $\mathfrak{R}_{1}=
\mathfrak{a}_{1}^{-1}$ hence $\mathfrak{P}_{1}=\left (\mathfrak{a}_{1}^{-1}\right )^{-1}=\mathfrak{a}_{1}
=\mathfrak{p}_{1}(\mathfrak{a}_{1})$. Lets us suppose that $m=2$, then from a straightforward calculation, we can see that
$(\mathfrak{R}_{2})_{21}=-(\mathfrak{a}_{2}\mathfrak{a}_{1}-I)^{-1}$, thus (\ref{ii7}) holds. If $m=3$ the entries of $\mathfrak{R}_{3}$,
$(\mathfrak{R}_{3})_{11}$, $(\mathfrak{R}_{3})_{21}$ and $(\mathfrak{R}_{3})_{31}$ can be calculated of the following equations
\begin{equation*}
\mathfrak{a}_{3}(\mathfrak{R}_{3})_{11}+(\mathfrak{R}_{3})_{21}=I, (\mathfrak{R}_{3})_{11}+\mathfrak{a}_{2}(\mathfrak{R}_{3})_{21}+(\mathfrak{R}_{3})_{31}=O,
(\mathfrak{R}_{3})_{21}+\mathfrak{a}_{1}(\mathfrak{R}_{3})_{31}=O,
\end{equation*}
thus $(\mathfrak{R}_{3})_{21}=-\mathfrak{a}_{1}(\mathfrak{R}_{3})_{31}$ and $(\mathfrak{R}_{3})_{11}=(\mathfrak{a}_{2}\mathfrak{a}_{1}-I)(\mathfrak{R}_{3})_{31}$. It shows that
$$(\mathfrak{R}_{3})_{31}=\left (\mathfrak{a}_{3}(\mathfrak{a}_{2}\mathfrak{a}_{1}-I)-\mathfrak{a}_{1} \right )^{-1},$$
hence $\mathfrak{P}_{3}=\left (\mathfrak{a}_{3}(\mathfrak{a}_{2}\mathfrak{a}_{1}-I)-\mathfrak{a}_{1} \right )=
\mathfrak{p}(\mathfrak{a}_{1},\mathfrak{a}_{2},\mathfrak{a}_{3})$.

In the general case, in order to find the first column of $\mathfrak{R}_{m}$, we must solve the linear system $:$
\begin{equation*}
\left(
                   \begin{array}{ccccc}
                     \mathfrak{a}_{m} & I & O & \cdots & O \\
                     I & \mathfrak{a}_{m-1} & \ddots & \ddots & \vdots \\
                     O & \ddots & \ddots & \ddots & O \\
                     \vdots & \ddots & \ddots & \ddots & I \\
                     O & \cdots & O & I & \mathfrak{a}_{1} \\
                   \end{array}
                 \right)\left(
                          \begin{array}{c}
                            (\mathfrak{R}_{s})_{11} \\
                            (\mathfrak{R}_{s})_{21} \\
                            \vdots \\
                            (\mathfrak{R}_{s})_{(s-1)1} \\
                            (\mathfrak{R}_{s})_{s1} \\
                          \end{array}
                        \right)=\left(
                                  \begin{array}{c}
                                    I \\
                                    O \\
                                    \vdots \\
                                    \vdots \\
                                    O \\
                                  \end{array}
                                \right),
\end{equation*}
now, using the last $(m-1)$ equations of this system from the bottom to the top, we obtain
\begin{align*}
(\mathfrak{R}_{m})_{(m-1)1}&=-\mathfrak{p}_{1}(\mathfrak{a}_{1})(\mathfrak{R}_{m})_{m1}, \\
(\mathfrak{R}_{m})_{(m-2)1}&=-\mathfrak{a}_{2}(\mathfrak{R}_{m})_{(m-1)1}-(\mathfrak{R}_{m})_{m1}=(\mathfrak{a}_{2}\mathfrak{p}_{1}(\mathfrak{a}_{1})-I)
(\mathfrak{R}_{m})_{m1}
=\mathfrak{p}_{2}(\mathfrak{a}_{1},\mathfrak{a}_{2})(\mathfrak{R}_{m})_{m1}, \\
&\qquad\qquad\cdots\qquad\cdots\qquad\cdots\qquad\cdots\qquad\cdots\qquad\cdots \\
(\mathfrak{R}_{m})_{21}&=(\mathfrak{R}_{m})_{(m-(m-2))1}=(-1)^{(m-2)}\mathfrak{p}_{m-2}(\mathfrak{a}_{1},\cdots,\mathfrak{a}_{m-2})(\mathfrak{R}_{m})_{m1}, \\
(\mathfrak{R}_{m})_{11}&=(\mathfrak{R}_{m})_{(m-(m-1))1}=(-1)^{(m-1)}\mathfrak{p}_{m-1}(\mathfrak{a}_{1},\cdots,\mathfrak{a}_{m-1})(\mathfrak{R}_{m})_{m1}.
\end{align*}

The first equation is $\mathfrak{a}_{m}(\mathfrak{R}_{s})_{11}+(\mathfrak{R}_{s})_{21}=I$, then
\begin{align*}
I&=\mathfrak{a}_{m}\left ( (-1)^{(m-1)}\mathfrak{p}_{m-1}(\mathfrak{a}_{1},\cdots,\mathfrak{a}_{m-1})+(-1)^{(m-2)}\mathfrak{p}_{m-2}(\mathfrak{a}_{1},\cdots,\mathfrak{a}_{m-2}) \right )(\mathfrak{R}_{m})_{m1} \\
&=(-1)^{(m-1)}\left (\mathfrak{a}_{m}\mathfrak{p}_{m-1}(\mathfrak{a}_{1},\cdots,\mathfrak{a}_{m-1})-\mathfrak{p}_{m-2}(\mathfrak{a}_{1},\cdots,\mathfrak{a}_{m-2}) \right )(\mathfrak{R}_{m})_{m1} \\
&=(-1)^{(m-1)}\mathfrak{p}_{m}(\mathfrak{a}_{1},\cdots,\mathfrak{a}_{m})(\mathfrak{R}_{m})_{m1},
\end{align*}
it shows that $(\mathfrak{R}_{m})_{m1}=(-1)^{(m-1)}(\mathfrak{p}_{m}(\mathfrak{a}_{1},\cdots,\mathfrak{a}_{m}))^{-1}$. Therefore (\ref{ii7}) holds.
\end{proof}

In this moment, we will show the relation between left matrix Chebyshev polynomials and matrix difference equations. Consider the recurrent equation
\begin{equation}\label{sf1}
\mathfrak{y}_{n+1}=\mathfrak{a}_{n}\mathfrak{y}_{n}-\mathfrak{y}_{n-1},\quad\quad\quad 1\leq n,
\end{equation}
where $(\mathfrak{a}_{n})\subset M_{l}(\mathbb{C})$. We put $\mathfrak{a}_{0}=I$, then
\begin{lemma}For $\mathfrak{y}_{0}, \mathfrak{y}_{1}$ given and $1\leq n$, we have
\begin{equation}\label{polche1}
\mathfrak{y}_{n+1}=\mathfrak{p}_{n}(\mathfrak{a}_{1},\cdots,\mathfrak{a}_{n})\mathfrak{y}_{1}-
\mathfrak{p}_{n-1}(\mathfrak{a}_{2},\cdots,\mathfrak{a}_{n})\mathfrak{y}_{0}.
\end{equation}
\end{lemma}
\begin{proof}We do the proof by complete induction. Clearly, the lemma holds for $n=1$. Now, if $n=2$ then
$$\mathfrak{y}_{3}=\mathfrak{a}_{2}\mathfrak{y}_{2}-\mathfrak{y}_{1}=\mathfrak{a}_{2}(\mathfrak{a}_{1}\mathfrak{y}_{1}-\mathfrak{y}_{0})-\mathfrak{y}_{1}
=(\mathfrak{a}_{2}\mathfrak{a}_{1}-I)\mathfrak{y}_{1}-\mathfrak{a}_{2}\mathfrak{y}_{0}=\mathfrak{p}_{2}(\mathfrak{a}_{1},\mathfrak{a}_{2})\mathfrak{y}_{1}
-\mathfrak{p}_{1}(\mathfrak{a}_{2})\mathfrak{y}_{0}.$$
Suppose the result holds for $n\leq k$, then
\begin{equation}\label{sf2}
\mathfrak{y}_{k+1}=\mathfrak{p}_{k}(\mathfrak{a}_{1},\cdots,\mathfrak{a}_{k})\mathfrak{y}_{1}-
\mathfrak{p}_{k-1}(\mathfrak{a}_{2},\cdots,\mathfrak{a}_{k})\mathfrak{y}_{0},
\end{equation}
and
\begin{equation}\label{sf3}
\mathfrak{y}_{k}=\mathfrak{p}_{k-1}(\mathfrak{a}_{1},\cdots,\mathfrak{a}_{k-1})\mathfrak{y}_{1}-
\mathfrak{p}_{k-2}(\mathfrak{a}_{2},\cdots,\mathfrak{a}_{k-1})\mathfrak{y}_{0},
\end{equation}
thus, if we multiply to the left of (\ref{sf2}) by $\mathfrak{a}_{k+1}$ and to the resulting equality one subtracts (\ref{sf3}) we obtain
$$\mathfrak{y}_{k+2}=\mathfrak{p}_{k+1}(\mathfrak{a}_{1},\cdots,\mathfrak{a}_{k+1})\mathfrak{y}_{1}-
\mathfrak{p}_{k}(\mathfrak{a}_{2},\cdots,\mathfrak{a}_{k+1})\mathfrak{y}_{0},$$
hence the result is also true for $n=k+1$.
\end{proof}

Suppose now that the sequence $(\mathfrak{a}_{n})$ is $N$-periodic, then from the previous lemma follows that in order to any solution
$(\mathfrak{y}_{n})$ of (\ref{sf1}) constitutes an $N$-antiperiodic sequence is necessary and sufficient that
\begin{equation}\label{sf4}
\mathfrak{p}_{N-1}(\mathfrak{a}_{1},\cdots,\mathfrak{a}_{N-1})=0,\,\,\,\,\,\,\,\,\,\,\,\mathfrak{p}_{N-2}(\mathfrak{a}_{2},\cdots,\mathfrak{a}_{N-1})=I,
\end{equation}
and
\begin{equation}\label{sf5}
\mathfrak{p}_{N}(\mathfrak{a}_{1},\cdots,\mathfrak{a}_{N})=-I,\,\,\,\,\,\,\,\,\,\,\,\mathfrak{p}_{N-1}(\mathfrak{a}_{2},\cdots,\mathfrak{a}_{N})=0,
\end{equation}
indeed, the conditions in (\ref{sf4}) and (\ref{sf5}) imply that $\mathfrak{y}_{N}=-\mathfrak{y}_{0}$, $\mathfrak{y}_{N+1}=-\mathfrak{y}_{1}$.

We present a more general result

\begin{proposition}Consider a recurrent relation (\ref{sf1}) for which the sequence $(\mathfrak{a}_{n})_{1\leq n}$ is $N$-periodic. Suppose that
the matrix $\mathfrak{m}$ satisfies $[\mathfrak{m}, \mathfrak{a}_{k}]=0$ for $k=1,2,\cdots,N$. Then, $\mathfrak{m}$ is the monodromy matrix
of all solution of (\ref{sf1}) (that is, if $(\mathfrak{y}_{k})_{0\leq k}$ is any solution of (\ref{sf1}), then $\mathfrak{y}_{k+N}=\mathfrak{m}\mathfrak{y}_{k}$
for all $0\leq k$) if and only if
\begin{equation}\label{sf6}
\mathfrak{p}_{N-1}(\mathfrak{a}_{1},\cdots,\mathfrak{a}_{N-1})=0,\,\,\,\,\,\,\,\,\,\,\,\mathfrak{p}_{N-2}(\mathfrak{a}_{2},\cdots,\mathfrak{a}_{N-1})=
-\mathfrak{m},
\end{equation}
and
\begin{equation}\label{sf7}
\mathfrak{p}_{N}(\mathfrak{a}_{1},\cdots,\mathfrak{a}_{N})=\mathfrak{m},\,\,\,\,\,\,\,\,\,\,\,\mathfrak{p}_{N-1}(\mathfrak{a}_{2},\cdots,\mathfrak{a}_{N})=0.
\end{equation}

Under the hypothesis of the proposition
\begin{equation}\label{sf71}
M(\mathfrak{a}_{1},\cdots,\mathfrak{a}_{N})=\left(
                                              \begin{array}{cc}
                                                \mathfrak{m} & O \\
                                                O & \mathfrak{m} \\
                                              \end{array}
                                            \right).
\end{equation}
\end{proposition}
\begin{proof}Let $(\mathfrak{y}_{k})_{0\leq k}$ be an arbitrary solution. From (\ref{polche1}) and (\ref{sf6}) follow that $\mathfrak{y}_{N}=\mathfrak{m}\mathfrak{y}_{0}$, and taking into account (\ref{polche1}) and
(\ref{sf7}) is easy to obtain that $\mathfrak{y}_{N+1}=\mathfrak{m}\mathfrak{y}_{1}$. Now using complete induction it is immediate to see that for all $0\leq k$
we have $\mathfrak{y}_{k+N}=\mathfrak{m}\mathfrak{y}_{k}$. In fact, let us assume that this is hold for $k\leq r$ (we already know that the statement is true for
$k=0$ and $k=1$) then
\begin{equation*}
\mathfrak{y}_{k+1+N}=\mathfrak{a}_{k+N}\mathfrak{y}_{k+N}-\mathfrak{y}_{k-1+N}=\mathfrak{a}_{k}\mathfrak{m}\mathfrak{y}_{k}-\mathfrak{m}\mathfrak{y}_{k-1}
=\mathfrak{m}(\mathfrak{a}_{k}\mathfrak{y}_{k}-\mathfrak{y}_{k-1})=\mathfrak{m}\mathfrak{y}_{k+1},
\end{equation*}
since $\mathfrak{y}_{0}$ and $\mathfrak{y}_{1}$ are arbitrary it implies that if (\ref{sf6}) and (\ref{sf7}) are satisfied, then all solution of (\ref{sf1}) has to $\mathfrak{m}$ as monodromy matrix. The proof of the necessity follows from (\ref{polche1}).
Finally, observe that (\ref{sf71}) can be obtained from (\ref{ii2}).
\end{proof}

\begin{remark}Under the conditions of the previous proposition any solution $(\mathfrak{y}_{k})_{0\leq k}$ of (\ref{sf1}) where $(\mathfrak{a}_{n})_{1\leq n}$ is $N$-periodic can be extended to the left obtaining a solution $(\widehat{\mathfrak{y}}_{k})_{k\in\mathbb{Z}}$ of (\ref{sf1}) which maintains the property that $\widehat{\mathfrak{y}}_{k+N}=\mathfrak{m}\widehat{\mathfrak{y}}_{k}$ for all $k\in\mathbb{Z}$. On the other hand, if $[\mathfrak{m}, \mathfrak{a}_{k}]=0$ for $k=1,2,\cdots,N$ then for any solution $(\mathfrak{a}_{1},\cdots,\mathfrak{a}_{N})$ of the equation (\ref{sf71}) each of its
cyclic permutations is also a solution (\ref{sf71}). Indeed, suppose that $[\mathfrak{m}, \mathfrak{b}]=0$ then
\begin{equation*}
\left [\left(
         \begin{array}{cc}
           \mathfrak{m} & O \\
           O & \mathfrak{m} \\
         \end{array}
       \right)
,\left(
   \begin{array}{cc}
     \mathfrak{b} & -I \\
     I & O \\
   \end{array}
 \right)
 \right ]=\left(
            \begin{array}{cc}
              O & O \\
              O & O \\
            \end{array}
          \right),
\end{equation*}
and this implies the assertion.
\end{remark}

Consider again the recurrent relation (\ref{sf1}) where the sequence $(\mathfrak{a}_{n})_{1\leq n}$ is $N$-periodic such that for $0\leq k$
\begin{equation}\label{simulmon1}
\left(
  \begin{array}{c}
    \mathfrak{y}_{k+1+N} \\
    \mathfrak{y}_{k+N} \\
  \end{array}
\right)=\left(
          \begin{array}{cc}
            \mathfrak{m}_{11} & \mathfrak{m}_{12} \\
            \mathfrak{m}_{21} & \mathfrak{m}_{22} \\
          \end{array}
        \right)\left(
  \begin{array}{c}
    \mathfrak{y}_{k+1} \\
    \mathfrak{y}_{k} \\
  \end{array}
\right),
\end{equation}
for every solution $(\mathfrak{y}_{k})_{0\leq k}$. Then, we obtain (taking $k=0$)
\begin{equation}\label{simulmon2}
\left(
\begin{array}{cc}
\mathfrak{p}_{N}(\mathfrak{a}_{1},\cdots,\mathfrak{a}_{N}) & -\mathfrak{p}_{N-1}(\mathfrak{a}_{2},\cdots,\mathfrak{a}_{N}) \\
\mathfrak{p}_{N-1}(\mathfrak{a}_{1},\cdots,\mathfrak{a}_{N-1}) & -\mathfrak{p}_{N-2}(\mathfrak{a}_{2},\cdots,\mathfrak{a}_{N-1}) \\
\end{array}
\right )=M(\mathfrak{a}_{1},\cdots,\mathfrak{a}_{N})=\left(
          \begin{array}{cc}
            \mathfrak{m}_{11} & \mathfrak{m}_{12} \\
            \mathfrak{m}_{21} & \mathfrak{m}_{22} \\
          \end{array}
        \right)=\mathfrak{M},
\end{equation}
even more, the matrix sequence $(\mathfrak{a}_{2},\cdots,\mathfrak{a}_{N})$ is a cyclic solution of the equation $M(\mathfrak{a}_{1},\cdots,\mathfrak{a}_{N})=\mathfrak{M}$. But then, necessarily each block matrix $M(\mathfrak{a}_{k})$ must commute with
$\mathfrak{M}$ and this leads to the following facts $\mathfrak{m}_{11}=\mathfrak{m}_{22}=\mathfrak{m}$, $\mathfrak{m}_{12}=\mathfrak{m}_{21}=O$
and $[\mathfrak{m}, \mathfrak{a}_{k}]=0$ where $k=1,2,\cdots,N$.

Let $(\mathfrak{l}_{n})_{1\leq n}$ and $(\mathfrak{s}_{n})_{1\leq n}$ be two sequences of $l\times l$ matrices. We define a pair of left signed Chebychev polynomials $\left ((\mathfrak{p}_{n})_{1\leq n}, (\mathfrak{q}_{n})_{1\leq n} \right )$ of the following form
\begin{equation}\label{sf8}
\mathfrak{p}_{n}=\mathfrak{l}_{n}\mathfrak{p}_{n-1}+\mathfrak{s}_{n}\mathfrak{p}_{n-2},\qquad\qquad \mathfrak{p}_{0}=I, \mathfrak{p}_{-1}=O,
\end{equation}
and
\begin{equation}\label{sf9}
\mathfrak{q}_{n}=\mathfrak{l}_{n}\mathfrak{q}_{n-1}+\mathfrak{s}_{n}\mathfrak{q}_{n-2},\qquad\qquad \mathfrak{q}_{0}=O, \mathfrak{q}_{-1}=I,
\end{equation}
for instance
$\mathfrak{p}_{1}=\mathfrak{l}_{1}$, $\mathfrak{q}_{1}=\mathfrak{s}_{1}$, $\mathfrak{p}_{2}=\mathfrak{l}_{2}\mathfrak{l}_{1}+\mathfrak{s}_{2}$,
$\mathfrak{q}_{2}=\mathfrak{l}_{2}\mathfrak{s}_{1}$, etc.

We should call to $\left ((\mathfrak{p}_{n})_{1\leq n}, (\mathfrak{q}_{n})_{1\leq n} \right )$ a left Chebychev polynomial pair. Consider the recurrent equation
\begin{equation}\label{sf10}
\mathfrak{Y}_{k+1}=\mathfrak{l}_{k}\mathfrak{Y}_{k}+\mathfrak{s}_{k}\mathfrak{Y}_{k-1},
\end{equation}
where for all $1\leq k$, $\mathfrak{Y}_{k}=(\mathfrak{y}_{k}^{1}\,\,\, \mathfrak{y}_{k}^{2})\in M_{l\times 2l}(\mathbb{C})$.

We have
\begin{equation}\label{sf11}
\mathfrak{Y}_{k+1}=\mathfrak{p}_{k}\mathfrak{Y}_{1}+\mathfrak{q}_{k}\mathfrak{Y}_{0},
\end{equation}
indeed, it is clear that $\mathfrak{Y}_{2}=\mathfrak{l}_{1}\mathfrak{Y}_{1}+\mathfrak{s}_{1}\mathfrak{Y}_{0}=\mathfrak{p}_{1}\mathfrak{Y}_{1}+\mathfrak{q}_{1}\mathfrak{Y}_{0}$
and $\mathfrak{Y}_{3}=\mathfrak{l}_{2}\mathfrak{Y}_{2}+\mathfrak{s}_{2}\mathfrak{Y}_{1}=\mathfrak{l}_{2}(\mathfrak{p}_{1}\mathfrak{Y}_{1}+
\mathfrak{q}_{1}\mathfrak{Y}_{0}) +\mathfrak{s}_{2}\mathfrak{Y}_{1}=\mathfrak{l}_{2}(\mathfrak{l}_{1}\mathfrak{Y}_{1}+\mathfrak{s}_{1}\mathfrak{Y}_{0}) +\mathfrak{s}_{2}\mathfrak{Y}_{1}=\mathfrak{p}_{2}\mathfrak{Y}_{1}+\mathfrak{q}_{2}\mathfrak{Y}_{0}$, thus (\ref{sf11}) holds for $k=1$
and $k=2$. Let us suppose that (\ref{sf11}) is true for $k\leq r$ then combining (\ref{sf10})-(\ref{sf11}) conveniently, we obtain
$$\mathfrak{Y}_{r+1}=\mathfrak{l}_{r}\mathfrak{Y}_{r}+\mathfrak{s}_{r}\mathfrak{Y}_{r-1}=
\mathfrak{l}_{r}(\mathfrak{p}_{r-1}\mathfrak{Y}_{1}+\mathfrak{q}_{r-1}\mathfrak{Y}_{0})+\mathfrak{s}_{r}
(\mathfrak{p}_{r-2}\mathfrak{Y}_{1}+\mathfrak{q}_{r-2}\mathfrak{Y}_{0})=\mathfrak{p}_{r}\mathfrak{Y}_{1}+\mathfrak{q}_{r}\mathfrak{Y}_{0}.$$

Observe that for $1\leq n$
\begin{equation}\label{sf12}
\left(
  \begin{array}{cc}
    \mathfrak{p}_{n} & \mathfrak{q}_{n} \\
    \mathfrak{p}_{n-1} & \mathfrak{q}_{n-1} \\
  \end{array}
\right)=\left(
          \begin{array}{cc}
            \mathfrak{l}_{n} & \mathfrak{s}_{n} \\
            I & O \\
          \end{array}
        \right)\left(
  \begin{array}{cc}
    \mathfrak{p}_{n-1} & \mathfrak{q}_{n-1} \\
    \mathfrak{p}_{n-2} & \mathfrak{q}_{n-2} \\
  \end{array}
\right)
,\,\,\,\,\,\,\,\,\,\,\,\,\,\,\,\,\,\,\,\,\,\,\,\,\left(
  \begin{array}{cc}
    \mathfrak{p}_{0} & \mathfrak{q}_{0} \\
    \mathfrak{p}_{-1} & \mathfrak{q}_{-1} \\
  \end{array}
\right)=\left(
          \begin{array}{cc}
            I & O \\
            O & I \\
          \end{array}
        \right),
\end{equation}
and so from (\ref{sf12}) follows that
\begin{equation}\label{sf13}
M((\overline{\mathfrak{l}};\overline{\mathfrak{s}})_{n})=\left(
                                                           \begin{array}{cc}
                                                             \mathfrak{l}_{n} & \mathfrak{s}_{n} \\
                                                             I & O \\
                                                           \end{array}
                                                         \right)\cdots \left(
                                                           \begin{array}{cc}
                                                             \mathfrak{l}_{1} & \mathfrak{s}_{1} \\
                                                             I & O \\
                                                           \end{array}
                                                         \right)
=\left(
  \begin{array}{cc}
    \mathfrak{p}_{n} & \mathfrak{q}_{n} \\
    \mathfrak{p}_{n-1} & \mathfrak{q}_{n-1} \\
  \end{array}
\right).
\end{equation}

\section*{Acknowledgment}

The author was supported by CONAHCYT project $45886$.


\begin{thebibliography}{20}

\bibitem{baur} Baur K., Parsons M. J. and Tschabold M.  \emph {Infinite friezes.} European J. Combin. \textbf{54} (2016), 220-237.
\bibitem{baur1} Baur K. and Marsh R. J.  \emph {Frieze patterns for punctured discs.} J Algebr Comb \textbf{30} (2009), 349-379.
\bibitem{Chapoton} Chapoton F.  \emph {Operads and Algebraic Combinatorics of Trees.} S\'{e}m. Lothar. Combin. \textbf{58} (2008), Article B58c.
\bibitem{Clelland} Clelland J. N. \emph {From Frenet to Cartan: the method of moving frames.} Graduate Studies in Mathematics. v 178, 2017.
\bibitem{Conley} Conley C. H., Ovsienko V.  \emph {Rotundus: Triangulations, Chebyshev polynomials, and Pfaffians.}  Math. Intelligencer
\textbf{40} (2018), no. 3, 45-50.
\bibitem{Conway} Conway J.H., Coxeter, H.S.M. \emph {Triangulated polygons and frieze patterns.} Math. Gaz. \textbf{57} (1973), 87-94 and 175-183.
\bibitem{Coxeter} Coxeter H.S.M. \emph {Frieze patterns.} Acta Arith. \textbf{18} (1971), 297-310.
\bibitem{crawley} Crawley-Boevey w. \emph {On matrices in prescribed conjugacy classes with no common invariant subspace and sum zero.}
Duke Math. J. \textbf{118} (2003), 339-352.
\bibitem{cuntz} Cuntz M., Heckenberger I. \emph {Reflection groupoids of rank two and cluster algebras of type $A$.} J. Combin. Theory Ser. A
\textbf{118} (2011), 1350-1363.
\bibitem{cuntz1} Cuntz M. \emph {On wild fireze patterns.} Exp. Math. \textbf{26} no. 3 (2017), 342-348.
\bibitem{cuntz2} Cuntz M. \emph {On subsequences of quiddity cycles and Nichols algebras.} J. Algebra \textbf{502} (2018), 315-327.
\bibitem{felipe} Felipe R., Mari-Beffa G. \emph {The pentagram map on Grassmannians.} Ann. Inst. Fourier, Grenoble \textbf{61}
no.1 (2019), 421-456.
\bibitem{glick}  Glick M. \emph {The pentagram map and Y-patterns.} Adv. Math. \textbf{227} (2011), 1019-1045.
\bibitem{kedem} Kedem R., Vichitkunakorn P. \emph {T-systems and the pentagram map.} J. Geom. Phys. \textbf{87} (2015), 233-247.
\bibitem{Krich} Krichever I. \emph {Commuting difference operators and the combinatorial gale transform.} Funct. Anal. Appl. \textbf{49}
(2015), 175-188.
\bibitem{Ilina} Ilina A. V., Krichever I. \emph {Triangular reductions of the $2D$ Toda hierarchy.} Funct. Anal. Appl. \textbf{51} (2017), 48-65.
\bibitem{male} Male C. \emph {The distribution of traffics and their free product.} preprint. arXiv:math111.4662v4 [math. PR], 2013.
\bibitem{Mansfield} Mansfield E., Mari-Beffa G. and Wang J. P.  \emph {Discrete moving frames and discrete integrable systems.}
Found Comput Math \textbf{13} (2013), 545-582.
\bibitem{Maribeffa} Mari-Beffa G. and Wang J. P.  \emph {Hamiltonian evolutions of twisted polygons in $\mathbb{RP}^{n}$.} Nonlinearity
\textbf{26} (2013), 2515-2551.
\bibitem{Maribeffa1} Mari-Beffa G. \emph {On generalizations of the pentagram maps: Discretizacions of a AGD flows.} J. Nonlinear Sci.
\textbf{23} (2013), 303-334.
\bibitem{MorGen} Morier-Genoud S. \emph {Coxeter's frieze patterns at the crossroads of algebra, geometry and combinatorics.}
Bull. London Math. Soc. \textbf{47} (2015), 895-938.
\bibitem{MorGen1} Morier-Genoud S. and Ovsienko V. \emph {Farey boat: continued fractions and triangulations, modular group
and polygon dissections.} Jahresber. Dtsch. Math-Ver \textbf{121} (2019), 91-136.
\bibitem{neto-silva} Neto O. and Silva F. C. \emph {Singular regular differential equations and eigenvalues of products of matrices.}
Linear Multilinear Algebra \textbf{46} (1999), 145-164.x
\bibitem{oven} Ovenhouse N.  \emph {Non-commutative integrability of the Grassmann pentagram map.} Adv. Math. \textbf{373} (2020), 1-54.
\bibitem{ovsienko1} Ovsienko V. \emph {Partitions of unity in $SL(2,\mathbb{Z})$}, negative continued fractions, and dissections of
polygons. Res Math Sci (2018) 5:21.
\bibitem{ovsienko} Ovsienko V., Schwartz R., Tabachnikov S. \emph {The Pentagram Map: A Discrete Integrable System.} Commun. Math. Phys.
\textbf{299} (2010), 409-446.
\bibitem{schwarz-zaks} Schwarz B. and Zaks A. \emph{Geometry of matrix differential systems.} J. Math. Anal. Appl. \textbf{112} (1985), 165-177.
\bibitem{spi} Spivak D. \emph {The operad of wiring diagrams: formalizing a graphical language for
databases, recursion, and plug-and-play circuits.} arXiv:1305.0297 [cs.DB], 2013.
\bibitem{zabrodin} Zabrodin A., \emph {Discrete Hirota’s equation in quantum integrable models.} Internat. J. Modern
Phys. B, \textbf{11}:26–27 (1997), 3125–3158.
\bibitem{zhang} Zhang F. \emph{The Schur complement and its applications.} Springer, $2005$.

\end{thebibliography}
\end{document}